\RequirePackage{rotating}
\documentclass[a4paper,10pt]{amsart}
\usepackage[margin=1in]{geometry}
\usepackage{enumerate, amsmath, amsfonts, amssymb, amsthm, mathtools, thmtools, wasysym, graphics, graphicx, xcolor, frcursive,xparse,comment,ytableau,stmaryrd,bbm,array,colortbl,tensor}
\usepackage[all]{xy}
\usepackage[boxed,norelsize]{algorithm2e}
\usepackage{hyperref}
\usepackage{cancel}

\makeatletter
\newcommand{\specialcell}[1]{\ifmeasuring@#1\else\omit$\displaystyle#1$\ignorespaces\fi}
\makeatother

\usepackage{url, hypcap}
\hypersetup{colorlinks=true, citecolor=darkblue, linkcolor=darkblue}

\usepackage{tikz}
\usetikzlibrary{calc,through,backgrounds,shapes,matrix}

\definecolor{darkblue}{rgb}{0.0,0,0.7} 
\newcommand{\darkblue}{\color{darkblue}} 
\definecolor{darkred}{rgb}{0.7,0,0} 
\definecolor{lightgrey}{rgb}{0.7,0.7,0.7} 

\definecolor{meet}{RGB}{255,205,111}
\definecolor{join}{RGB}{0,77,178}

\newtheorem{theorem}{Theorem}[section]
\newtheorem{proposition}[theorem]{Proposition}
\newtheorem{corollary}[theorem]{Corollary}
\newtheorem{lemma}[theorem]{Lemma}

\theoremstyle{definition}
\newtheorem{definition}[theorem]{Definition}
\newtheorem{example}[theorem]{Example}

\newtheorem{conjecture}[theorem]{Conjecture}
\newtheorem{question}[theorem]{Question}

\newtheorem{remarkx}[theorem]{Remark}
\newenvironment{remark}
  {\pushQED{\qed}\remarkx}
  {\popQED\endremarkx}

\usepackage[procnames]{listings}
\usepackage[nameinlink]{cleveref}
\usepackage[all]{xy}
\usepackage[T1]{fontenc}
\crefformat{footnote}{#2\footnotemark[#1]#3}
\Crefname{conjecture}{Conjecture}{Conjectures}

\usepackage[cal=boondoxo]{mathalfa}
\usepackage[colorinlistoftodos]{todonotes}
\newcommand{\defn}[1]{\emph{\darkblue #1}}

\usetikzlibrary{math}
\usepackage{xifthen}
\usepackage{xstring}
\SetKwInput{kwset}{set}
\usetikzlibrary{arrows,backgrounds,calc,trees}
\pgfdeclarelayer{background}
\pgfsetlayers{background,main}

\newcommand{\row}{\mathsf{Row}}

\newcommand{\popnone}{\mathsf{Pop}}
\newcommand{\pop}{\mathsf{Pop}^\downarrow}
\newcommand{\popdown}{\mathsf{Pop}^\downarrow}
\newcommand{\popup}{\mathsf{Pop}^\uparrow}

\newcommand{\D}{\mathcal{D}}
\newcommand{\U}{\mathcal{U}}
\newcommand{\x}{v}

\newcommand{\J}{\mathcal{J}}
\newcommand{\I}{\mathrm{Ind}}
\newcommand{\M}{\mathcal{M}}
\newcommand{\In}{\mathrm{In}}
\newcommand{\Out}{\mathrm{Out}}
\newcommand{\face}{\mathrm{Face}}
\newcommand{\shard}{\mathrm{Shard}}
\newcommand{\popshard}{\mathrm{Shard}_{\mathsf{Pop}}}
\newcommand{\rowshard}{\mathrm{Shard}_\row}

\title{Semidistrim Lattices}

\author[C.~Defant]{Colin Defant}
\address[C.~Defant]{Princeton University}
\email{cdefant@princeton.edu}

\author[N.~Williams]{Nathan Williams}
\address[N.~Williams]{University of Texas at Dallas}
\email{nathan.williams1@utdallas.edu}

\begin{document}


\begin{abstract}
We introduce \emph{semidistrim lattices}, a simultaneous generalization of semidistributive and trim lattices that preserves many of their common properties.  We prove that the elements of a semidistrim lattice correspond to the independent sets in an associated graph called the \emph{Galois graph}, that products and intervals of semidistrim lattices are semidistrim, and that the order complex of a semidistrim lattice is either contractible or homotopy equivalent to a sphere.

Semidistrim lattices have a natural \emph{rowmotion} operator, which simultaneously generalizes Barnard's $\overline\kappa$ map on semidistributive lattices as well as Thomas and the second author's rowmotion on trim lattices.  Every lattice has an associated \emph{pop-stack sorting} operator that sends an element $x$ to the meet of the elements covered by $x$. For semidistrim lattices, we are able to derive several intimate connections between rowmotion and pop-stack sorting, one of which involves independent dominating sets of the Galois graph. 
\end{abstract}
\maketitle

\section{Introduction}
In this paper, all lattices are assumed to be finite.  Two families of lattices that extend the family of distributive lattices are the family of \emph{semidistributive} lattices and the family of \emph{trim} lattices (see~\Cref{sec:lattices}). The union of these two families contains several well-studied classes of lattices such as weak orders of finite Coxeter groups, facial weak orders of simplicial hyperplane arrangements \cite{Dermenjian}, finite Cambrian lattices \cite{readingCambrian}, biCambrian lattices \cite{BarnardBiCambrian}, $\nu$-Tamari lattices \cite{PrevilleViennot}, Grid-Tamari lattices \cite{McConvilleGrid}, and lattices of torsion classes of artin algebras \cite{Demonet, GarverMcConville, thomas2019rowmotion}.
Although these two families are distinct (a semidistributive lattice that is not trim and a trim lattice that is not semidistributive are illustrated on the left and in the middle of~\Cref{fig:semidistrim}), they share many common properties. For example, the following hold for each lattice $L$ in each of these families:
\begin{itemize}
    \item there is a canonical bijection between join-irreducible and meet-irreducible elements of $L$;
    \item cover relations of $L$ are canonically labeled by join-irreducible elements;
    \item each element in $L$ is uniquely determined by the labels of its down-covers and also by the labels of its up-covers;
    \item the collection of down-cover label sets of the elements of $L$ equals the collection of up-cover label sets of the elements of $L$, and each of these collections is equal to the collection of independent sets in a certain graph called the \emph{Galois graph};
    \item every interval of $L$ is also in the family;
   \item there is a natural way of defining a certain bijective operator called \emph{rowmotion} on $L$;
   \item $L$ is crosscut simplicial; in particular, its order complex is either contractible or homotopy equivalent to a sphere.
\end{itemize}

In this paper, we develop a theory of \emph{semidistrim lattices}, which we propose as a common generalization of semidistributive and trim lattices.  An example of a semidistrim lattice that is neither semidistributive nor trim is illustrated on the right of~\Cref{fig:semidistrim}.

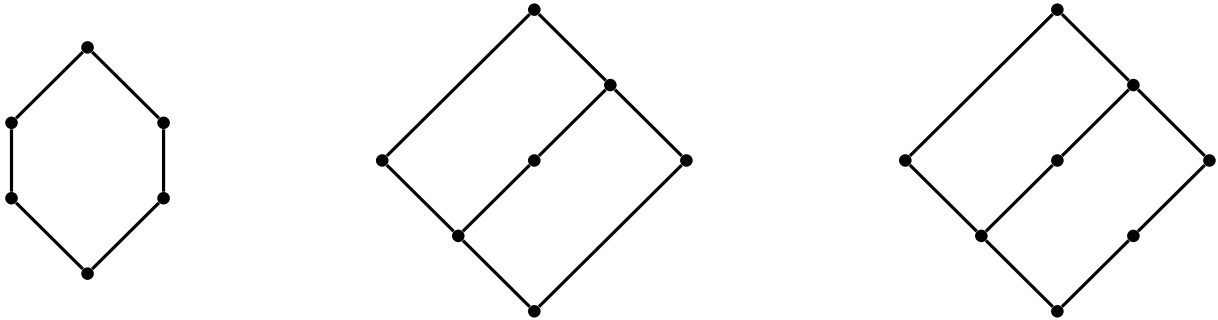
\begin{figure}[htbp]
\begin{center}
\raisebox{-.5\height}{\scalebox{1}{\begin{tikzpicture}[scale=1]
\node[shape=circle,fill=black, scale=0.5] (0) at (0,0) {};
\node[shape=circle,fill=black, scale=0.5] (1) at (-1,1) {};
\node[shape=circle,fill=black, scale=0.5] (2) at (1,1) {};
\node[shape=circle,fill=black, scale=0.5] (3) at (-1,2) {};
\node[shape=circle,fill=black, scale=0.5] (4) at (1,2) {};
\node[shape=circle,fill=black, scale=0.5] (5) at (0,3) {};
\draw[very thick] (0) to (1) to (3) to (5) to (4) to (2) to (0);
\end{tikzpicture}}}
\hfill
\raisebox{-.5\height}{\scalebox{1}{\begin{tikzpicture}[scale=1]
\node[shape=circle,fill=black, scale=0.5] (0) at (0,0) {};
\node[shape=circle,fill=black, scale=0.5] (1) at (-1,1) {};
\node[shape=circle,fill=black, scale=0.5] (2) at (2,2) {};
\node[shape=circle,fill=black, scale=0.5] (3) at (0,2) {};
\node[shape=circle,fill=black, scale=0.5] (4) at (1,3) {};
\node[shape=circle,fill=black, scale=0.5] (5) at (-2,2) {};
\node[shape=circle,fill=black, scale=0.5] (6) at (0,4) {};
\draw[very thick] (0) to (1) to (5) to (6) to (4) to (2) to (0);
\draw[very thick] (1) to (3) to (4);
\end{tikzpicture}}}
\hfill
\raisebox{-.5\height}{\scalebox{1}{\begin{tikzpicture}[scale=1]
\node[shape=circle,fill=black, scale=0.5] (0) at (0,0) {};
\node[shape=circle,fill=black, scale=0.5] (1) at (-1,1) {};
\node[shape=circle,fill=black, scale=0.5] (2) at (2,2) {};
\node[shape=circle,fill=black, scale=0.5] (3) at (0,2) {};
\node[shape=circle,fill=black, scale=0.5] (4) at (1,3) {};
\node[shape=circle,fill=black, scale=0.5] (5) at (-2,2) {};
\node[shape=circle,fill=black, scale=0.5] (6) at (0,4) {};
\node[shape=circle,fill=black, scale=0.5] (7) at (1,1) {};
\draw[very thick] (0) to (1) to (5) to (6) to (4) to (2) to (7) to (0);
\draw[very thick] (1) to (3) to (4);
\end{tikzpicture}}}
\end{center}
\caption{{\it Left}: A semidistributive lattice that is not trim.  {\it Middle}: A trim lattice that is not semidistributive.  {\it Right}: A semidistrim lattice that is neither trim nor semidistributive.}
\label{fig:semidistrim}
\end{figure}

\subsection{Semidistrim lattices}
In \Cref{sec:lattices}, we recall basic notions from lattice theory, including the definitions of semidistributive lattices and trim lattices.   If $L$ is an arbitrary lattice, then the set $\J_L$ of join-irreducible elements of $L$ might not have the same cardinality as the set $\M_L$ of meet-irreducible elements of $L$. However, these two sets do have the same size for several interesting lattices $L$. In \Cref{sec:paired}, we define a \emph{pairing} on $L$ to be a bijection $\J_L\to \M_L$ satisfying some additional natural desiderata. We say $L$ is \emph{uniquely paired} if it has a unique pairing; in this case, we let $\kappa_L$ denote its unique pairing. We also define the \emph{Galois graph} $G_L$ of a uniquely paired lattice $L$ to be the directed graph with vertex set $\J_L$ in which there is an edge $j\to j'$ whenever $j\neq j'$ and $j\not\leq\kappa_L(j')$. We prove that semidistributive lattices and trim lattices are uniquely paired. 

The purpose of \Cref{sec:rowmotable} is to define and establish basic properties about a new family of lattices that we call \emph{compatibly dismantlable}. Roughly speaking, a uniquely paired lattice is compatibly dismantlable if it can be broken into two disjoint intervals that are each compatibly dismantlable so that the join-irreducible elements, the meet-irreducible elements, and the unique pairing of the entire lattice are compatible with those of the two intervals in a precise sense. Our terminology is inspired by the notion of an interval-dismantlable lattice~\cite{adaricheva2018interval}. We also prove that semidistributive lattices and trim lattices are compatibly dismantlable, which is not obvious from the definitions.

In \Cref{sec:independent}, we define a uniquely paired lattice $L$ to be \emph{overlapping} if for each cover relation $x\lessdot y$, there is a unique join-irreducible element $j_{xy}$ such that $j_{xy}\leq y$ and $\kappa_L(j_{xy})\geq x$. If $L$ is overlapping and $x\in L$, then we define the \emph{downward label set} $\D_L(x) = \{j_{yx} : y \lessdot x\}$ and the \emph{upward label set} $\U_L(x)=\{j_{xy} : x \lessdot y\}$. These label sets are crucial for defining semidistrim lattices and their rowmotion operators. We also prove that every compatibly dismantlable lattice is overlapping. Moreover, we show that every element of a compatibly dismantlable lattice is uniquely determined by its downward label set and also by its upward label set.

\Cref{sec:semidistrim} begins with the central definition of the paper: that of a semidistrim lattice. A lattice $L$ is \emph{semidistrim} if it is compatibly dismantlable and if for every $x\in L$, the label sets $\D_L(x)$ and $\U_L(x)$ are independent sets in the Galois graph $G_L$ (this is an analogue of canonical join or meet representations for elements of semidistributive lattices).  We prove that this class of lattices contains all semidistributive lattices and all trim lattices, thereby justifying the name \emph{semidistrim}. Imposing the additional condition about independent sets on compatibly dismantlable lattices leads to several pleasant properties that we explore throughout the rest of the article. For instance, we prove that each of the maps $\D_L$ and $\U_L$ is actually a bijection from $L$ to the collection of independent sets of $G_L$. We also describe how the sets $\D_L(x)$ and $\U_L(x)$ fit together inside $G_L$ by proving that $(\D_L(x),\U_L(x))$ is a \emph{tight orthogonal pair}. 

In \Cref{sec:products_intervals}, we prove that the class of semidistrim lattices is closed under taking products and intervals. While it is relatively straightforward to show that products of semidistrim lattices are semidistrim, the corresponding fact about intervals requires quite a bit of work to prove. However, the payoff is well worth the effort: knowing that intervals of semidistrim lattices are semidistrim allows us to establish several further results via induction. One of the other pleasant corollaries of the proof of this result states that if $[u,v]$ is an interval in a semidistrim lattice $L$, then the unique pairing, the Galois graph, and the edge labels of $L$ are compatible in a precise way with those of $[u,v]$. Intervals of compatibly dismantlable lattices need not be compatibly dismantlable; thus, the additional condition about independent sets is essential.

\Cref{sec:crosscut} is brief and is devoted to proving that semidistrim lattices are crosscut simplicial. This implies, in particular, that the order complex of a semidistrim lattice is contractible or homotopy equivalent to a sphere.

\subsection{Rowmotion and pop-stack sorting}
One of the most well-studied operators in the field of \emph{dynamical algebraic combinatorics} is \emph{rowmotion}, a certain invertible operator on the set of order ideals of a finite poset~\cite{brouwer1974period,cameron1995orbits,striker2012promotion}. Equivalently (by Birkhoff's representation theorem), one can view rowmotion as a bijective operator on a distributive lattice. In recent years, there has been interest in extending the definition of rowmotion to more general classes of lattices. Barnard \cite{barnard19canonical} showed how to define rowmotion on semidistributive lattices, while Thomas and the second author \cite{thomas2019rowmotion} defined rowmotion on trim lattices. We refer to \cite{striker2018rowmotion,thomas2019rowmotion} for a more thorough historical account of rowmotion. In \Cref{sec:pop_and_row}, we define rowmotion for semidistrim lattices, thereby generalizing and unifying all previous definitions of rowmotion on lattices. 

Given a lattice $L$, we define the \defn{pop-stack sorting operator} $\popdown_L\colon L\to L$ and the \defn{dual pop-stack sorting} operator $\popup_L\colon L\to L$ by 
\[\popdown_L(x)=x\wedge\bigwedge\{y\in L:y\lessdot x\}\quad\text{and}\quad\popup_L(x)=x\vee\bigvee\{y\in L:x\lessdot y\}.\] When $L$ is the right weak order on the symmetric group $S_n$, the pop-stack sorting operator coincides with the \emph{pop-stack sorting map}, which acts by reversing the descending runs of a permutation. Recently, the pop-stack sorting map has received significant attention by  combinatorialists~\cite{Asinowski, Asinowski2, Elder, ClaessonPop, Pudwell}. The first author has previously studied pop-stack sorting operators on weak orders of arbitrary Coxeter groups in \cite{defant2021stack} and on $\nu$-Tamari lattices in \cite{defant21meeting}. M\"uhle \cite{MuhleCore} studied $\pop_L$ when $L$ is congruence-uniform, where he called $\pop_L(x)$ the \emph{nucleus} of $x$. The dual pop-stack sorting operator on the lattice of order ideals of a type $A$ root poset is equivalent to the \emph{filling} operator on Dyck paths analyzed in \cite{Sapounakis}. The authors have defined other variants of pop-stack sorting in \cite{DefantWilliamsTorsing, DefantWilliamsCrystal}. 

The main purpose of \Cref{sec:pop_and_row} is to show that rowmotion, pop-stack sorting, and dual pop-stack sorting on a semidistrim lattice $L$ are intimately connected. For example, we will show in \Cref{thm:pop_and_row} that if $x\in L$, then \[\pop_L(x)=x\wedge\row_L(x).\] In fact, we will prove that $\row_L(x)$ is a maximal element of the set $\{z\in L:\pop_L(x)=x\wedge z\}$. This naturally leads to a definition of rowmotion on meet-semidistributive lattices that need not be semidistrim; we prove that such a rowmotion operator is \emph{not} invertible whenever the lattice is not semidistributive. It remains completely open to investigate the basic properties of these noninvertible operators. 

Strengthening the links among rowmotion, pop-stack sorting, and dual pop-stack sorting on a semidistrim lattice $L$, we ask how many times rowmotion on $L$ ``goes down.'' More precisely, we are interested in the number of elements $x\in L$ such that $\row_L(x)\leq x$.
It turns out that this quantity is equal to the size of the image of $\pop_L$, which is also equal to the size of the image of $\popup_L$: \[|\{ x \in L : \row_L(x) \leq x\}|=|\pop_L(L)|=|\popup_L(L)|;\] moreover, each of these three quantities is equal to the number of \emph{independent dominating sets} of (the undirected version of) the Galois graph $G_L$. Even the equality $|\pop_L(L)|=|\popup_L(L)|$ is interesting and nontrivial here because this equality does \emph{not} hold for arbitrary lattices. This result motivates the study of the sizes of the images of the pop-stack sorting operators on interesting classes of lattices. This investigation was already initiated in \cite{Asinowski2, ClaessonPop2} for the weak order on $S_n$ (where $\pop_L$ is the classical pop-stack sorting map) and in \cite{Sapounakis} for the lattice of order ideals of a type $A$ root poset.  In \Cref{{sec:enumeration}}, we state several enumerative conjectures about the sizes of the images of pop-stack sorting operators on other specific lattices such as Tamari lattices, bipartite Cambrian lattices, and distributive lattices of order ideals in positive root posets.

\subsection{Further directions}
In \Cref{sec:shards}, we show that there is analogy between join-prime elements of semidistrim lattices and the basic hyperplanes in Reading's theory of shards.  It would be interesting if this analogy could be tightened, and we indicate a few possible directions to pursue in \Cref{sec:core_label}.      In \Cref{sec:further}, we collect several open questions about semidistrim lattices to help guide future research.

\section{Background}
\label{sec:lattices}

In this section, we review notions from lattice theory, and we briefly discuss the two families of lattices that we will later unify with a common generalization.

\subsection{Posets and lattices}
We assume basic familiarity with standard terminology from the theory of posets, as discussed in~\cite[Chapter 3]{stanley11enumerative_vol1}). For example, we write $x\lessdot y$ (or $y\gtrdot x$) to indicate that an element $y$ covers an element $x$. Given elements $x$ and $y$ in a poset $P$ with $x\leq y$, the \defn{interval} $[x,y]$ is defined to be the set $[x,y]=\{z\in P:x\leq z\leq y\}$. The \defn{dual} of the poset $P$ is the poset $P^*$ that has the same underlying set as $P$ but with all order relations reversed; that is, $x\leq y$ in $P$ if and only if $y\leq x$ in $P^*$. We write $\min(P)$ and $\max(P)$ for the set of minimal elements of $P$ and the set of maximal elements of $P$, respectively. The \defn{order complex} of $P$ is the abstract simplicial complex whose faces are the chains of $P$.

A \defn{lattice} is a poset $L$ such that any two elements $x,y\in L$ have a unique greatest lower bound, which is called their \defn{meet} and denoted $x\wedge y$, and a unique least upper bound, which is called their \defn{join} and denoted $x\vee y$. The meet and join operations are associative and commutative, so it makes sense to consider the meet and join of an arbitrary subset $X\subseteq L$; we denote these by $\bigwedge X$ and $\bigvee X$, respectively. Each lattice has a unique minimal element, which we denote by $\hat 0$, and a unique maximal element, which we denote by $\hat 1$. An \defn{atom} of $L$ is an element that covers $\hat 0$, and a \defn{coatom} of $L$ is an element that is covered by $\hat 1$. 

\subsection{Irreducibles and primes}

Let $L$ be a lattice. An element $j\in L$ is called \defn{join-irreducible} if it covers exactly one element; if this is the case, we denote by $j_*$ the unique element covered by $j$. An element $m\in L$ is called \defn{meet-irreducible} if it is covered by exactly one element; if this is the case, we denote by $m^*$ the unique element covering $m$. We write $\J_L$ and $\M_L$ for the set of join-irreducible elements of $L$ and the set of meet-irreducible elements of $L$, respectively. An element $j \in L$ is called \defn{join-prime} if for all $x,y \in L$ satisfying $x\vee y\geq j$, we have $x \geq j$ or $y \geq j$.  Similarly, an element $m \in L$ is called \defn{meet-prime} if for all $x,y \in L$ satisfying $x\wedge y\leq m$, we have $x \leq m$ or $y \leq m$.  Join-prime elements are necessarily join-irreducible, and meet-prime elements are necessarily meet-irreducible.

\begin{proposition}[{\cite[Theorem 6]{markowsky92primes}}]
\label{prop:join_decomp}

Let $L$ be a lattice. An element $j_0\in L$ is join-prime if and only if there exists $m_0\in L$ such that $L=[\hat 0,m_0]\sqcup [j_0,\hat 1]$. An element $m_0\in L$ is meet-prime if and only if there exists $j_0\in L$ such that $L=[\hat 0,m_0]\sqcup [j_0,\hat 1]$.
\end{proposition}
\begin{proof}
We prove only the first statement since a completely analogous dual argument handles the second. Assume first that $L = [\hat{0},m_0] \sqcup [j_0,\hat{1}]$.  Suppose $x,y\in L$ satisfy $x\vee y\geq j_0$. If we had $x\not\geq j_0$ and $y\not\geq j_0$, then we would have $x\leq m_0$ and $y\leq m_0$. However, this would force $x\vee y\leq m_0$, contradicting the fact that $x\vee y\geq j_0$. This shows that we must have $x\geq j_0$ or $y\geq j_0$. Hence, $j_0$ is join-prime. 

To prove the converse, assume $j_0$ is join-prime. 
We claim that $L \setminus [j_0, \hat{1}]$ has a unique maximal element $m_0$.  To see this, suppose instead that this set has two distinct maximal elements, say $m_1$ and $m_2$.  Then $m_1 \vee m_2 \in [j_0, \hat{1}]$, so the assumption that $j_0$ is join-prime forces us to have $m_1 \geq j_0$ or $m_2 \geq j_0$, contradicting the fact that $m_1,m_2 \in L \setminus [j_0, \hat{1}]$.  This proves the claim, which implies that $L=[\hat{0},m_0] \sqcup [j_0,\hat{1}]$.  
\end{proof}

If we can write $L=[\hat 0,m_0]\sqcup[j_0,\hat 1]$ (so that $j_0$ is join-prime and $m_0$ is meet-prime), then we call the pair $(j_0,m_0)$ a \defn{prime pair} for $L$.  

\subsection{Semidistributive lattices}
\label{sec:semidistributive}

A lattice $L$ is \defn{join-semidistributive} if for all $x,y,z \in L$ satisfying $x \vee y = x\vee z$, we have $x \vee (y \wedge z) = x \vee y$.  Equivalently, $L$ is join-semidistributive if for all $a,b\in L$ with $a\leq b$, the set $\{w\in L:w\vee a=b\}$ has a unique minimal element.  Join-semidistributive lattices are characterized among finite lattices as those lattices having a certain \defn{canonical join representation} for their elements~\cite[Theorem~2.24]{freese95free},\cite[Theorem~3.1]{barnard19canonical}; the canonical join representations $x=\bigvee A$ are \defn{irredundant} in the sense that $\bigvee A' < \bigvee A$ for every proper subset $A' \subset A$, and the elements of $A$ are taken to be as small as possible in the partial order.

Dually, $L$ is called \defn{meet-semidistributive} if for all $x,y,z \in L$ satisfying $x \wedge y = x\wedge z$, we have $x \wedge (y \vee z) = x \wedge y$. The lattice $L$ is meet-semidistributive if and only if for all $a,b\in L$ with $a\leq b$, the set $\{w\in L:w\wedge b=a\}$ has a unique maximal element.  Meet-semidistributive lattices are characterized among finite lattices as having \defn{canonical meet representations} $x=\bigwedge A$---these representations are irredundant, and the elements of $A$ are taken to be as large as possible.

A lattice is \defn{semidistributive} if it is both join-semidistributive and meet-semidistributive.  It is known that in a semidistributive lattice, every atom is join-prime and every coatom is meet-prime~\cite{gaskill1981join}.

\subsection{Trim lattices} 
We say a lattice $L$ is \defn{extremal} if it has a maximum-length chain $\hat 0=x_0\lessdot x_1\lessdot x_2\lessdot \cdots \lessdot x_n=\hat 1$ such that $|\J_L|=|\M_L|=n$~\cite{markowsky92primes}.   An element $x\in L$ is called \defn{left modular} if for all $y,z\in L$ with $y \leq z$, we have the equality
$(y \vee x) \wedge z = y \vee (x \wedge z)$.  A lattice is called \defn{left modular} if it has a maximal chain of left modular elements~\cite{liu1999left}.  A lattice is called \defn{trim} if it is extremal and left modular~\cite{thomas2006analogue,thomas2019rowmotion}.  Every trim lattice has at least one atom that is join-prime. 

\medskip

\Cref{fig:semidistrim} shows a semidistributive lattice that is not trim and a trim lattice that is not semidistributive.  It was shown in~\cite[Theorem 1.4]{thomas2019rowmotion} that an extremal semidistributive lattice is necessarily trim.

\section{Uniquely Paired Lattices and Galois Graphs}
\label{sec:paired}

As before, we write $\J_L$ and $\M_L$ for the set of join-irreducible elements and the set of meet-irreducible elements, respectively, of a lattice $L$.  

\subsection{Uniquely paired lattices}
\label{sec:uniquely_paired_lattices}
For $j\in \J_L$ and $m\in \M_L$, let us write $\M_L(j)=\max\{z\in L : j_*=j \wedge z\}$ and $\J_L(m)=\min\{z\in L : m^*=m \vee z\}$. Note that $\M_L(j)$ is nonempty because it contains $j_*$ and that $\J_L(m)$ is nonempty because it contains $m^*$. One can show that $L$ is meet-semidistributive if and only if $\M_L(j)$ is a singleton set for every $j\in \J_L$ (see \cite[Theorem~2.56]{freese95free}). Similarly, $L$ is join-semidistributive if and only if $\J_L(m)$ is a singleton set for every $m\in\M_L$. 

\begin{lemma}\label{lem:contained_in_J} For $j\in \J_L$ and $m\in \M_L$, we have $\J_L(m)\subseteq \J_L$ and $\M_L(j)\subseteq\M_L$.
\end{lemma}
\begin{proof} Suppose $y \in \J_L(m)$ is not join-irreducible. Then $y=\bigvee X$ for some set $X\subseteq L$ with $y \not \in X$.  Consider some $x\in X$. We have $y > x$, so $m^*=y \vee m \geq x\vee m \geq m$. This implies that $x\vee m$ is either $m$ or $m^*$. But we know that $x\vee m\neq m^*$ because $y$ is a minimal element of $\{z\in L: m\vee z = m^*\}$.  Hence, $x\leq m$. As $x$ was arbitrary, it follows that $y=\bigvee X\leq m$. However, this implies that $y\vee m=m$, which contradicts the fact that $y\vee m=m^*$. This proves that $\J_L(m)\subseteq \J_L$. A dual argument proves that $\M_L(j)\subseteq \M_L$. 
\end{proof}

\begin{definition}
A \defn{pairing} on a lattice $L$ is a bijection $\kappa: \J_L \to \M_L$ such that $\kappa(j) \in \M_L(j)$ for every $j\in \J_L$ and $\kappa^{-1}(m) \in \J_L(m)$ for every $m\in\M_L$.  We say $L$ is \defn{paired} if it has a pairing, and we say $L$ is \defn{uniquely paired} if it has a unique pairing.
\label{def:uniquely_paired}
\end{definition}

When $L$ is uniquely paired, we will use the symbol $\kappa_L$ for its unique pairing. \Cref{fig:nc3} gives an example of a paired lattice that is not uniquely paired.

\begin{figure}[htbp]
\begin{center}
\raisebox{-.5\height}{\scalebox{1}{\begin{tikzpicture}[scale=1]
\node[shape=circle,fill=black, scale=0.5] (0) at (0,0) {};
\node[shape=circle,fill=black, scale=0.5] (1) at (-1,1) {};
\node[shape=circle,fill=black, scale=0.5] (2) at (0,1) {};
\node[shape=circle,fill=black, scale=0.5] (3) at (1,1) {};
\node[shape=circle,fill=black, scale=0.5] (4) at (0,2) {};
\draw[very thick] (0) to (1) to (4) to (2) to (0) to (3) to (4);
\end{tikzpicture}}}
\end{center}
\caption{A paired lattice that is not uniquely paired.}
\label{fig:nc3}
\end{figure}
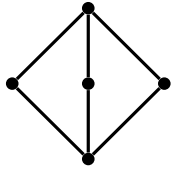

If $\kappa$ is a pairing on $L$ and $j\in\J_L$, then $j \not \leq \kappa(j)$. Indeed, this is immediate from the fact that $\kappa(j)\wedge j =j_*\neq j$. In fact, the elements $m=\kappa(j)$ and $j$ can only be comparable if $m = j_*$ and $j = m^*$.

\begin{lemma}
Let $L$ be a lattice. A bijection $\kappa\colon\J_L\to\M_L$ is a pairing if and only if the following hold for all $j\in \J_L$ and $m\in\M_L$ with $m=\kappa(j)$:
\begin{itemize}
\item $m\geq j_*$;
\item $m^*\geq j$;
\item $m\not\geq j$.
\end{itemize}
\label{lem:paired_check}
\end{lemma}
\begin{proof}
Let $\kappa\colon\J_L\to \M_L$ be a bijection. Let $j\in\J_L$, and write $m=\kappa(j)$.  Suppose $m\geq j_*$, $m^*\geq j$, and $m\not\geq j$.  Since $j \geq j \wedge m \geq j_*$ and $m \not \geq j$, we must have $j \wedge m=j_*$.  The element $m$ is maximal among elements $z$ with $j \wedge z = j_*$ because we have $j\wedge x\geq j \wedge m^* = j$ for all $x>m$.  Therefore $m \in \M_L(j)$.  An analogous argument shows that $j \in \J_L(m)$. Hence, $\kappa$ is a pairing.

Conversely, suppose that $\kappa$ is a pairing.  Fix $j\in\J_L$, and let $m=\kappa(j)$. Since $m \in \M_L(j)$, we have $j \wedge m = j_*$.  This implies that $m \geq j_*$ and $m \not \geq j$. We also have $m^*\wedge j\geq m\wedge j=j_*$. The maximality of $m$ coming from the definition of $\M_L(j)$ guarantees that $m^*\wedge j\neq j_*$, so $m^*\wedge j\geq j$. Hence, $m^* \geq j$.
\end{proof}

\begin{proposition}
\label{prop:pairing}
If $\kappa$ is a pairing on a lattice $L$ and $(j_0,m_0)$ is a prime pair for $L$, then $\kappa(j_0)=m_0$.
\end{proposition}

\begin{proof}
Since $m_0^* \not \in [\hat{0},m_0]$, we have $m_0^* \in [j_0,\hat{1}]$, so $m_0^* \geq j_0$.  It follows that $m_0 \vee j_0 = m_0^*$. Since $m_0\vee (j_0)_*=m_0$, we find that $j_0 \in \J_L(m_0)$.  Since $L=[\hat{0},m_0]\sqcup [j_0,\hat{1}]$, any $z$ satisfying $m_0^*=m_0\vee z $ must satisfy $z \geq j_0$.  Therefore, $\J_L(m) = \{j_0\}$.  By the definition of a pairing, we must have $\kappa(j_0)=m_0$.   
\end{proof}

\subsection{Galois graphs}
In this subsection, we specialize Markowsky's \emph{poset of irreducibles} for general lattices to uniquely paired lattices. 

\label{sec:galois}
\begin{definition} Let $L$ be uniquely paired with unique pairing $\kappa_L\colon\J_L\to\M_L$. 
The \defn{Galois graph} of $L$, denoted $G_L$, is the directed graph with vertex set $\J_L$ such that for all $j,j'\in \J_L$, there is an edge $j \to j'$ if and only if $j \not \leq \kappa_L(j')$ and $j\neq j'$.
\end{definition}

Slightly abusing notation, we write $j \to j'$ to mean that there is a directed edge from $j$ to $j'$ in $G_L$, and we define 
\begin{align*}\Out(j) &= \{ j' \in \J_L : j \to j'\} \quad \text{and} \\ \In(j) &= \{ j' \in \J_L : j' \to j \}.\end{align*} 

\begin{proposition}[{\cite[Remark 2]{markowsky92primes}}]
\label{prop:no_prob}
Let $L$ be a uniquely paired lattice with Galois graph $G_L$.  A join-irreducible element $j$ of $L$ is join-prime if and only if $j'\to j''$ whenever we have $j' \to j$ and $j \to j''$.  
\end{proposition}

\begin{proposition}
\label{prop:no_double_edges}
Let $L$ be a uniquely paired lattice with Galois graph $G_L$.  If $j\in L$ is join-prime, then there is no join-irreducible element $j'$ with $j \to j'$ and $j' \to j$.
\end{proposition}

\begin{proof}
Suppose that $j$ is join-prime and that $j'\in\J_L$ satisfies $j\to j'$. Write $m=\kappa_L(j)$ and $m'=\kappa_L(j')$.  Then $j\not\leq m'$. \Cref{prop:pairing} tells us that $(j,m)$ is a prime pair, so $L=[\hat 0,m]\sqcup[j,\hat 1]$.  Consequently, $m'< m$.  Because $m'$ is a maximal element of $\{z\in L: j'_*=j'\wedge z\}$, we cannot have $j'\wedge m=j'_*$. We know that $j' \wedge m \geq j' \wedge m' = j'_*$, so we must have $j' \wedge m > j'_*$. Hence, $j'\wedge m=j'$. This means that $j'\leq m$, so there is no edge $j'\to j$ in $G_L$.
\end{proof}

\subsection{Semidistributive and extremal lattices are uniquely paired}

\begin{proposition}\label{prop:semidistributive_uniquely_paired}
Semidistributive lattices are uniquely paired.
\end{proposition}

\begin{proof}
Let $L$ be a semidistributive lattice. Then $\M_L(j)$ and $\J_L(m)$ are singleton sets for all $j\in\J_L$ and $m\in \M_L$. Furthermore, one readily checks that $m$ is the unique element of $\M_L(j)$ if and only if $j$ is the unique element of $\J_L(m)$. Therefore, we obtain a pairing $\kappa\colon\J_L\to M_L$ by declaring $\kappa(j)$ to be the unique element of $\M_L(j)$. This is the only pairing on $L$.  
\end{proof}

Even though $\J_L(m)$ and $\M_L(j)$ can contain multiple elements in extremal lattices, such lattices are still uniquely paired.  An example is given in~\Cref{fig:extremal_paired}.

\begin{figure}[htbp]
\begin{center}
\scalebox{1}{\begin{tikzpicture}[scale=3]
\node[shape=circle,fill=black,scale=0.5,label={[xshift=0ex, yshift=-4.5ex]$\hat{0}$}] (a) at (0,0) {};
\node[shape=circle,fill=black,scale=0.5,label={[xshift=-2.5ex, yshift=-2ex]$j_1$}] (1) at (-1,1) {};
\node[shape=circle,fill=black,scale=0.5,label={[xshift=2ex, yshift=-2ex]$j_2$}]  (2) at (0,1) {};
\node[shape=circle,fill=black,scale=0.5,label={[xshift=2.5ex, yshift=-2ex]$j_3$}] (3) at (1,1) {};
\node[shape=circle,fill=black,scale=0.5,label={[xshift=-2.5ex, yshift=-2ex]$m_4$}] (e) at (-1,2) {};
\node[shape=circle,fill=black,scale=0.5,label={[xshift=-2.5ex, yshift=-3ex]$j_4$},label={[xshift=-2.5ex, yshift=-1ex]$m_3$}]  (4) at (0,1.75) {};
\node[shape=circle,fill=black,scale=0.5,label={[xshift=2.5ex, yshift=-2ex]$m_2$}]  (g) at (1,2) {};
\node[shape=circle,fill=black,scale=0.5,label={[xshift=-2.5ex, yshift=-2ex]$m_1$}] (h) at (0,2.5) {};
\node[shape=circle,fill=black,scale=0.5,label={[xshift=0ex, yshift=0ex]$\hat{1}$}] (i) at (0,3) {};
\draw[very thick] (a) to (1);
\draw[very thick] (a) to (2);
\draw[very thick] (a) to (3);
\draw[very thick] (1) to  (e);
\draw[very thick] (2) to  (e);
\draw[very thick] (2) to (4);
\draw[very thick] (4) to (h);
\draw[very thick] (h) to (i);
\draw[very thick] (1) to  (g);
\draw[very thick] (3) to (g);
\draw[very thick] (3) to (h);
\draw[very thick] (e) to (i);
\draw[very thick] (g) to (i);
\end{tikzpicture}}
\end{center}
\caption{An extremal (but not trim) lattice; by~\Cref{prop:extremal_paired}, this lattice is uniquely paired.  Even though both $m_3$ and $m_4$ are maximal elements of the set $\{z : (j_3)_*=j_3 \wedge z\}$, the element $m_4$ must be paired with $j_4$ because $m_4$ is the only element of $\{z : (j_4)_*=j_4 \wedge z\}$; this then forces $m_3$ to be paired with $j_3$.}
\label{fig:extremal_paired}
\end{figure}
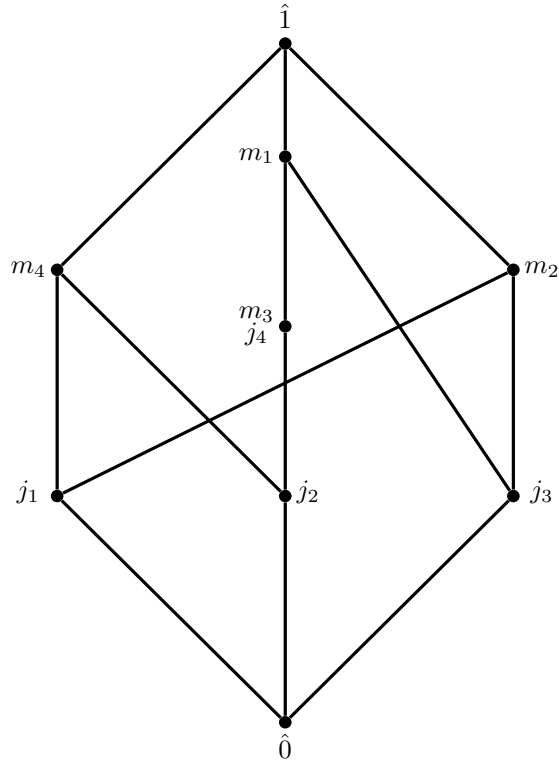

\begin{proposition}
\label{prop:extremal_paired}
Extremal lattices are uniquely paired.
\end{proposition}
\begin{proof} 
Let $L$ be an extremal lattice, and fix a maximum-length chain $\hat{0}=x_0\lessdot x_1\lessdot x_2\lessdot \cdots \lessdot x_n=\hat{1}$. For each $i\in[n]$, there is a unique $j_i\in\J_L$ such that $j_i \vee x_{i-1}=x_i$, and there is a unique $m_i\in\M_L$ such that $m_i \wedge x_i = x_{i-1}$. This gives rise to a bijection $\kappa\colon\J_L\to\M_L$ defined by $\kappa(j_i)=m_i$. 
Let us check that $\kappa$ is a pairing. Fix $i\in[n]$. First, suppose $j$ is a join-irreducible with $j\leq (j_i)_*$. Then $j\vee x_{i-1}\leq j_i\vee x_{i-1}=x_i$. By the uniqueness of $j_i$, we know that $j\vee x_{i-1}\neq x_i$. This shows that $j\vee x_{i-1}=x_{i-1}$, so $j\leq x_{i-1}$. As this is true for every join-irreducible $j$ satisfying $j\leq (j_i)_*$, we conclude that $(j_i)_*\leq x_{i-1}$. Because $j_i\vee x_{i-1}=x_i$, we have $j_i\not\leq x_{i-1}$, so \[j_i\neq x_{i-1} \wedge j_i = (m_i \wedge x_i) \wedge j_i = m_i \wedge (x_i \wedge j_i) = m_i \wedge j_i.\] Consequently, $m_i\wedge j_i<j_i$. But since $(j_i)_* \leq x_{i-1}$, we have $m_i \wedge j_i = x_{i-1}\wedge j_i= (j_i)_*$.  This shows that $m_i\geq (j_i)_*$ and that $m_i\not\geq j_i$. A similar argument shows that $m_i^*\geq j_i$. As this is true for all $i$, \Cref{lem:paired_check} tells us that $\kappa$ is a pairing.

We now show that $\kappa$ is the only pairing on $L$. Suppose instead that there is some pairing $\kappa'\colon\J_L\to\M_L$ with $\kappa'\neq \kappa$. There must be some $i,k\in[n]$ with $i<k$ such that $\kappa'(j_i)=m_k$. Then $x_k > x_{k-1} \geq x_i \geq j_i$, so $m_k \wedge j_i = m_k \wedge (x_k \wedge j_i) = (m_k\wedge x_k)\wedge j_i=x_{k-1} \wedge j_i = j_i$. This contradicts the fact that $\kappa'(j_i)\wedge j_i=(j_i)_*$ by the definition of a pairing. 
\end{proof}

Since trim lattices are extremal by definition, we have the following corollary. 

\begin{corollary}
Trim lattices are uniquely paired. 
\end{corollary}

\section{Compatibly Dismantlable Lattices}
\label{sec:rowmotable}

As a next step toward semidistrim lattices, we impose additional structure on uniquely paired lattices to obtain the family of \emph{compatibly dismantlable} lattices. This new structure is an analogue of interval-dismantlability (see \cite{adaricheva2018interval}) that additionally requires a certain compatibility condition for join-irreducible elements and for meet-irreducible elements. 

\subsection{Compatibly dismantlable lattices}
\begin{definition}
\label{def:semidistrim}
A uniquely paired lattice $L$ is \defn{compatibly dismantlable} if it has cardinality $1$ or if it contains a prime pair $(j_0,m_0)$ such that the following compatibility conditions hold:
\begin{itemize}
    \item $[j_0,\hat{1}]$ is compatibly dismantlable, and there is a bijection
    \[\alpha\colon\{j\in \J_L : j_0 \leq \kappa_L(j)\} \to \J_{[j_0,1]}\]
     given by $\alpha(j)=j_0\vee j$ such that $\kappa_{[j_0,\hat 1]}(\alpha(j))=\kappa_L(j)$ for all $j\in \J_L$ with $j_0\leq \kappa_L(j)$;
    \item $[\hat{0},m_0]$ is compatibly dismantlable, and there is a bijection \[\beta\colon\{m\in \M_L : \kappa_L^{-1}(m) \leq m_0\} \to \M_{[0,m_0]}\] given by $\beta(m)=m_0\wedge m$ such that $\kappa_{[\hat 0,m_0]}^{-1}(\beta(m))=\kappa_L^{-1}(m)$ for all $m\in\M_L$ with $\kappa_L^{-1}(m)\leq m_0$.
\end{itemize}
We call such a prime pair $(j_0,m_0)$ a \defn{dismantling pair} for $L$ and write $L_0 = [\hat{0},m_0]$ and $L^0 = [j_0,\hat{1}]$.
\end{definition}

\begin{figure}[htbp]
\begin{center}
\raisebox{-.5\height}{\scalebox{1}{\begin{tikzpicture}[scale=2]
\node[shape=circle,fill=black,scale=0.5,label={[xshift=0ex, yshift=-4.5ex]$\hat{0}$}] (0) at (0,0) {};
\node[shape=circle,fill=black,scale=0.5,label={[xshift=-2ex, yshift=-2ex]$j_0$}] (j2) at (-1,1) {};
\node[shape=circle,fill=black,scale=0.5,label={[xshift=-2ex, yshift=-2ex]$j_1$}] (j1) at (0,1) {};
\node[shape=circle,fill=black,scale=0.5,label={[xshift=2.5ex, yshift=-1ex]$m_1$},label={[xshift=2.5ex, yshift=-3ex]$j_2$}] (j0) at (1,1) {};
\node[shape=circle,fill=black,scale=0.5,label={[xshift=-2ex, yshift=-3ex]$j_3$},label={[xshift=-2ex, yshift=-1ex]$m_4$}] (j3) at (-1,2) {};
\node[shape=circle,fill=black,scale=0.5,label={[xshift=2.5ex, yshift=-3ex]$j_4$},label={[xshift=2.5ex, yshift=-1ex]$m_3$}] (j4) at (-.5,2) {};
\node[shape=circle,fill=black,scale=0.5,label={[xshift=2.5ex, yshift=-2ex]$m_0$}] (m2) at (1,2) {};
\node[shape=circle,fill=black,scale=0.5,label={[xshift=-2.5ex, yshift=-2ex]$m_2$}] (m0) at (0,3) {};
\node[shape=circle,fill=black,scale=0.5,label={[xshift=0ex, yshift=0ex]$\hat{1}$}] (1) at (1,4) {};
\draw[very thick] (0) to (j2) to (j3) to (m0) to (1) to (m2) to (j0) to (0);
\draw[very thick] (0) to (j1) to (m0) to (j4) to (j2);
\draw[very thick] (j1) to (m2);
\end{tikzpicture}}}
\hspace{1in}
\raisebox{-.5\height}{\scalebox{1}{\begin{tikzpicture}[scale=2]
\node[shape=circle,draw=black,very thick] (2) at (0,0) {$j_0$};
\node[shape=circle,draw=black,very thick] (0) at (0,-1) {$j_2$};
\node[shape=circle,draw=black,very thick] (1) at (0,1) {$j_1$};
\node[shape=circle,draw=black,very thick] (3) at (1,0) {$j_3$};
\node[shape=circle,draw=black,very thick] (4) at (-1,0) {$j_4$};
\draw[->,very thick] (0) to (4);
\draw[->,very thick] (0) to (3);
\draw[->,very thick] (4) to (2);
\draw[->,very thick] (3) to (2);
\draw[->,very thick] (2) to (1);
\draw[<->,very thick] (4) to (1);
\draw[<->,very thick] (3) to (1);
\end{tikzpicture}}}
\end{center}

\caption{{\it Left:} A compatibly dismantlable lattice.  {\it Right:} The corresponding Galois graph.}
\label{fig:reduced}
\end{figure}
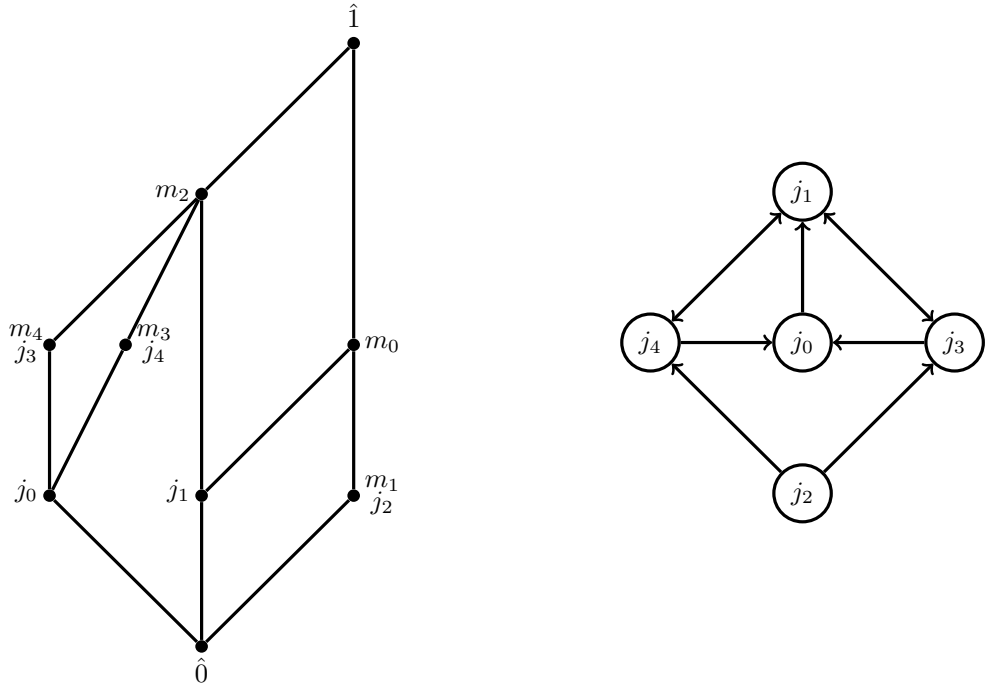
\begin{example}

\Cref{fig:reduced} illustrates a compatibly dismantlable lattice $L$ with dismantling pair $(j_0,m_0)$. The join-irreducible and meet-irreducible elements are named in such a way that $\kappa_L(j_i)=m_i$ for all $0\leq i\leq 4$. The join-irreducible elements $j\in \J_L$ satisfying $j_0\leq \kappa_L(j)$ are $j_2,j_3,j_4$. These elements correspond bijectively to the join-irreducible elements of $L^0$: we have $\alpha(j_2)=j_0\vee j_2=\hat 1$, $\alpha(j_3)=j_0\vee j_3=j_3$, and $\alpha(j_4)=j_0\vee j_4=j_4$. The meet-irreducible elements $m\in \M_L$ satisfying $\kappa_L^{-1}(m)\leq m_0$ are $m_1$ and $m_2$. These elements correspond bijectively to the meet-irreducible elements of $L_0$: we have $\beta(m_1)=m_0\wedge m_1=m_1$ and $\beta(m_2)=m_0\wedge m_2=j_1$. Notice how $\alpha$ and $\beta$ are compatible with the pairings $\kappa_L$, $\kappa_{[\hat 0, m_0]}$, and $\kappa_{[j_0,\hat 1]}$, as required by \Cref{def:semidistrim}. For example, $\kappa_{[j_0,\hat 1]}(\alpha(j_2))=\kappa_{[j_0,\hat 1]}(\hat 1)=m_2=\kappa_L(j_2)$. 
\end{example}

\begin{remark}
The lattice in \Cref{fig:reduced} has the undesirable property that the cover relations on its elements do not yield irredundant (as in~\Cref{sec:semidistributive}) join or meet representations: for example, $\hat{0}$ is covered by $j_0$, $j_1$, and $j_2$ so that $\hat{0} = m_0 \wedge m_1 \wedge m_2$---but it also has the meet representation $\hat{0} = m_1 \wedge m_2$.  

In~\Cref{sec:semidistrim}, we will define semidistrim lattices by imposing a mild additional condition on compatibly dismantlable lattices to eliminate such occurrences.  This extra condition is shared by both semidistributive and trim lattices, and it will have drastic consequences concerning intervals and compatibility conditions. It will turn out that every prime pair in a semidistrim lattice is a dismantling pair. More generally, we will find that every interval $[u,v]$ in a semidistrim lattice $L$ is semidistrim and that there is a compatibility condition for its join-irreducible and meet-irreducible elements generalizing the condition from \Cref{def:semidistrim} for the intervals $[j_0,\hat 1]$ or $[\hat 0,m_0]$. 
\end{remark}

Suppose $L$ is a compatibly dismantlable lattice with dismantling pair $(j_0,m_0)$. Let us write $L_0=[\hat 0,m_0]$ and $L^0=[j_0,\hat 1]$. The first compatibility condition in \Cref{def:semidistrim} allows us to use the bijection $\alpha$ to identify $\J_{L^0}$, which is the vertex set of the Galois graph $G_{L^0}$, with the subset $\{j\in\J_L:j_0\leq\kappa_L(j)\}$ of the vertex set of $G_L$. Furthermore, we can naturally view $\J_{L_0}$, the vertex set of $G_{L_0}$, as a subset of the vertex set of $G_L$ (since every join-irreducible element of $L_0$ is also join-irreducible in $L$).
The next proposition tells us that $G_{L^0}$ and $G_{L_0}$ can be identified with genuine induced subgraphs of $G_L$.

\begin{proposition}
\label{prop:galois}
Let $L$ be a compatibly dismantlable lattice with a dismantling pair $(j_0,m_0)$, and let $L_0=[\hat 0,m_0]$ and $L^0=[j_0,\hat 1]$.   
\begin{itemize}
\item Under the bijection $\alpha\colon\{j\in \J_L : j_0 \leq \kappa_L(j)\}\to\J_{L^0}$ from \Cref{def:semidistrim}, the Galois graph $G_{L^0}$ corresponds to the induced subgraph of $G_L$ on the vertex set $\J_L\setminus(\{j_0\}\cup\Out(j_0))$.
\item The Galois graph $G_{L_0}$ is the induced subgraph of $G_L$ on the vertex set $\J_L\setminus(\{j_0\}\cup\In(j_0))$. 
\end{itemize}
\end{proposition}

\begin{proof}
It is immediate from the definition of the edges in the Galois graph that $\{j\in\J_L:j_0\leq\kappa_L(j)\}=\J_L\setminus(\{j_0\}\cup\Out(j_0))$. Now suppose $j',j''\in\J_L$ are such that $j_0\leq\kappa_L(j'),\kappa_L(j'')$. Then we have $j'\to j''$ in $G_{L}$ if and only if $j'\not\leq\kappa_L(j'')$, and this occurs if and only if $j_0\vee j'\not\leq \kappa_L(j'')$. The first compatibility condition in \Cref{def:semidistrim} tells us that $\kappa_L(j'')=\kappa_{L^0}(j_0\vee j'')$, so it follows that $j'\to j''$ in $G_{L}$ if and only if $\alpha(j')\to\alpha(j'')$ in $G_{L^0}$. 

Now observe that $\J_{L_0}=\{j\in \J_L: j\leq m_0\}=\{j\in \J_L:j\leq\kappa_L(j_0)\}=\J_L\setminus(\{j_0\}\cup\In(j_0))$. Suppose $j',j''\in\J_{L_0}$. We have $j'\to j''$ in $G_L$ if and only if $j'\not\leq\kappa_L(j'')$, and this occurs if and only if $j'\not\leq m_0\wedge \kappa_L(j'')$. The second compatibility condition in \Cref{def:semidistrim} tells us that $\kappa_{L_0}(j'')=m_0\wedge \kappa_{L}(j'')$, so it follows that $j'\to j''$ in $G_{L}$ if and only if $j'\to j''$ in $G_{L_0}$. 
\end{proof}

\subsection{Semidistributive and trim lattices are compatibly dismantlable}\label{subsec:semidistributive_trim_compatibly}

\begin{proposition}
\label{thm:semidistributive_semidistrim}
Semidistributive lattices are compatibly dismantlable.
\end{proposition}
\begin{proof}
Let $L$ be a semidistributive lattice of cardinality at least $2$. We know by \Cref{prop:semidistributive_uniquely_paired} that $L$ is uniquely paired; let $\kappa_L\colon\J_L\to\M_L$ be its unique pairing. Let $j_0$ be an atom of $L$. By~\cite[Lemma 1]{gaskill1981join}, $j_0$ is necessarily join-prime.  Let $L_0=[\hat 0,m_0]$ and $L^0=[j_0,\hat 1]$, where $m_0=\kappa(j_0)$ so that $(j_0,m_0)$ is a prime pair. Since intervals of semidistributive lattices are semidistributive, it follows by induction on the size of $L$ that $L_0$ and $L^0$ are both compatibly dismantlable. Suppose $j\in\J_L$ is such that $j_0\leq \kappa_L(j)$, and let $m=\kappa_L(j)$. We want to show that $j_0\vee j$ is join-irreducible in $L^0$.  Because $L$ is semidistributive, $j$ is the unique element of the singleton set $\J_L(m)=\min\{z\in L: m^*=m\vee z\}$. Since $L^0$ is semidistributive, the set $\{z\in L^0 : m^*=m \vee z \}$ has a unique minimal element. Certainly $j_0 \vee j$ is in this set since $m \vee (j_0 \vee j) = m \vee j = m^*$.  Suppose there is some $x\in L^0$ with $x< j\vee j_0$ and $m\vee x=m^*$. We have $m\vee x=j\vee x=m^*$, so it follows from the fact that $L$ is join-semidistributive that $m\vee (x\wedge j)= m^*$. Now $j\in\J_L(m)$, so we must have $x \wedge j = j$, meaning $x\geq j$.
But $x\geq j_0$, so this contradicts the assumption that $x<j\vee j_0$. This shows that no such $x$ exists, so $j_0\vee j$ is the unique minimal element of $\{z\in L^0 : m^*=m \vee z \}$. In other words, $j_0\vee j\in\J_{L^0}(m)$. By \Cref{lem:contained_in_J}, $j_0\vee j\in\J_{L^0}$. 

We now have a map $\alpha\colon\{j\in\J_L:j_0\leq \kappa_L(j)\}\to\J_{L^0}$ given by $\alpha(j)=j_0\vee j$. Let us show that $\alpha$ is injective. Suppose $j,j'\in \J_L$ are such that $j_0\leq\kappa_L(j),\kappa_L(j')$ and $j_0\vee j=j_0\vee j'$. Let $m=\kappa_L(j)$ and $m'=\kappa_L(j')$. Then $(j_0\vee j)\vee m=j\vee (j_0\vee m)=j\vee m=m^*$, so $j'\vee m=j'\vee (j_0\vee m)=(j_0\vee j')\vee m=(j_0\vee j)\vee m=m^*$. Since $j$ is the unique minimal element of $\{z\in L:m^*=m\vee z\}$, we must have $j'\geq j$.  Reversing the roles of $j$ and $j'$ shows that $j\geq j'$. Hence, $j=j'$.

We now show that $\alpha$ is surjective and that $\kappa_{L^0}(\alpha(j))=\kappa_L(j)$ for every $j\in\J_L$ with $j_0\leq\kappa_L(j)$. Suppose $j'\in\J_{L^0}$. Let $m'=\kappa_{L^0}(j')$. Then $m'$ is meet-irreducible in $L^0$, so it is meet-irreducible in $L$. Let $j=\kappa_L^{-1}(m')$. We aim to show that $j_0\vee j=j'$. Since $m'\vee j'=(m')^*$ and $j$ is the unique minimal element of $\{z\in L:(m')^*=m'\vee z\}$, we have $j'\geq j$. Thus, $j'\geq j_0\vee j$. On the other hand, we have $(j_0\vee j)\vee m'=(m')^*$. Since $j'$ is the unique minimal element of the set $\{z\in L^0:(m')^*=m'\vee z\}$, we must have $j'\leq j_0\vee j$. This proves that $j'=j_0\vee j$. Moreover, $\kappa_L(j)=m'=\kappa_{L^0}(j')=\kappa_{L^0}(\alpha(j))$.  
\end{proof}

\begin{proposition}
\label{thm:trim_semidistrim}
Trim lattices are compatibly dismantlable.
\end{proposition}

\begin{proof}
By~\Cref{prop:extremal_paired}, trim lattices are uniquely paired.  \cite[Theorem 2.4]{thomas2019rowmotion} states that the Galois graph $G_L$ of a trim lattice is acyclic. Let $j_0$ be a sink of $G_L$. Then~\cite[Theorem 15]{markowsky92primes} tells us that $j_0$ is a join-prime atom of $L$; let $m_0$ be the corresponding meet-prime element so that $(j_0,m_0)$ is a prime pair. The proof of~\cite[Proposition 3.13]{thomas2019rowmotion} (following the proof of ~\cite[Theorem 1]{thomas2006analogue}) shows that the compatibility conditions in~\Cref{def:semidistrim} hold for this choice of prime pair. 
\end{proof}

\begin{remark}
\label{rem:recurrence}
Depending on the lattice $L$, the inductive dismantling $L=[\hat{0},m_0] \sqcup [j_0,\hat{1}]$ has been called various names:
\begin{itemize}
    \item if $L$ is the weak order of a finite Coxeter group $W$ and $j_0$ is an atom (a simple reflection), then $L$ is semidistributive and $[\hat{0},m_0]$ is a \emph{maximal parabolic quotient};
    \item if $L$ is a $c$-Cambrian lattice for a finite Coxeter group $W$ with Coxeter element $c$, then $L$ is both trim and semidistributive.  If $s$ is a simple reflection that is initial in $c$, then for $j_0=s$, the interval $[\hat{0},m_0]$ is the Cambrian lattice for a parabolic subgroup of $W$, and we obtain the \emph{Cambrian recurrence}~\cite{reading2007sortable};
    \item this recurrence was used in~\cite{thomas2019rowmotion} with $j_0$ taken to be a join-prime atom (i.e., a sink of the Galois graph) to prove many structural properties for trim lattices. \qedhere 
\end{itemize}  
\end{remark}

\begin{remark}
Although intervals of semidistributive lattices are again semidistributive and intervals of trim lattices are again trim, the same does not hold for compatibly dismantlable lattices.  An example of this failure is given in~\Cref{fig:not_intervals}.  We will address this deficiency in~\Cref{sec:semidistrim} by imposing one final condition on compatibly dismantlable lattices to form our titular \emph{semidistrim lattices}. The fact that the family of semidistrim lattices is closed under taking intervals will have several notable consequences. 
\end{remark}

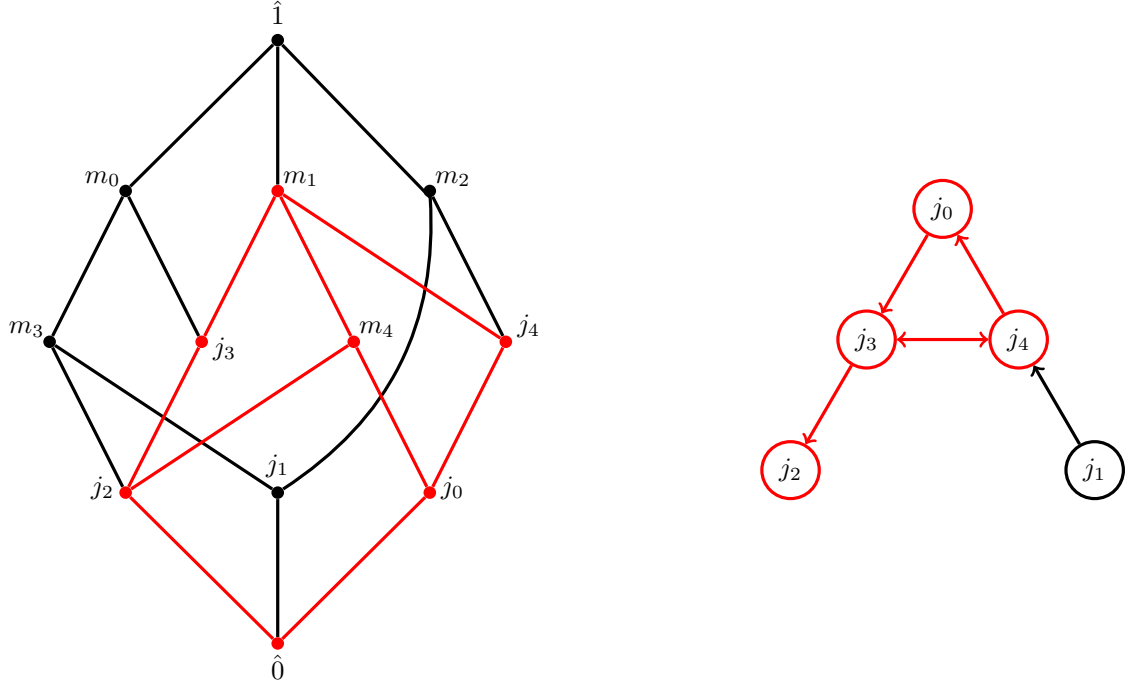
\begin{figure}[htbp]
\begin{center}
\raisebox{-.5\height}{\scalebox{1}{\begin{tikzpicture}[scale=2]
\node[shape=circle,fill=red, scale=0.5,label={[xshift=0ex, yshift=-4.5ex]$\hat{0}$}] (0) at (0,0) {};
\node[shape=circle,fill=red, scale=0.5,label={[xshift=-2ex, yshift=-2ex]$j_2$}] (1) at (-1,1) {};
\node[shape=circle,fill=black, scale=0.5,label={[xshift=0ex, yshift=0ex]$j_1$}] (2) at (0,1) {};
\node[shape=circle,fill=red, scale=0.5,label={[xshift=2ex, yshift=-2ex]$j_0$}] (3) at (1,1) {};
\node[shape=circle,fill=black, scale=0.5,label={[xshift=-2ex, yshift=-1ex]$m_3$}] (4) at (-1.5,2) {};
\node[shape=circle,fill=red, scale=0.5,label={[xshift=2ex, yshift=-3ex]$j_3$}] (5) at (-.5,2) {};
\node[shape=circle,fill=red, scale=0.5,label={[xshift=2ex, yshift=-1ex]$m_4$}] (6) at (.5,2) {};
\node[shape=circle,fill=red, scale=0.5,label={[xshift=2ex, yshift=-1ex]$j_4$}] (7) at (1.5,2) {};
\node[shape=circle,fill=black, scale=0.5,label={[xshift=-2ex, yshift=-1ex]$m_0$}] (8) at (-1,3) {};
\node[shape=circle,fill=red, scale=0.5,label={[xshift=2ex, yshift=-1ex]$m_1$}] (9) at (0,3) {};
\node[shape=circle,fill=black, scale=0.5,label={[xshift=2ex, yshift=-1ex]$m_2$}] (10) at (1,3) {};
\node[shape=circle,fill=black, scale=0.5,label={[xshift=0ex, yshift=0ex]$\hat{1}$}] (11) at (0,4) {};
\draw[very thick] (0) to (2) to (4) to (8) to (11) to (9);
\draw[very thick] (1) to (4);
\draw[very thick] (7) to (10);
\draw[very thick] (5) to (8);
\draw[very thick] (2) to [bend right] (10) to (11);
\draw[very thick,red] (0) to (1) to (5) to (9) to (6) to (1);
\draw[very thick,red] (0) to (3) to (6);
\draw[very thick,red] (3) to (7) to (9);
\end{tikzpicture}}}
\hspace{1in}
\raisebox{-.5\height}{\scalebox{1}{\begin{tikzpicture}[scale=2]
\node[shape=circle,very thick,fill=white,draw=red] (0) at (.5,-.866) {$j_2$};
\node[shape=circle,very thick,fill=white,draw=red] (1) at (1.5,.866) {$j_0$};
\node[shape=circle,very thick,fill=white,draw=black] (2) at (2.5,-.866) {$j_1$};
\node[shape=circle,very thick,fill=white,draw=red] (3) at (1,0) {$j_3$};
\node[shape=circle,very thick,fill=white,draw=red] (4) at (2,0) {$j_4$};
\draw[very thick,->,red] (3) to (0);
\draw[very thick,->,red] (1) to (3);
\draw[very thick,->,red] (4) to (1);
\draw[very thick,->] (2) to (4);
\draw[very thick,<->,red] (3) to (4);
\end{tikzpicture}}}
\end{center}

\caption{{\it Left:} A compatibly dismantlable lattice $L$ with a lower interval (in red) that is not compatibly dismantlable.    {\it Right:} The corresponding Galois graph $G_L$ with the Galois graph of the interval indicated as a subgraph (in red).}
\label{fig:not_intervals}
\end{figure}

\section{Overlapping Lattices}
\label{sec:independent}

Throughout this section, we assume $L$ is uniquely paired with unique pairing $\kappa_L\colon\J_L\to\M_L$.  Recall that for  $x\in L$, we write $J_L(x)=\{j\in\J_L:j\leq x\}$ and $M_L(x)=\{j\in\J_L:\kappa_L(j)\geq x\}$. Note that $M_L(x)\cap J_L(x)=\emptyset$ because if there were some $j\in M_L(x)\cap J_L(x)$, then we would have $j\leq \kappa_L(j)$, which is impossible.

\begin{definition}\label{def:overlapping}
A uniquely paired lattice $L$ is \defn{overlapping} if for every cover $x \lessdot y$ in $L$, $M_L(x) \cap J_L(y)$ contains a single element, which we denote $j_{xy}$.  If $L$ is overlapping and $x\in L$, then we define the \defn{downward label set} $\D_L(x) = \{j_{yx} : y \lessdot x\}$ and the \defn{upward label set} $\U_L(x)=\{j_{xy} : x \lessdot y\}$.
\end{definition}

By~\cite[Lemma 4.4]{reading2019fundamental} and~\cite[Theorem 3.4]{thomas2019rowmotion}, both semidistributive and trim lattices are overlapping. In fact, it was shown in~\cite{thomas2019rowmotion} that an extremal lattice is trim if and only if it is overlapping; in particular, extremal semidistributive lattices are trim.  \Cref{fig:not_defining} gives an example of an overlapping uniquely paired lattice that is not compatibly dismantlable.

\begin{lemma}
\label{prop:covers}
Let $L$ be an overlapping lattice with unique pairing $\kappa_L$.  If a cover relation $x \lessdot y$ is labeled by $j_{xy}$, then $x \vee j_{xy} = y$ and $y \wedge \kappa_L(j_{xy}) = x$.
\end{lemma}

\begin{proof}
Since $j_{xy}$ is in $M_L(x)$, it is not in $J_L(x)$. But $j_{xy}\in J_L(y)$, so $x\vee j_{xy}=y$. A similar argument establishes that $y \wedge \kappa_L(j_{xy}) = x$.
\end{proof}

We now show that every compatibly dismantlable lattice is overlapping. This will provide the edge-labeling needed for the definition of a semidistrim lattice; it will also be crucial for defining rowmotion on semidistrim lattices. 

\begin{proposition}
Compatibly dismantlable lattices are overlapping.
\end{proposition}

\begin{proof}
Suppose $L$ is compatibly dismantlable, and let $(j_0,m_0)$ be a dismantling pair for $L$. Write $L_0=[\hat 0,m_0]$ and $L^0=[j_0,\hat 1]$. Then $L_0$ and $L^0$ are compatibly dismantlable, and $L=L_0\sqcup L^0$. By induction on the size of $L$, we know that $L_0$ and $L^0$ are overlapping. Let $x \lessdot y$ be a cover relation in $L$. 

First, suppose $x,y\in L_0$. Then $J_{L_0}(y)=J_L(y)$, and we can use the second compatibility condition in \Cref{def:semidistrim} to see that $\kappa_{L_0}(j)=m_0\wedge \kappa_L(j)$ for every $j\in J_{L_0}(y)$. It follows that $M_L(x)\cap J_L(y)=M_{L_0}(x)\cap J_{L_0}(y)$, and this set is a singleton because $L_0$ is overlapping. 

Now suppose $x,y\in L^0$. An argument similar to the one used in the previous paragraph allows us to see that the map $\alpha$ from \Cref{def:semidistrim} yields a bijection from $M_L(x)\cap J_L(y)$ to $M_{L^0}(x)\cap J_{L^0}(y)$. As $L^0$ is overlapping, we have $|M_L(x)\cap J_L(y)|=|M_{L^0}(x)\cap J_{L^0}(y)|=1$.

We may now assume that $x \in L_0$ and $y \in L^0$. Since $\kappa_L(j_0)=m_0$ by \Cref{prop:pairing}, we have that $j_0 \in M_L(x) \cap J_L(y)$. Suppose by way of contradiction that there is some $j \in M_L(x) \cap J_L(y)$ with $j\neq j_0$. Let $m=\kappa_L(j)$ so that $m \geq x$ and $j \leq y$.  We have $x \leq x \vee j \leq y$ and $y \geq y \wedge m \geq x$.  If $x \vee j = x$, then $j\leq x \leq m$, which is impossible.  Therefore $x \vee j = y$.  Similarly, $y \wedge m = x$.  We have $x \leq m_0$, so if $j \leq m_0$, then $x \vee j = y \leq m_0$, which is a contradiction. Consequently, $j\not\leq m_0=\kappa_L(j_0)$.  A similar argument shows that $\kappa_L(j)=m\not\geq j_0$. This shows that we have edges $j\to j_0$ and $j_0\to j$ in the Galois graph $G_L$, which contradicts~\Cref{prop:no_double_edges} because $j_0$ is join-prime. Hence, $M_L(x)\cap J_L(y)=\{j_0\}$. 
\end{proof}

\begin{corollary}
\label{cor:down_labels}
Suppose $L$ is a compatibly dismantlable lattice with dismantling pair $(j_0,m_0)$. Let $x\lessdot y$ be a cover relation in $L$. If $x\leq m_0$ and $y\geq j_0$, then the label $j_{xy}$ of the cover $x\lessdot y$ is $j_0$.
\end{corollary}

\begin{corollary}\label{cor:label_sets_in_intervals}
Suppose $L$ is a compatibly dismantlable lattice with dismantling pair $(j_0,m_0)$. If $x \leq m_0$, then $\D_{L_0}(x) = \D_{L}(x)$, and $\U_{L_0}(x)=\U_L(x)\setminus \{j_0\}$. If $x\geq j_0$, then the map $\alpha$ given by $\alpha(j)=j_0\vee j$ induces bijections $\D_L(x) \setminus \{j_0\}\to\D_{L^0}(x)$ and $\U_{L}(x)\to\U_{L^0}(x)$.
\end{corollary}
\begin{proof}
We prove the case in which $x\geq j_0$; the proof when $x\leq m_0$ is similar. It follows from \Cref{cor:down_labels} that $\D_L(x)\setminus\{j_0\}$ is the set of labels of covers of the form $y\lessdot x$ with $y\in L^0$. If $j$ is the label of such a cover $y\lessdot x$ in $L$, then $j\leq x$ and $\kappa_L(j)\geq y$. This implies that $\alpha(j)\leq x$ and $\kappa_{L^0}(\alpha(j))=\kappa_L(j)\geq y$, so $\alpha(j)$ is the label of $y\lessdot x$ in $L^0$. This proves that $\alpha$ induces a bijection $\D_L(x)\setminus\{j_0\}\to\D_{L^0}(x)$. A similar argument shows that $\alpha$ induces a bijection $\U_L(x)\to\U_{L^0}(x)$. 
\end{proof}

When we define rowmotion on semidistrim lattices in \Cref{sec:pop_and_row}, we will use the downward and upward label sets of the elements. In order for the definition to work, we will need to know that every element is uniquely determined by its downward label set and also by its upward label set; this is the content of the next theorem. 

\begin{theorem}
\label{thm:labels}
Let $L$ be a compatibly dismantlable lattice.  Every element $x \in L$ is uniquely determined by its downward label set $\D_L(x)$, and it is also uniquely determined by its upward label set $\U_L(x)$. More precisely, \[x = \bigvee \D_L(x) = \bigwedge\kappa_L(\U_L(x)).\]
\end{theorem}
\begin{proof}
We show that $x = \bigvee\D_L(x)$; the proof that $x = \bigwedge\kappa_L(\U_L(x))$ follows from a completely analogous dual argument. Let $(j_0,m_0)$ be a dismantling pair for $L$. By \Cref{def:semidistrim}, the intervals $L_0=[\hat 0,m_0]$ and $L^0=[j_0,\hat 1]$ are compatibly dismantlable. We will make use of the containment $\J_{L_0}\subseteq\J_L$ and the bijection $\alpha\colon\{j\in\J_L:j_0\leq \kappa_L(j)\}\to\J_{L^0}$. If $x\in L_0$, then we can use \Cref{cor:label_sets_in_intervals} and induction on the size of the lattice to see that $x=\bigvee\D_{L_0}(x)=\bigvee\D_L(x)$. Therefore, in what follows, we may assume $x \in L^0$.  

If $j_0 \in \D_L(x)$, then we can use \Cref{cor:label_sets_in_intervals} and induction on the size of the lattice to see that \[x=\bigvee\D_{L^0}(x)=\bigvee_{j\in \D_L(x)\setminus\{j_0\}}\alpha(j)=\bigvee_{j\in \D_L(x)\setminus\{j_0\}}(j_0\vee j)=j_0\vee\bigvee(\D_L(x)\setminus\{j_0\})=\bigvee\D_L(x).\]

Now suppose $j_0\not\in\D_L(x)$. In particular, $x\neq j_0$, so there is some element of $L^0$ covered by $x$. Let $y = \bigvee\D_L(x)$. Our goal is to show that $y=x$, so assume otherwise. As in the previous case, we can use \Cref{cor:label_sets_in_intervals} and induction to find that \[x=\bigvee\D_{L^0}(x)=\bigvee_{j\in \D_L(x)}\alpha(j)=\bigvee_{j\in \D_L(x)}(j_0\vee j)=j_0\vee\bigvee\D_L(x)=j_0\vee y.\] Since $x\neq y$, this implies that $y\not\geq j_0$. Hence, $y\in L_0$. There exist elements $a\in L_0$ and $b\in L^0$ such that $y\leq a\lessdot b\leq x$. The label $j_{ab}$ of the cover $a\lessdot b$ is equal to $j_0$ by \Cref{cor:down_labels}. Since $j_0\not\in\D_L(x)$, this implies that $b<x$. However, this shows that $y\vee j_0\leq b<x$, which contradicts the fact that $y\vee j_0=x$.  
\end{proof}

The next lemma will be handy when we prove that intervals in semidistrim lattices are semidistrim. 

\begin{lemma}
\label{lem:meets_are_labeled}
Let $L$ be a compatibly dismantlable lattice, and let $x \in L$. Let $j\in\J_L$, and let $m=\kappa_L(j)\in\M_L$. If $j \leq x$, then $j \in \U_L(x \wedge m)$.  If $m \geq x$, then $j \in \D_L(x \vee j)$.
\end{lemma}
\begin{proof}
We only prove the first statement, as the second follows from an analogous dual argument. Thus, suppose $j\leq x$. Our goal is to show that there exists an element $y$ covering $x\wedge m$ such that $j\leq y$. Indeed, if we can do this, then (since $m\geq x\wedge m$) the cover relation $x\wedge m\lessdot y$ will have label $j$. Let $(j_0,m_0)$ be a dismantling pair for $L$, and let $L_0=[\hat 0,m_0]$ and $L^0=[j_0,\hat 1]$. Then $L_0$ and $L^0$ are compatibly dismantlable. 

Assume first that $x\in L_0$. By \Cref{def:semidistrim}, we have $\kappa_{L_0}(j)=\kappa_{L_0}(\kappa_L^{-1}(m))=\beta(m)=m_0\wedge m$, so we can use induction on the size of the lattice to see that $j\in\U_{L_0}(x\wedge(m_0\wedge m))=\U_{L_0}(x\wedge m)$. This means that there exists an element $y\in L_0$ covering $x\wedge m$ such that the cover relation $x\wedge m\lessdot y$ is labeled by $j$ in $L_0$. But then $j\leq y$, as desired. 

We may now assume that $x\in L^0$.  If $m\in L^0$, then we can use the first compatibility condition in \Cref{def:semidistrim} to see that $\alpha(j)$ is a join-irreducible element of the compatibly dismantlable lattice $L^0$ with $\kappa_{L^0}(\alpha(j))=m$ and $\alpha(j)\leq x$. By induction on the size of the lattice, we have $\alpha(j)\in\U_{L^0}(x\wedge m)$, so there exists an element $y$ covering $x\wedge m$ such that the cover relation $x\wedge m\lessdot y$ is labeled by $\alpha(j)$ in $L^0$. But then $j\leq j_0\vee j= \alpha(j)\leq y$. 

If $m=m_0$ (so $j=j_0$), then there exist $y,z$ such that $x\wedge m\leq z\lessdot y\leq x$ with $z\in L_0$ and $y\in L^0$. But then $z = x \wedge m_0$ and $j_0\leq y$, as desired.

Finally, suppose $m < m_0$. This means that the edge $j_0\to j$ appears in $G_L$, so \Cref{prop:no_double_edges} guarantees that the edge $j\to j_0$ does not appear. That is, $j \leq m_0$.  The second compatibility condition in \Cref{def:semidistrim} tells us that $\kappa_{L_0}(j)=\beta(m)=m_0\wedge m=m$. Since $j\leq m_0\wedge x$, we can use induction on the size of the lattice to see that $j\in\U_{L_0}((m_0\wedge x)\wedge m)=\U_{L_0}(x\wedge m)$. This means that there exists an element $y\in L_0$ covering $x\wedge m$ such that the cover relation $x\wedge m\lessdot y$ is labeled by $j$ in $L_0$. Then $j\leq y$.  
\end{proof}

\section{Semidistrim Lattices}
\label{sec:semidistrim}

An \defn{independent set} of a (directed or undirected) graph $G$ is a subset $I$ of the vertex set of $G$ such that no two vertices of $I$ are adjacent. Let $\I(G)$ denote the collection of independent sets of $G$. The extra condition we impose on compatibly dismantlable lattices to obtain semidistrim lattices is that the downward label sets and upward label sets of all of the elements are independent sets in the Galois graph.

\begin{definition}
A lattice $L$ is \defn{semidistrim} if it is compatibly dismantlable and $\D_L(x),\U_L(x)\in\I(G_L)$ for all $x \in L$.
\label{def:semidistrim2}
\end{definition}

Note that there are compatibly dismantlable lattices that are not semidistrim; \Cref{fig:reduced} shows an example.  We remark that semidistrim lattices are distinct from congruence normal (and hence congruence uniform) lattices.

The next theorem, most of which we already proved in \Cref{subsec:semidistributive_trim_compatibly}, justifies the name \emph{semidistrim}. 

\begin{theorem}
Semidistributive lattices are semidistrim, and trim lattices are semidistrim.
\end{theorem}

\begin{proof}
We already know that semidistributive lattices and trim lattices are compatibly dismantlable by \Cref{thm:semidistributive_semidistrim} and \Cref{thm:trim_semidistrim}. If $x$ is an element in a semidistributive lattice $L$, then $\D_L(x)$ is its canonical join representation, and $\U_L(x)$ is its canonical meet representation. It is known that these sets are independent sets in the Galois graph~\cite{barnard19canonical}.  If $x$ is an element in a trim lattice $L$, then $\D_L(x)$ and $\U_L(x)$ are independent sets in the Galois graph by~\cite[Corollary 5.6]{thomas2019rowmotion}. 
\end{proof}

\begin{lemma}\label{lem:decomposition_intervals_semidistrim}
If $L$ is a semidistrim lattice with dismantling pair $(j_0,m_0)$, then the intervals $L_0=[\hat 0,m_0]$ and $L^0=[j_0,\hat 1]$ are semidistrim. 
\end{lemma}

\begin{proof}
We know that these intervals are compatibly dismantlable by \Cref{def:semidistrim}. \Cref{prop:galois} tells us that $G_{L_0}$ is an induced subgraph of $G_L$ and that $G_{L^0}$ is isomorphic to an induced subgraph of $G_L$, where the isomorphism is the restriction of $\alpha^{-1}$ to the vertex set of $G_{L^0}$. Suppose $x\in L_0$ and $y\in L^0$; we need to show that $\D_{L_0}(x),\U_{L_0}(x)\in\I(G_{L_0})$ and that $\D_{L^0}(y),\U_{L^0}(y)\in\I(G_{L^0})$. Since $L$ is semidistrim, all of the sets $\D_L(x),\U_L(x),\D_L(y),\U_L(y)$ are independent sets in $G_L$. \Cref{cor:label_sets_in_intervals} tells us that $\D_{L_0}(x)$ and $\U_{L_0}(x)$ are subsets of $\D_L(x)$ and $\U_L(x)$, respectively, so they are independent sets in $G_{L_0}$. The same proposition tells us that $\alpha^{-1}(\D_{L^0}(y))\subseteq\D_L(y)$ and $\alpha^{-1}(\U_{L^0}(y))\subseteq\U_L(y)$. This shows that $\D_{L^0}(y)$ and $\U_{L^0}(y)$ are independent sets in $G_{L^0}$.
\end{proof}

The next theorem will allow us to define rowmotion on semidistrim lattices in~\Cref{sec:pop_and_row}.
\begin{theorem}
\label{thm:independent}
Let $L$ be a semidistrim lattice. The maps $\D_L\colon L\to\I(G_L)$ and $\U_L\colon L\to\I(G_L)$ are bijections.
\end{theorem}

\begin{proof}
The result is trivial if $|L|=1$, so we may assume $|L|\geq 2$ and proceed by induction on $|L|$. Let us prove that $\D_L$ is bijective; the proof that $\U_L$ is bijective follows from a completely analogous dual argument. We know that $\D_L$ is injective by \Cref{thm:labels}, so it suffices to show that it is surjective. Let $(j_0,m_0)$ be a dismantling pair for $L$, and let $L_0=[\hat 0,m_0]$ and $L^0=[j_0,\hat 1]$. The intervals $L_0$ and $L^0$ are semidistrim by \Cref{lem:decomposition_intervals_semidistrim}.

Fix $I\in\I(G_L)$. If $I\cap(\In(j_0)\cup\{j_0\})=\emptyset$, then \Cref{prop:galois} implies that $I$ is an independent set of $G_{L_0}$, so it follows by induction that there exists $z\in L_0$ such that $\D_{L_0}(z)=I$. But $\D_{L_0}(z)=\D_L(z)$ by \Cref{cor:label_sets_in_intervals}, so $I$ is in the image of $\D_L$. 

We may now assume that $I\cap(\In(j_0)\cup\{j_0\})\neq\emptyset$. If $j_0\in I$, then $I\cap\Out(j_0)=\emptyset$ because $I$ is an independent set. \Cref{prop:no_prob} tells us that every vertex in $\In(j_0)$ is adjacent to every vertex in $\Out(j_0)$; therefore, if $I$ contains an element of $\In(j_0)$, then $I\cap\Out(j_0)=\emptyset$. In either case, $I\cap\Out(j_0)=\emptyset$. According to \Cref{prop:galois}, $\alpha(I\setminus\{j_0\})\in \I(G_{L^0})$, so we can use induction on the size of $L$ to see that there exists $y\in L^0$ with $\D_{L^0}(y)=\alpha(I\setminus\{j_0\})$. But \Cref{cor:label_sets_in_intervals} tells us that $\alpha^{-1}(\D_{L^0}(y))=\D_L(y)\setminus\{j_0\}$, so $\D_L(y)\setminus\{j_0\}=I\setminus\{j_0\}$. We will prove that $\D_L(y)=I$, which will show that $I$ is in the image of $\D_L$, as desired. So assume by way of contradiction that $\D_L(y)\neq I$. This implies that either $I=\D_L(y)\sqcup\{j_0\}$ or $\D_L(y)=I\sqcup\{j_0\}$. 

If $I=\D_L(y)\sqcup\{j_0\}$, then $\D_L(y)$ is an independent set in $G_L$ that does not intersect $\In(j_0)\cup\{j_0\}$, so we know from an earlier case of this proof that $\D_L(y)=\D_L(z)$ for some $z\in L_0$. However, this contradicts the fact that $\D_L$ is injective (since $y\in L^0$ and $z\in L_0$). 

If $\D_L(y)=I\sqcup\{j_0\}$, then since we have assumed $I\cap(\In(j_0)\cup\{j_0\})\neq\emptyset$, the set $I$ must contain an element of $\In(j_0)$. However, this contradicts the fact that $\D_L(y)$ is an independent set in $G_L$. 
\end{proof}

\begin{remark}
Semidistrim lattices do not form the most general class of lattices for which the map $\I(G_L) \to L$ given by $I \mapsto \bigvee I$ is a bijection.  \Cref{fig:not_defining} gives an example of an overlapping---but not compatibly dismantlable---lattice with this property. 
\end{remark}

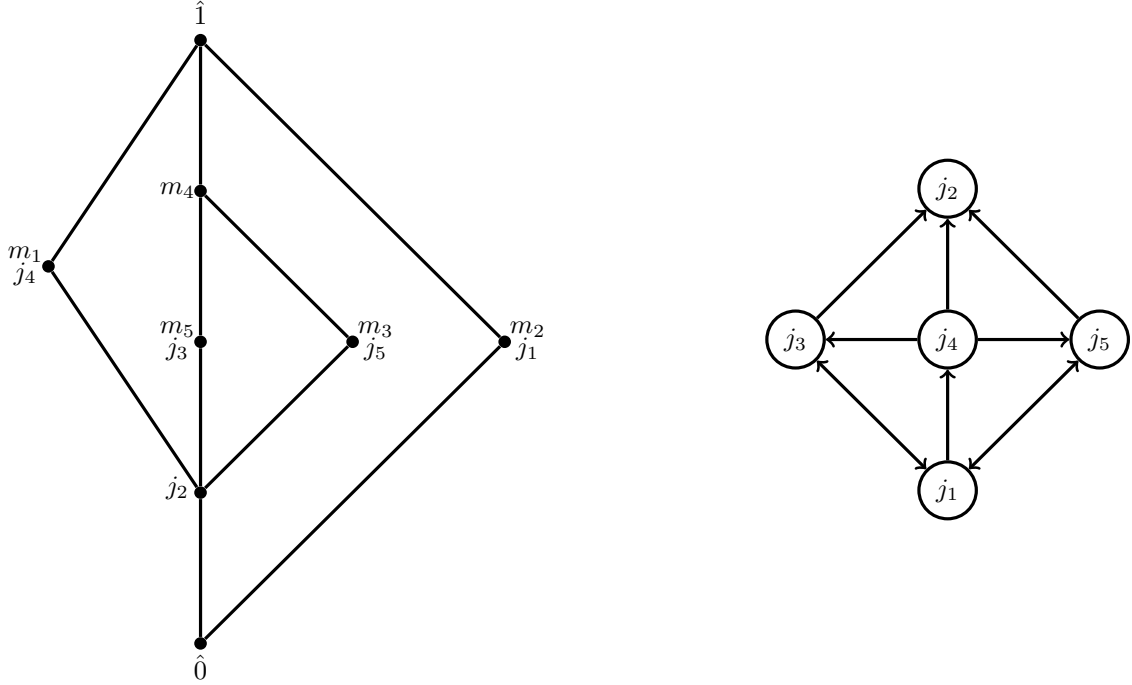
\begin{figure}[htbp]
\begin{center}
\raisebox{-.5\height}{\scalebox{1}{\begin{tikzpicture}[scale=2]
\node[shape=circle,fill=black,scale=0.5,label={[xshift=0ex, yshift=-4.5ex]$\hat{0}$}] (0) at (0,0) {};
\node[shape=circle,fill=black,scale=0.5,label={[xshift=-2ex, yshift=-2ex]$j_2$}] (1) at (0,1) {};
\node[shape=circle,fill=black,scale=0.5,label={[xshift=-2ex, yshift=-1ex]$m_1$},label={[xshift=-2ex, yshift=-3ex]$j_4$}] (2) at (-1,2.5) {};
\node[shape=circle,fill=black,scale=0.5,label={[xshift=-2ex, yshift=-3ex]$j_3$},label={[xshift=-2ex, yshift=-1ex]$m_5$}] (3) at (0,2) {};
\node[shape=circle,fill=black,scale=0.5,label={[xshift=2ex, yshift=-3ex]$j_5$},label={[xshift=2ex, yshift=-1ex]$m_3$}] (4) at (1,2) {};
\node[shape=circle,fill=black,scale=0.5,label={[xshift=2ex, yshift=-1ex]$m_2$},label={[xshift=2ex, yshift=-3ex]$j_1$}] (5) at (2,2) {};
\node[shape=circle,fill=black,scale=0.5,label={[xshift=-2ex, yshift=-2ex]$m_4$}] (6) at (0,3) {};
\node[shape=circle,fill=black,scale=0.5,label={[xshift=0ex, yshift=0ex]$\hat{1}$}] (7) at (0,4) {};
\draw[very thick] (0) to (1) to (3) to (6) to (7) to (5) to (0);
\draw[very thick] (1) to (2) to (7);
\draw[very thick] (1) to (4) to (6);
\end{tikzpicture}}}
\hspace{1in}
\raisebox{-.5\height}{\scalebox{1}{\begin{tikzpicture}[scale=2]
\node[very thick,shape=circle,fill=white,draw=black] (0) at (0,1) {$j_2$};
\node[very thick,shape=circle,fill=white,draw=black] (1) at (0,-1) {$j_1$};
\node[very thick,shape=circle,fill=white,draw=black] (2) at (1,0) {$j_5$};
\node[very thick,shape=circle,fill=white,draw=black] (3) at (0,0) {$j_4$};
\node[very thick,shape=circle,fill=white,draw=black] (4) at (-1,0) {$j_3$};
\draw[very thick,->] (3) to (0);
\draw[very thick,->] (1) to (3);
\draw[very thick,->] (3) to (4);
\draw[very thick,->] (3) to (2);
\draw[very thick,<->] (4) to (1);
\draw[very thick,->] (2) to (0);
\draw[very thick,->] (4) to (0);
\draw[very thick,<->] (2) to (1);
\end{tikzpicture}}}
\end{center}

\caption{{\it Left:} An overlapping lattice that is not compatibly dismantlable; the map $\I(G_L) \to L$ given by $I \mapsto \bigvee I$ is a bijection.  {\it Right:} The corresponding Galois graph.}
\label{fig:not_defining}
\end{figure}

If $L$ is semidistrim and $x\in L$, then the sets $\D_L(x)$ and $\U_L(x)$ are independent sets in $G_L$ by definition. The next theorem tells us more precisely how these sets fit together within the Galois graph. 

Suppose $G$ is a directed graph. An \defn{orthogonal pair} of $G$ is a pair $(X,Y)$ such that $X$ and $Y$ are disjoint independent sets of $G$ and such that there does not exist an edge of the form $j\to j'$ with $j\in X$ and $j'\in Y$. An orthogonal pair $(X,Y)$ is called \defn{tight} if the following additional conditions hold: 
\begin{itemize}
    \item If $j$ is a vertex of $G$ that is not in $X\cup Y$, then the pairs $(X\cup\{j\},Y)$ and $(X,Y\cup\{j\})$ are not orthogonal pairs. 
    \item If $j\to j'$ is an edge in $G$ such that $j\not\in X\cup Y$ and $j'\in X$, then the pair $((X\setminus\{j'\})\cup\{j\},Y)$ is not an orthogonal pair.
    \item If $j'\to j$ is an edge in $G$ such that $j\not\in X\cup Y$ and $j'\in Y$, then the pair $(X,(Y\setminus\{j'\})\cup\{j\})$ is not an orthogonal pair.
\end{itemize}

\begin{theorem}\label{thm:tops}
Let $L$ be a semidistrim lattice. For every $x\in L$, the pair $(\D_L(x),\U_L(x))$ is a tight orthogonal pair of $G_L$. 
\end{theorem}

\begin{proof}
Fix $x\in L$. If $j\in\D_L(x)$ and $j'\in \U_L(x)$, then $j\leq x\leq\kappa_L(j')$, so there is no edge $j\to j'$ in $G_L$. This shows that $(\D_L(x),\U_L(x))$ is an orthogonal pair. We need to show that this pair is tight. 

First, suppose $j$ is a vertex of $G_L$ that is not in $\D_L(x)\cup\U_L(x)$. To prove that $(\D_L(x)\cup\{j\},\U_L(x))$ and $(\D_L(x),\U_L(x)\cup\{j\})$ are not orthogonal pairs, it suffices to show that one of the sets $\In(j)\cap\D_L(x)$ or $\Out(j)\cap\U_L(x)$ is nonempty. Suppose instead that both of these sets are empty. Since $\Out(j)\cap\U_L(x)=\emptyset$, we have $j\leq\kappa_L(j')$ for all $j'\in \U_L(x)$, so it follows from \Cref{thm:labels} that $j\leq x$. Since $\In(j)\cap\D_L(x)=\emptyset$, we have $j'\leq\kappa_L(j)$ for all $j'\in \D_L(x)$, so it follows from \Cref{thm:labels} that $x\leq\kappa_L(j)$. Putting these together, we find that $j\leq x\leq\kappa_L(j)$, which is impossible. This proves that $(\D_L(x)\cup\{j\},\U_L(x))$ and $(\D_L(x),\U_L(x)\cup\{j\})$ are not orthogonal pairs. 

Now suppose there is an edge $j\to j'$ in $G_L$ with $j\not\in\D_L(x)\cup\U_L(x)$ and $j'\in\D_L(x)$. Our goal is to show that $((\D_L(x)\setminus\{j'\})\cup\{j\},\U_L(x))$ is not an orthogonal pair. Suppose instead that this pair is orthogonal. Then $(\D_L(x)\setminus\{j'\})\cup\{j\}$ is an independent set in $G_L$, so \Cref{thm:independent} tells us that there exists $x'\in L$ with $\D_L(x')=(\D_L(x)\setminus\{j'\})\cup\{j\}$. Furthermore, we have $\Out(j)\cap\U_L(x)=\emptyset$, so $j\leq\kappa_L(j'')$ for all $j''\in\U_L(x)$. By \Cref{thm:labels}, this implies that $j\leq x$. We also know that $j''\leq x$ for all $j''\in\D_L(x)$, so \Cref{thm:labels} tells us that $x'=\bigvee((\D_L(x)\setminus\{j'\})\cup\{j\})\leq x$. Because $\D_L(x)\neq\D_L(x')$, we must have $x'<x$. Let $y$ be an element such that $x'\leq y\lessdot x$, and consider the label $j_{yx}$ of the cover relation $y\lessdot x$. We have $x'\leq y\leq\kappa_L(j_{yx})$. If $j_{yx}=j'$, then $j\leq x'\leq\kappa_L(j_{yx})=\kappa_L(j')$, contradicting the assumption that the edge $j\to j'$ appears in $G_L$. On the other hand, if $j_{yx}$ is not $j'$, then $j_{yx}\in\D_L(x)\setminus\{j'\}\subseteq\D_L(x')$, so $j_{yx}\leq x'\leq\kappa_L(j_{yx})$, which is impossible. 

Finally, we must show that if there is an edge $j'\to j$ in $G_L$ such that $j\not\in \D_L(x)\cup \U_L(x)$ and $j'\in \U_L(x)$, then $(\D_L(x),(\U_L(x)\setminus\{j'\})\cup\{j\})$ is not an orthogonal pair. This follows from an argument that is dual to the one used in the previous paragraph. 
\end{proof}

\begin{remark}
Given a semidistrim lattice $L$, \Cref{thm:tops} yields a map $x\mapsto (\D_L(x),\U_L(x))$ from $L$ to the set of tight orthogonal pairs of $G_L$. \Cref{thm:labels} implies that this map is injective.
Thomas and the second author \cite{thomas2019independence} proved that if $L$ is trim, then this map is actually a bijection. However, for arbitrary semidistrim lattices, the map can fail to be surjective. An example is given in~\Cref{fig:tops}.
\end{remark}

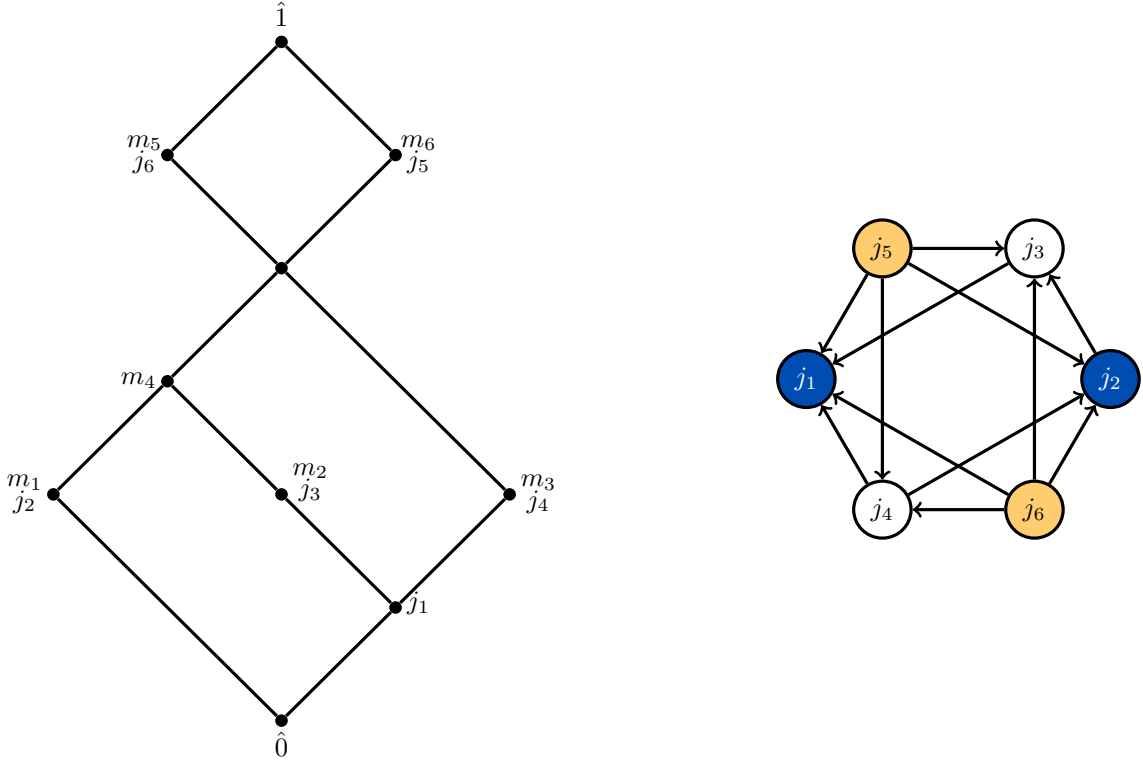
\begin{figure}[htbp]
\begin{center}
\raisebox{-.5\height}{\scalebox{1}{\begin{tikzpicture}[scale=1.5]
\node[shape=circle,fill=black,scale=0.5,label={[xshift=0ex, yshift=-4.5ex]$\hat{0}$}] (0) at (0,0) {};
\node[shape=circle,fill=black,scale=0.5,label={[xshift=2ex, yshift=-2ex]$j_1$}] (j1) at (1,1) {};
\node[shape=circle,fill=black,scale=0.5,label={[xshift=-2.5ex, yshift=-2ex]$m_4$}] (m4) at (-1,3) {};

\node[shape=circle,fill=black,scale=0.5,label={[xshift=-2.5ex, yshift=-1ex]$m_1$},label={[xshift=-2.5ex, yshift=-3ex]$j_2$}] (j2) at (-2,2) {};
\node[shape=circle,fill=black,scale=0.5,label={[xshift=2.5ex, yshift=-2ex]$j_3$},label={[xshift=2.5ex, yshift=0ex]$m_2$}] (j3) at (0,2) {};
\node[shape=circle,fill=black,scale=0.5,label={[xshift=2.5ex, yshift=-3ex]$j_4$},label={[xshift=2.5ex, yshift=-1ex]$m_3$}] (j4) at (2,2) {};
\node[shape=circle,fill=black,scale=0.5] (11) at (0,4) {};

\node[shape=circle,fill=black,scale=0.5,label={[xshift=-2ex, yshift=-3ex]$j_6$},label={[xshift=-2ex, yshift=-1ex]$m_5$}] (j6) at (-1,5) {};
\node[shape=circle,fill=black,scale=0.5,label={[xshift=2ex, yshift=-3ex]$j_5$},label={[xshift=2ex, yshift=-1ex]$m_6$}] (j5) at (1,5) {};
\node[shape=circle,fill=black,scale=0.5,label={[xshift=0ex, yshift=0ex]$\hat{1}$}] (1) at (0,6) {};
\draw[very thick] (0) to (j1) to (j4) to (11) to (j6) to (1) to (j5) to (11) to (m4) to (j2) to (0);
\draw[very thick] (j1) to (j3) to (m4);
\end{tikzpicture}}}
\hspace{1in}
\raisebox{-.5\height}{\scalebox{1}{\begin{tikzpicture}[scale=2]
\node[shape=circle,draw=black,fill=join,very thick] (1) at (-1,0) {\textcolor{white}{$j_1$}};
\node[shape=circle,draw=black,fill=join,very thick] (2) at (1,0) {\textcolor{white}{$j_2$}};
\node[shape=circle,draw=black,very thick] (3) at (.5,.866) {$j_3$};
\node[shape=circle,draw=black,very thick] (4) at (-.5,-.866) {$j_4$};
\node[shape=circle,draw=black,fill=meet,very thick] (5) at (-.5,.866) {$j_5$};
\node[shape=circle,draw=black,fill=meet,very thick] (6) at (.5,-.866) {$j_6$};
\draw[->,very thick] (4) to (1);
\draw[->,very thick] (3) to (1);
\draw[->,very thick] (5) to (1);
\draw[->,very thick] (6) to (1);
\draw[->,very thick] (6) to (2);
\draw[->,very thick] (5) to (2);
\draw[->,very thick] (4) to (2);
\draw[->,very thick] (2) to (3);
\draw[->,very thick] (6) to (3);
\draw[->,very thick] (5) to (3);
\draw[->,very thick] (6) to (4);
\draw[->,very thick] (5) to (4);
\end{tikzpicture}}}
\end{center}

\caption{{\it Left:} A semidistrim lattice $L$.  {\it Right:} The corresponding Galois graph, with a tight orthogonal pair $(X,Y)$ that is not $(\D(x),\U(x))$ for any element $x \in L$ (the elements of $X$ are shaded blue, while the elements of $Y$ are shaded yellow).}
\label{fig:tops}
\end{figure}

\begin{remark}
Even though semidistributive lattices and trim lattices have join-prime atoms and meet-prime coatoms (not necessarily corresponding to each other under the pairing), the same is not generally true of semidistrim lattices.  A counterexample is given in~\Cref{fig:noatoms}.
\end{remark}

\begin{figure}[htbp]
\begin{center}
\raisebox{-.5\height}{\scalebox{1}{\begin{tikzpicture}[scale=1]
\node[shape=circle,fill=black, scale=0.5] (1) at (0,0) {};
\node[shape=circle,fill=black, scale=0.5] (2) at (-1,1) {};
\node[shape=circle,fill=black, scale=0.5] (3) at (1,1) {};
\node[shape=circle,fill=black, scale=0.5] (4) at (-2,2) {};
\node[shape=circle,fill=black, scale=0.5] (5) at (-1,2) {};
\node[shape=circle,fill=black, scale=0.5] (6) at (1,2) {};
\node[shape=circle,fill=black, scale=0.5] (7) at (2,2) {};
\node[shape=circle,fill=black, scale=0.5] (8) at (0,3) {};
\node[shape=circle,fill=black, scale=0.5] (9) at (0,4) {};
\node[shape=circle,fill=black, scale=0.5] (10) at (.5,4.5) {};
\node[shape=circle,fill=black, scale=0.5] (11) at (0,9.5) {};
\node[shape=circle,fill=black, scale=0.5] (12) at (-1,8.5) {};
\node[shape=circle,fill=black, scale=0.5] (13) at (1,8.5) {};
\node[shape=circle,fill=black, scale=0.5] (14) at (-2,7.5) {};
\node[shape=circle,fill=black, scale=0.5] (15) at (-1,7.5) {};
\node[shape=circle,fill=black, scale=0.5] (16) at (1,7.5) {};
\node[shape=circle,fill=black, scale=0.5] (17) at (2,7.5) {};
\node[shape=circle,fill=black, scale=0.5] (18) at (0,6.5) {};
\node[shape=circle,fill=black, scale=0.5] (19) at (0,5.5) {};
\node[shape=circle,fill=black, scale=0.5] (20) at (.5,5) {};
\draw[very thick] (1) to (2) to (4) to (9) to (10) to (7) to (3) to (1);
\draw[very thick] (2) to (5) to (8) to (9);
\draw[very thick] (3) to (6) to (8);
\draw[very thick] (11) to (12) to (14) to (19) to (20) to (17) to (13) to (11);
\draw[very thick] (12) to (15) to (18) to (19);
\draw[very thick] (13) to (16) to (18);
\draw[very thick] (10) to (20);
\end{tikzpicture}}}
\end{center}
\caption{A semidistrim lattice with no join-prime atom or meet-prime coatom (in contrast to semidistributive and trim lattices).}
\label{fig:noatoms}
\end{figure}
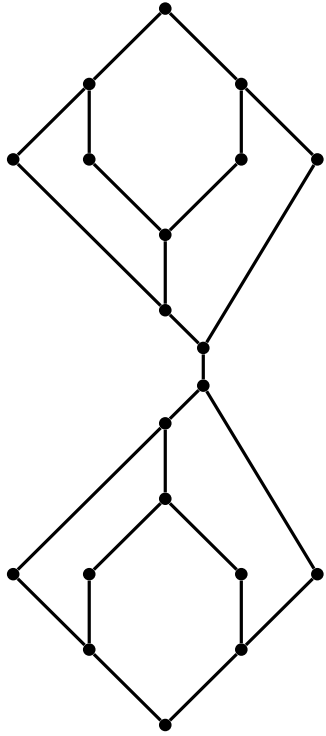

\section{Products and Intervals of Semidistrim Lattices}\label{sec:products_intervals}

The goal of this section is to show that products of semidistrim lattices are semidistrim and that intervals in semidistrim lattices are semidistrim. The proof of the statement about products is more straightforward, so we will begin with that. 

\subsection{Products}

Recall that if $P$ and $P'$ are two posets, then their \defn{product} is the poset $P\times P'$ whose underlying set is the Cartesian product of $P$ and $P'$, where $(x,x')\leq (y,y')$ if and only if $x\leq y$ in $P$ and $x'\leq y'$ in $P'$. We will actually show that the classes of uniquely paired lattices, compatibly dismantlable lattices, and semidistrim lattices are each closed under taking products. 

In what follows, $L$ and $L'$ are lattices. The product $L\times L'$ is also a lattice, and its meet and join operations are given by $(x,y)\wedge(x',y')=(x\wedge y,x'\wedge y')$ and $(x,y)\vee(x',y')=(x\vee y,x'\vee y')$. We write $\hat 0$ and $\hat 1$ (respectively, $\hat 0'$ and $\hat 1'$) for the minimum and maximum elements of $L$ (respectively, $L'$). It is straightforward to check that 
\begin{equation}\label{Eq1}\J_{L\times L'}=(\J_L\times\{\hat 0'\})\sqcup(\{\hat 0\}\times\J_{L'})\quad\text{and}\quad\M_{L\times L'}=(\M_L\times\{\hat 1'\})\sqcup(\{\hat 1\}\times\M_{L'}).
\end{equation}

\begin{proposition}\label{prop:product_uniquely}
If $L$ and $L'$ are uniquely paired, then $L\times L'$ is uniquely paired. 
\end{proposition}

\begin{proof}
Suppose $\lambda\colon\J_{L\times L'}\to\M_{L\times L'}$ is a pairing. If $j\in\J_L$, then since $(j,\hat 0')\not\leq\lambda(j,\hat 0')$, it follows from \eqref{Eq1} that $\lambda(j,\hat 0')=(\mu(j),\hat 1')$ for some $\mu(j)\in\M_L$. Similarly, if $j'\in\J_{L'}$, then $\lambda(\hat 0,j')=(\hat 1,\mu'(j'))$ for some $\mu'(j')\in\M_{L'}$. It is straightforward to check that the resulting maps $\mu\colon\J_L\to\M_L$ and $\mu'\colon\J_{L'}\to\M_{L'}$ are pairings on $L$ and $L'$, respectively. Since $L$ and $L'$ are uniquely paired, we have $\mu=\kappa_L$ and $\mu'=\kappa_{L'}$. This shows that if $L\times L'$ is paired, then it is uniquely paired. On the other hand, one can readily check that the map $\kappa_{L\times L'}\colon\J_{L\times L'}\to\M_{L\times L'}$ given by $\kappa_{L\times L'}(j,\hat 0')=(\kappa_L(j),\hat 1')$ and $\kappa_{L\times L'}(\hat 0,j')=(\hat 1,\kappa_{L'}(j'))$ is indeed a pairing on $L\times L'$. 
\end{proof}

\begin{proposition}\label{prop:product_compatibly}
If $L$ and $L'$ are compatibly dismantlable, then $L\times L'$ is compatibly dismantlable. 
\end{proposition}

\begin{proof}
\Cref{prop:product_uniquely} tells us that $L\times L'$ is uniquely paired, and the proof shows how to compute $\kappa_{L\times L'}$ in terms of $\kappa_L$ and $\kappa_{L'}$. Let $(j_0,m_0)$ and $(j_0',m_0')$ be dismantling pairs for $L$ and $L'$, respectively. One readily checks that $((j_0,\hat 0'),(m_0,\hat 1'))$ is a dismantling pair for $L\times L'$. Using the easily-verified decomposition \[\{(a,b)\in \J_{L\times L'}:(j_0,\hat 0')\leq \kappa_{L\times L'}(a,b)\}=\{(j,\hat 0'):j_0\leq\kappa_L(j)\}\sqcup(\{\hat 0\}\times\J_{L'}),\] one can check that the map $\alpha\colon\{(a,b)\in \J_{L\times L'}:(j_0,\hat 0')\leq \kappa_{L\times L'}(a,b)\}\to\J_{[(j_0,\hat 0'),(\hat 1,\hat 1')]}$ given by $\alpha(a,b)=(j_0,\hat 0')\vee(a,b)$ is a bijection such that $\kappa_{[(j_0,\hat 0'),(\hat 1,\hat 1')]}(\alpha(a,b))=\kappa_{L\times L'}(a,b)$ for all $(a,b)\in\J_{L\times L'}$ with $(j_0,\hat 0')\leq \kappa_{L\times L'}(a,b)$. In other words, the first compatibility condition in \Cref{def:semidistrim} holds for $L\times L'$. The second compatibility condition also holds similarly. 
\end{proof}

\begin{theorem}\label{thm:product_semidistrim}
If $L$ and $L'$ are semidistrim, then $L\times L'$ is semidistrim. 
\end{theorem}

\begin{proof}
\Cref{prop:product_compatibly} tells us that $L\times L'$ is compatibly dismantlable. Using \eqref{Eq1} and the description of $\kappa_{L\times L'}$ given in the proof of \Cref{prop:product_uniquely}, it is straightforward to show that the Galois graph $G_{L\times L'}$ is naturally isomorphic to $G_L\sqcup G_{L'}$, the disjoint union of $G_L$ and $G_{L'}$ (with no vertices in $G_L$ adjacent to any vertices in $G_{L'}$). More precisely, the isomorphism sends $(j,\hat 0')\in\J_L\times\{\hat 0'\}$ to $j$ and sends $(\hat 0,j')\in\{\hat 0\}\times\J_{L'}$ to $j'$. 

Consider $(y,y')\in L\times L'$. The elements covered by $(y,y')$ in $L\times L'$ are precisely the elements of the form $(x,y')$ with $x\lessdot y$ in $L$ or of the form $(y,x')$ with $x'\lessdot y'$ in $L'$. Moreover, the label of $(x,y')\lessdot (y,y')$ is $(j_{xy},\hat 0')$, while the label of $(y,x')\lessdot (y,y')$ is $(\hat 0,j_{x'y'})$. It follows that under the aforementioned isomorphism between $G_{L\times L'}$ and $G_L\sqcup G_{L'}$, the set $\D_{L\times L'}(y,y')$ corresponds to $\D_L(y)\sqcup\D_{L'}(y')$. Since $\D_L(y)\in\I(G_L)$ and $\D_{L'}(y')\in\I(G_{L'})$, we have $\D_{L\times L'}(y,y')\in\I(G_{L\times L'})$. A similar argument shows that $\U_{L\times L'}(y,y')\in\I(G_{L\times L'})$.  
\end{proof}

\subsection{Intervals}\label{subsec:intervals}
We now show that intervals in semidistrim lattices are again semidistrim.  The strategy is to prove that lower intervals of semidistrim lattices are again semidistrim---the result will then follow because the dual of a semidistrim lattice is again semidistrim.

\begin{lemma}
Let $L$ be a semidistrim lattice, and let $\x\in L$. Let $L_v=[\hat 0,v]$. For $j\in\mathcal J_{L_\x}$, let $\kappa^{\wedge \x}(j)=\kappa_L(j)\wedge \x$. Then $\kappa^{\wedge \x}\colon\J_{L_\x}\to\M_{L_\x}$ is a pairing on $L_\x$.
\label{lem:same_num}
\end{lemma}

\begin{proof}
Let $(j_0,m_0)$ be a dismantling pair for $L$, and let $L_0=[\hat 0,m_0]$ and $L^0=[j_0,\hat 1]$. First, suppose $\x\in L_0$. \Cref{lem:decomposition_intervals_semidistrim} tells us that $L_0$ is semidistrim, and the second compatibility condition in \Cref{def:semidistrim} states that $\kappa_{L_0}(j)=m_0\wedge \kappa_L(j)$ for all $j\in\J_{L}$ with $j\leq m_0$. In particular, for $j\in\J_{L_\x}$, we have $\kappa^{\wedge \x}(j)=\kappa_L(j)\wedge v=\kappa_{L_0}(j)\wedge v$. Since $L_v$ is a lower interval in $L_0$, we can use induction on the size of the lattice to see that $\kappa^{\wedge v}$ is a pairing on $L_v$.  

Now assume $\x\in L^0$. Let $j\in\J_{L_v}$. By \Cref{lem:meets_are_labeled}, we have $j \in \U_{L}(\kappa^{\wedge\x}(j))$. In other words, there is some cover relation $\kappa^{\wedge\x}(j)\lessdot y$ with label $j$. Note that $y=\kappa^{\wedge\x}(j)\vee j\leq \x$. Suppose by way of contradiction that there is some element $z\in L_v$ with $z\neq y$ and $\kappa^{\wedge\x}(j)\lessdot z$. Let $j'$ be the label of the cover $\kappa^{\wedge\x}(j)\lessdot z$. Then $j'\not\leq \kappa^{\wedge\x}(j)=\kappa_L(j)\wedge\x$. We have $j'\leq \x$, so $j'\not\leq \kappa_L(j)$. This means that $j'\to j$ is an edge in $G_L$, which contradicts the fact that $\U_L(\kappa^{\wedge\x}(j))$ is an independent set in $G_L$ (because $L$ is semidistrim). This shows that no such element $z$ exists, which implies that $y$ is the only element of $L_v$ that covers $\kappa^{\wedge\x}(j)$. Hence, $\kappa^{\wedge \x}(j)\in\M_{L_\x}$. This shows that $\kappa^{\wedge\x}$ maps $\J_{L_\x}$ to $\M_{L_\x}$. We have also seen that $\kappa^{\wedge\x}$ is injective because we can recover $j$ as the label of the cover $\kappa^{\wedge\x}\lessdot y$. 

To see that $\kappa^{\wedge\x}$ is surjective, consider some $m\in\M_{L_\x}$. Let $m^*$ be the unique element of $L_v$ that covers $m$. Consider the label $j_{mm^*}$ of the cover $m\lessdot m^*$. Then $\kappa^{\wedge \x}(j_{mm^*})=\kappa_L(j_{mm^*})\wedge\x\geq m$. If we had $\kappa^{\wedge\x}(j_{mm^*})\geq m^*$, then we would have $j_{mm^*}\leq m^*\leq\kappa^{\wedge\x}(j_{mm^*})\leq\kappa_L(j_{mm^*})$, which is impossible. Therefore, $\kappa^{\wedge\x}(j_{mm^*})=m$. 

We now know that $\kappa^{\wedge\x}$ is a bijection, so we need to check that it is a pairing. To do this, we will use \Cref{lem:paired_check}. Fix $j'\in\J_{L_\x}$, and let $m'=\kappa^{\wedge\x}(j')$. We need to show that $m'\geq j'_*$, $(m')^*\geq j'$, and $m'\not\geq j'$. We saw above that $j'$ is the label of the cover $m'\lessdot (m')^*$, so $(m')^*=m'\vee j'$. This proves that $(m')^*\geq j'$ and $m'\not\geq j'$. Since $\kappa_L$ is a pairing, we have $\kappa_L(j')\geq j'_*$. We also know that $\x\geq j'_*$, so $m'=\kappa_L(j')\wedge \x\geq j'_*$. 
\end{proof}

Our next step is to show that the pairing $\kappa^{\wedge\x}$ defined in \Cref{lem:same_num} is the only pairing on $L_\x$.

\begin{lemma}
Let $L$ be a semidistrim lattice, and fix $\x\in L$. The interval $L_\x= [\hat{0},\x]$ is uniquely paired; its unique pairing is the map $\kappa^{\wedge\x}\colon\mathcal J_{L_\x}\to\mathcal M_{L_\x}$ defined by $\kappa^{\wedge\x}(j)=\kappa_L(j)\wedge \x$. 
\label{lem:paired}
\end{lemma}

\begin{proof}
We know by \Cref{lem:same_num} that $\kappa^{\wedge\x}$ is a pairing on $L_\x$, so we just need to show that it is the 
unique pairing. Let $\lambda\colon\mathcal J_{L_\x}\to\mathcal M_{L_\x}$ be a pairing on $L_\x$, and define $\sigma\colon\mathcal J_{L_\x}\to\mathcal J_{L_\x}$ by 
$\sigma=(\kappa^{\wedge\x})^{-1}\circ\lambda$. Our goal is to show that $\lambda=\kappa^{\wedge\x}$ or, equivalently, that $\sigma$ is the identity map on $\mathcal J_{L_\x}$. Define a map $\widetilde\kappa\colon\mathcal J_L\to\mathcal M_L$ by \[
\widetilde\kappa(j) = \begin{cases}
\kappa_L(j) & \text{ if } j \not\in L_\x \\
\kappa_L(\sigma(j)) & \text{ if } j\in L_{\x}.
\end{cases}\]   
We will prove that $\widetilde\kappa$ is a pairing on $L$. Since $L$ is uniquely paired, this will imply that $\widetilde\kappa=\kappa_L$. It will then follow that $\kappa_L(j)=\kappa_L(\sigma(j))$ for all $j\in\mathcal J_{L_\x}$. Because $\kappa_L$ is injective, this will imply that $\sigma$ is the identity, as desired. 

It is straightforward to check that $\widetilde\kappa$ is a bijection. Now fix $j\in\mathcal J_L$, and let $m=\widetilde\kappa(j)$. According to \Cref{lem:paired_check}, we need to check that $m\geq j_*$, $m^*\geq j$, and $m\not\geq j$. If $j\not\in L_\x$, then $m=\kappa_L(j)$, so these conditions follow from the fact that $\kappa_L$ is a pairing on $L$. Thus, we may assume $j\in L_\x$. 

Note that \[m\wedge \x=\widetilde\kappa(j)\wedge \x=\kappa_L((\kappa^{\wedge\x})^{-1}(\lambda(j)))\wedge \x=\kappa^{\wedge\x}((\kappa^{\wedge\x})^{-1}(\lambda(j)))=\lambda(j).\] We have \[m\wedge j=m\wedge(\x\wedge j)=(m\wedge \x)\wedge j=\lambda(j)\wedge j=j_*,\] where the last equality is due to the fact that $\lambda$ is a pairing on $L_\x$. This proves that $m\geq j_*$ and $m\not\geq j$. 

We are left to show that $m^*\geq j$. As above, note that $m\wedge \x =\lambda(j)$. Because $\lambda$ is a pairing on $\M_{L_v}$, the element $\lambda(j)$ is meet-irreducible in $L_v$. Let $y$ be the unique element of $L_\x$ that covers $\lambda(j)$. Since $\lambda(j)=\kappa^{\wedge\x}(\sigma(j))$ and $\kappa^{\wedge\x}$ is a pairing on $L_\x$, we have $y=\lambda(j)\vee\sigma(j)$. Since $m=\kappa_L(\sigma(j))$ and $\kappa_L$ is a pairing on $L$, we have $m^*\geq \sigma(j)$. We also have $m^*\gtrdot m\geq m\wedge \x=\lambda(j)$, so $m^*\geq \lambda(j)\vee\sigma(j)=y$. Since $\lambda$ is a pairing on $L_\x$ and $y$ is the unique element of $L_\x$ covering $\lambda(j)$, we must have $y\geq j$. Thus, $m^*\geq y\geq j$, as desired.  
\end{proof}

We now know that the interval $L_\x=[\hat 0,v]$ of $L$ is uniquely paired, so we will usually write $\kappa_{L_\x}$ instead of $\kappa^{\wedge\x}$ for its unique pairing. One should keep in mind that $\kappa_{L_v}$ is described explicitly in terms of $\kappa_L$ by the formula $\kappa_{L_v}(j)=\kappa_L(j)\wedge v$.

\begin{lemma}\label{lem:intervalsrowmotable}
If $L$ is semidistrim and $\x\in L$, then the interval $L_\x= [\hat{0},\x]$ is compatibly dismantlable.
\end{lemma}

\begin{proof}
We know by \Cref{lem:paired} that $L_\x$ is uniquely paired with pairing $\kappa_{L_\x}\colon\mathcal J_{L_\x}\to\mathcal M_{L_\x}$ given by $\kappa_{L_\x}(j)=\kappa_L(j)\wedge \x$. Let $(j_0,m_0)$ be a dismantling pair for $L$, and let $L_0=[\hat 0,m_0]$ and $L^0=[j_0,\hat 1]$. According to \Cref{lem:decomposition_intervals_semidistrim}, the intervals $L_0$ and $L^0$ are semidistrim. Therefore, if $\x\in L_0$, then the desired result follows by induction on the size of $L$. 

We may now assume $\x\in L^0$. Note that $j_0$ is also join-prime in $L^0$. Let $m_0^\x=m_0\wedge \x$ so that $(j_0,m_0^\x)$ is a dismantling pair for $L_\x$. Since $L_0$ and $L^0$ are semidistrim, they have unique pairings $\kappa_{L_0}\colon\mathcal J_{L_0}\to\mathcal M_{L_0}$ and $\kappa_{L^0}\colon\mathcal J_{L^0}\to\mathcal M_{L^0}$. By induction, the intervals $(L_\x)_0=[\hat{0},m_0^v]\subseteq L_0$ and $(L_\x)^0=[j_0,\x]\subseteq L^0$ are compatibly dismantlable. Moreover, it follows from \Cref{lem:paired} that the unique pairings on these intervals are the bijections $\kappa_{(L_\x)_0}\colon\mathcal J_{(L_\x)_0}\to\mathcal M_{(L_\x)_0}$ and $\kappa_{(L_\x)^0}\colon\mathcal J_{(L_\x)^0}\to\mathcal M_{(L_\x)^0}$ defined by $\kappa_{(L_\x)_0}(j)=\kappa_{L_0}(j)\wedge m_0^\x$ and $\kappa_{(L_\x)^0}(j)=\kappa_{L^0}(j)\wedge \x$.

Because $L$ is compatibly dismantlable, there is a bijection $\alpha_L\colon\{j\in\J_L:j_0\leq\kappa_L(j)\}\to\J_{L^0}$ given by $\alpha_L(j)=j_0\vee j$. We want to show that there is a bijection $\alpha_{L_\x}\colon\{j \in \J_{L_\x}:j_0\leq \kappa_{L_\x}(j)\}\to\J_{(L_\x)^0}$ given by $\alpha_{L_\x}(j)=j_0\vee j$.  To this end, suppose $j$ is a join-irreducible element of $L_\x$ with $j_0\leq \kappa_{L_\x}(j)$. Then we have $j\in\J_L$ and $j_0\leq \kappa_L(j)$, so $j_0 \vee j=\alpha_L(j)\in\J_{L^0}$. Because $j_0 \vee j \leq \x$, the element $j_0\vee j$ is actually in $\J_{(L_\x)^0}$.  This shows that $\alpha_{L_\x}$ does actually map $\{j \in \J_{L_\x}:j_0\leq \kappa_{L_\x}(j)\}$ to $\J_{(L_\x)^0}$. Noting that \[\{ j \in \J_{L_\x} : j_0\leq \kappa_{L_\x}(j)\}= \{j \in \J_L : j_0\leq \kappa_L(j)\}\cap L_\x,\] we see that the injectivity of $\alpha_{L_\x}$ follows immediately from the injectivity of $\alpha_L$. 

To see that $\alpha_{L_\x}$ is surjective, let us choose a join-irreducible element $j'$ of $\J_{(L_\x)^0}$. Then $j'$ is join-irreducible in $L^0$, so it is of the form $\alpha_L(j'')$ for some $j'' \in \J_L$ with $j_0\leq \kappa_L(j'')$. But then $j''\leq j'\leq \x$, so $j''\in\{j \in \J_L : j_0\leq \kappa_L(j)\}\cap L_\x=\{j \in \J_{L_\x} : j_0\leq \kappa_{L_\x}(j)\}$.

We also need to check that if $j\in\mathcal J_{L_\x}$ satisfies $j_0\leq \kappa_{L_\x}(j)$, then $\kappa_{(L_\x)^0}(\alpha_{L_\x}(j))=\kappa_{L_\x}(j)$. We know by \Cref{def:semidistrim} that $\kappa_{L^0}(\alpha_L(j))=\kappa_L(j)$. Hence, \[\kappa_{(L_\x)^0}(\alpha_{L_\x}(j))=\kappa_{L^0}(\alpha_{L_\x}(j))\wedge \x=\kappa_{L^0}(\alpha_{L}(j))\wedge \x=\kappa_L(j)\wedge \x=\kappa_{L_\x}(j).\]

To finish the proof, we need to show that there is a bijection $\beta_{L_\x}\colon\{ m \in \M_{L_\x} : \kappa_{L_\x}^{-1}(m) \leq m_0^\x\}\to\mathcal M_{(L_\x)_0}$ given by $\beta_{L_\x}(m)=m_0^\x\wedge m$ such that $\kappa_{(L_\x)_0}^{-1}(\beta_{L_\x}(m))=\kappa^{-1}_{L_\x}(m)$ for all $m\in\mathcal M_{L_\x}$ with $\kappa_{L_\x}^{-1}(m)\leq m_0^\x$. Choose $m\in\mathcal M_{L_\x}$ with
$\kappa_{L_\x}^{-1}(m)\leq m_0^\x$, and let $j=\kappa_{L_\x}^{-1}(m)$. Then $m = \kappa_L(j) \wedge \x$, so 
\begin{equation}\label{eq:colin1}
m_0^\x\wedge m = (m_0 \wedge \x)\wedge (\kappa_L(j) \wedge \x) = m_0^\x\wedge \kappa_L(j)=\kappa_{(L_\x)_0}(j)=\kappa_{(L_\x)_0}(\kappa_{L_\x}^{-1}(m)),
\end{equation} where the equality $m_0^\x\wedge \kappa_L(j)=\kappa_{(L_\x)_0}(j)$ follows from the description of the unique pairing $\kappa_{(L_\x)_0}$ on the interval $(L_\x)_0=[\hat 0,m_0^\x]$ given by \Cref{lem:paired}. This proves that $m_0^\x\wedge m\in\mathcal M_{(L_\x)_0}$ and that $\kappa_{(L_\x)_0}^{-1}(m_0^\x\wedge m)=\kappa_{L_\x}^{-1}(m)$. This yields the map $\beta_{L_\x}$; we still need to prove that this map is a bijection. 

The map $\beta_{L_\x}$ is injective because we can use \eqref{eq:colin1} to see that $m=\kappa_{L_\x}(\kappa_{(L_v)_0}^{-1}(\beta_{L_\x}(m)))$ for every $m\in\M_{L_\x}$ with $\kappa_{L_\x}^{-1}(m)\leq m_0^\x$. To prove that $\beta_{L_\x}$ is surjective, consider $m'\in\M_{(L_\x)_0}$, and let $m=\kappa_{L_\x}(\kappa_{(L_\x)_0}^{-1}(m'))$. Then $m\in\M_{L_\x}$, and $\kappa_{L_\x}^{-1}(m)\leq m_0^\x$. According to \eqref{eq:colin1}, we have $\beta_{L_\x}(m)=m_0^\x\wedge m=\kappa_{(L_\x)_0}(\kappa_\x^{-1}(m))=m'$, as desired.
\end{proof}

\begin{lemma}
If $L$ is a semidistrim lattice and $\x\in L$, then the Galois graph $G_{L_\x}$ of the interval $L_\x = [\hat{0},\x]$ is the induced subgraph of $G_L$ with vertex set $\mathcal J_{L_\x}$.
\label{lem:galois_restriction}
\end{lemma}

\begin{proof}
Consider $j,j' \in \J_{L_\x}$. We must show that $j\to j'$ in $G_L$ if and only if $j\to j'$ in $G_{L_\x}$. This is equivalent to showing that $j\leq \kappa_L(j')$ if and only if $j\leq \kappa_{L_\x}(j')$. We know by \Cref{lem:paired} that $\kappa_{L_\x}(j')=\kappa_L(j')\wedge\x$. Since $j'\leq \x$, the desired result follows.  
\end{proof}

\begin{theorem}
\label{thm:intervals}
Intervals in semidistrim lattices are semidistrim.  
\end{theorem}

\begin{proof} 
We first argue that if $\x\in L$, then the interval $L_\x=[\hat{0},\x]$ is semidistrim. We know by \Cref{lem:intervalsrowmotable} that $L_\x$ is compatibly dismantlable. Now suppose $y\in L_\x$. We know by \Cref{lem:galois_restriction} that the Galois graph $G_{L_\x}$ of $L_\x$ is an induced subgraph of $G_L$. The sets $\D_{L_\x}(y)$ and $\U_{L_\x}(y)$ are the intersections of $\D_L(y)$ and $\U_L(y)$, respectively, with the vertex set of $G_{L_\x}$. Since $\D_L(y)$ and $\U_L(y)$ are independent sets in $G_L$, the sets $\D_{L_\x}(y)$ and $\U_{L_\x}(y)$ must be independent sets in $G_{L_\x}$. Hence, $L_\x$ is semidistrim.

Now fix an interval $[u,v]$ in $L$. It is immediate from the definition of a semidistrim lattice that a lattice is semidistrim if and only if its dual is semidistrim, so it follows from the preceding paragraph that $[u,\hat{1}]$ is semidistrim. We can now apply the preceding paragraph again to see that the interval $[u,v]$ of the semidistrim lattice $[u,\hat{1}]$ is semidistrim. 
\end{proof}

Suppose $L$ is a compatibly dismantlable lattice with dismantling pair $(j_0,m_0)$. By \Cref{def:semidistrim}, the unique pairings $\kappa_{L_0}$ and $\kappa_{L^0}$ on the intervals $L_0=[\hat 0,m_0]$ and $L^0=[j_0,\hat 1]$ are compatible with the pairing $\kappa_L$ via the maps $\alpha$ and $\beta$. If we additionally assume that $L$ is semidistrim, then the preceding theorem tells us that every interval in $L$ is semidistrim; the next corollary tells us that the pairings on the intervals of $L$ are compatible with $\kappa_L$. 

\begin{corollary}\label{cor:intervals_compatible}
Let $L$ be a semidistrim lattice, and let $[u,v]$ be an interval in $L$. There are bijections $\alpha_{u,v}\colon J_L(v)\cap M_L(u)\to\J_{[u,v]}$ and $\beta_{u,v}\colon\kappa_L(J_L(v)\cap M_L(u))\to\M_{[u,v]}$ given by $\alpha_{u,v}(j)=u\vee j$ and $\beta_{u,v}(m)=v\wedge m$, respectively. Moreover, $\kappa_{[u,v]}(\alpha_{u,v}(j))=\beta_{u,v}(\kappa_L(j))$ for all $j\in J_L(v)\cap M_L(u)$.
\end{corollary}

\begin{proof}
If $u=\hat 0$, then this follows from \Cref{lem:paired}. There is a natural dual version of \Cref{lem:paired} that yields the desired result when $v=\hat 1$. The general result then follows from combining these two cases (for example, by viewing $[u,v]$ as a lower interval in $[u,\hat 1]$). 
\end{proof}

Note that it follows from the preceding corollary that every prime pair in a semidistrim lattice is a dismantling pair. 

To end this section, let us record how the Galois graph and the edge labels of an interval in a semidistrim lattice relate to those of the entire lattice. 

\begin{corollary}\label{cor:galois_label_intervals}
Let $[u,v]$ be an interval in a semidistrim lattice $L$. The bijection $\alpha_{u,v}\colon J_L(v)\cap M_L(u)\to\J_{[u,v]}$ given by $\alpha_{u,v}(j)=u\vee j$ is an isomorphism from an induced subgraph of the Galois graph $G_L$ to the Galois graph $G_{[u,v]}$.  If $u\leq x\lessdot y\leq v$ and $j_{xy}$ is the label of the cover relation $x\lessdot y$ in $L$, then $\alpha_{u,v}(j_{xy})$ is the label of the same cover relation in $[u,v]$.  
\end{corollary}

\begin{proof}
Suppose $j,j'\in J_L(v)\cap M_L(u)$. This means that $j,j'\leq v$ and $\kappa_L(j),\kappa_L(j')\geq u$. We must show that $j\to j'$ in $G_L$ if and only if $\alpha_{u,v}(j)\to\alpha_{u,v}(j')$ in $G_{[u,v]}$. Equivalently, we must show that $j\leq\kappa_L(j')$ if and only if $\alpha_{u,v}(j)\leq \kappa_{[u,v]}(\alpha_{u,v}(j'))$. We have $\alpha_{u,v}(j)=u\vee j$ by definition, and we have $\kappa_{[u,v]}(\alpha_{u,v}(j'))=\beta_{u,v}(\kappa_L(j'))=v\wedge\kappa_L(j')$ by \Cref{cor:intervals_compatible}. The assumptions $j\leq v$ and $\kappa_L(j')\geq u$ imply that $j\leq\kappa_L(j')$ if and only if $u \vee j\leq v\wedge\kappa_L(j')$. 

To prove the last statement, note that $j_{xy}\leq y$ and $\kappa_L(j_{xy})\geq x$ by \Cref{def:overlapping}. Since $y\geq u$, this implies that $\alpha_{u,v}(j_{xy})=u\vee j_{xy}\leq y$. Also, since $x\leq v$, we have (by \Cref{cor:intervals_compatible}) $\kappa_{[u,v]}(\alpha_{u,v}(j_{xy}))=v\wedge\kappa_L(j_{xy})\geq x$. Hence, $\alpha_{u,v}(j_{xy})$ is the label of $x\lessdot y$ in $[u,v]$. 
\end{proof}

\section{Poset Topology}\label{sec:crosscut}
The \defn{crosscut complex} of a lattice $L$ is the abstract simplicial complex whose faces are the sets $A$ of atoms of $L$ such that $\bigvee A\neq \hat 1$. We say $L$ is \defn{crosscut simplicial} if for all $u,v\in L$ with $u\leq v$, the crosscut complex of the interval $[u,v]$ contains all proper subsets of the set of atoms of $[u,v]$ as faces. It is known \cite[Theorem~10.8]{bjornertopological} that the order complex of a lattice is homotopy equivalent to its crosscut complex; it follows that if $L$ is crosscut simplicial, then every interval in $L$ has an order complex that is either contractible or homotopy equivalent to a sphere. McConville \cite{mcconville2017crosscut} proved that semidistributive lattices are crosscut simplicial; in fact, Barnard \cite{barnard19canonical} showed that a lattice is semidistributive if and only if it is join-semidistributive and crosscut simplicial. In his initial work on trim lattices, Thomas \cite{thomas2006analogue} proved that the order complex of a trim lattice must be contractible or homotopy equivalent to a sphere. We generalize these results to semidistrim lattices in the following theorem. 

\begin{theorem}
Semidistrim lattices are crosscut simplicial. 
\end{theorem}

\begin{proof}
Let $L$ be a semidistrim lattice. Because intervals of semidistrim lattices are semidistrim (\Cref{thm:intervals}), it is enough to show that every proper subset of the set of atoms of $L$ is a face of the crosscut complex of $L$. Note that $\U_L(\hat 0)$ is the set of atoms of $L$. Let $A$ be a proper subset of $\U_L(\hat 0)$. By the definition of a semidistrim lattice, the set $\U_L(\hat 0)$ is an independent set in $G_L$. Hence, $A$ is also an independent set in $G_L$. It follows from \Cref{thm:independent} that there exist distinct $x,x'\in L$ such that $\D_L(x)=\U_L(\hat 0)$ and $\D_L(x')=A$. According to \Cref{thm:labels}, we have $x=\bigvee \U_L(\hat 0)$ and $x'=\bigvee A$. Then $x'<x$, so $\bigvee A<\hat 1$.  
\end{proof}

\begin{corollary}
The order complex of a semidistrim lattice is either contractible or homotopy equivalent to a sphere. 
\end{corollary}

\section{Pop-Stack Sorting and Rowmotion}
\label{sec:pop_and_row}

\subsection{Relationships between pop-stack sorting and rowmotion}
Let $L$ be a lattice. Following \cite{defant21meeting}, we define the \defn{pop-stack sorting operator} $\popdown_L\colon L\to L$ and the \defn{dual pop-stack sorting} operator $\popup_L\colon L\to L$ by 
\[
\popdown_L(x)=x\wedge\bigwedge\{y\in L:y\lessdot x\}\quad\text{and}\quad\popup_L(x)=x\vee\bigvee\{y\in L:x\lessdot y\}.
\] In particular, $\popdown_L(\hat 0)=\hat 0$, and $\popup_L(\hat 1)=\hat 1$. 
As mentioned in the introduction, $\pop_L$ coincides with the classical pop-stack sorting map (see \cite{ClaessonPop, Elder}) when $L$ is the right weak order on the symmetric group $S_n$. Pop-stack sorting operators on lattices were introduced in \cite{defant2021stack, defant21meeting} as generalizations of the pop-stack sorting map.

Now suppose $L$ is semidistrim. \Cref{thm:independent} tells us that the maps $\D_L\colon L\to\I(G_L)$ and $\U_L\colon L\to\I(G_L)$ are bijections. We define \defn{rowmotion} to be the bijection $\row_L\colon L\to L$ defined by declaring $\row_L(x)$ to be the unique element of $L$ that satisfies $\U_L(\row(x))=\D_L(x)$.
This definition of rowmotion extends the rowmotion operators on distributive, semidistributive, and trim lattices considered recently by several authors~\cite{thomas2019rowmotion,thomas2019independence,barnard19canonical,striker2018rowmotion}.

Our goal in this section is to show that pop-stack sorting, dual pop-stack sorting, and rowmotion are closely related. While discussing the connections among these operators, we will be led to questions and results that are new even for distributive lattices. 

\begin{theorem}
Let $L$ be a semidistrim lattice. For $x\in L$, we have 
\begin{align*}\row_L(x) = \bigwedge\kappa_L(\D_L(x)) &\text{ and } \row_L^{-1}(x) = \bigvee \U_L(x); \\ 
\popdown_L(x) = x \wedge \bigwedge\kappa_L(\D_L(x)) &\text{ and } \popup_L(x) = x \vee \bigvee\U_L(x).\end{align*}  In particular, \[\pop_L(x) = x \wedge \row_L(x) \text{ and } \popup_L(x) = x \vee \row_L^{-1}(x).\]
\label{thm:pop_and_row}
\end{theorem}

\begin{proof}
The statements about rowmotion follow from~\Cref{thm:labels}. For the statements about pop-stack sorting and dual pop-stack sorting, we compute \[\popdown_L(x)=x\wedge\bigwedge\{y\in L:y\lessdot x\}=x\wedge\bigwedge_{y\lessdot x}(x\wedge \kappa_L(j_{xy}))=x\wedge\bigwedge\kappa_L(\D_L(x))\] and \[\popup_L(x)=x\vee\bigvee\{y\in L:x\lessdot y\}=x\vee\bigvee_{x\lessdot y}(x\vee j_{xy})=x\vee\bigvee\U_L(x). \qedhere\] 
\end{proof}

The identity $\popdown_L(x)=x\wedge\row_L(x)$ in the preceding theorem gives our first glance at the connection between pop-stack sorting and rowmotion. However, there might be several elements $z$ such that $\popdown_L(x)=x\wedge z$, so we would like to know that there is something special about $\row_L(x)$. We will see that $\row_L(x)$ is actually a maximal element of the set of all $z$ satisfying $\popdown_L(x)=x\wedge z$. First, we need the following lemma. Recall from \Cref{thm:intervals} that intervals in semidistrim lattices are semidistrim.

\begin{lemma}\label{prop:row_lower_interval}
Let $L$ be a semidistrim lattice. Consider $v\in L$, and let $L_v=[\hat{0},v]$. For every $x\in L_v$, we have $\row_{L_v}(x)=v\wedge \row_L(x)$. 
\end{lemma}
\begin{proof}
We know by \Cref{lem:same_num} that $\kappa_{L_v}(j)=\kappa_L(j)\wedge v$ for all $j\in\J_{L_v}$. Now consider $x\in L_v$. We have $\D_{L_v}(x)=\D_L(x)$ by \Cref{cor:galois_label_intervals}. Using \Cref{thm:pop_and_row}, we find that \[\row_{L_v}(x)=\bigwedge\kappa_{L_v}(\D_L(x))=\bigwedge_{j\in\D_L(x)}(\kappa_L(j)\wedge v)=v\wedge\bigwedge_{j\in\D_L(x)}\kappa_L(j)=v\wedge\row_L(x). \qedhere\]
\end{proof}

Our next theorem tightens the relationship between pop-stack sorting and rowmotion beyond~\Cref{thm:pop_and_row}; it is new and nontrivial even for semidistributive and trim lattices. 

\begin{theorem}
\label{thm:row_max}
Let $L$ be a semidistrim lattice. If $x\in L$, then $\row_L(x)$ is a maximal element of the set $\{z\in L:\pop_L(x)= x\wedge z\}$, and $\row_L^{-1}(x)$ is a minimal element of the set $\{z\in L:\popup_L(x)= x\vee z\}$.
\end{theorem}

\begin{proof}
We prove that $\row_L(x)\in\max\{z\in L:\pop_L(x)= x\wedge z\}$; the second statement follows from this one by taking duals. Let $(j_0,m_0)$ be a dismantling pair for $L$, and let $L_0=[\hat{0},m_0]$ and $L^0=[j_0,\hat{1}]$. We consider three cases. 

\medskip 

\noindent{\bf Case 1}. Assume $x\in L_0$. Note that $\pop_L(x)=\pop_{L_0}(x)$. It follows by induction on the size of the lattice that $\row_{L_0}(x)\in\max\{z\in L_0:\pop_L(x)=x\wedge z\}$. Suppose by way of contradiction that there exists $w\in L$ satisfying $w>\row_L(x)$ and $\pop_L(x)=x\wedge w$. Let $y\in L$ be such that $\row_L(x)\lessdot y\leq w$. Then \[\pop_L(x)=x\wedge\row_L(x)\leq x\wedge y\leq x\wedge w=\pop_L(x),\] so $\pop_L(x)=x\wedge y$. We have $x\wedge (y\wedge m_0)=(x\wedge m_0)\wedge y=x\wedge y=\pop_L(x)$. Applying \Cref{prop:row_lower_interval} with $v=m_0$, we find that $\row_{L_0}(x)=\row_L(x)\wedge m_0\leq y\wedge m_0$. We have seen that $\row_{L_0}(x)$ is maximal in the set $\{z\in L_0:\pop_L(x)=x\wedge z\}$ and that $y\wedge m_0$ is in this set, so $y\wedge m_0=\row_{L_0}(x)$. Let $j$ be the join-irreducible element of $L$ that labels the cover relation $\row_L(x)\lessdot y$. Then $j\leq y$, and $j\not\leq \row_L(x)$. Now $j\in\U_L(\row_L(x))=\D_L(x)$, so $j\leq x\leq m_0$. This shows that $j\leq y\wedge m_0=\row_{L_0}(x)=\row_L(x)\wedge m_0\leq\row_L(x)$, which is our desired contradiction. 

\medskip 

\noindent {\bf Case 2}. Now assume $x\in L^0$ and $\row_L(x)\in L^0$. 
For each $j\in\D_L(x)=\U_L(\row_L(x))$, we have $\kappa_L(j)\geq\row_L(x)\geq j_0$. In particular, $j_0\not\in\D_L(x)$ because $\kappa_L(j_0)\not\geq j_0$.
\Cref{cor:label_sets_in_intervals} tells us that the map $\alpha$ given by $\alpha(j)=j_0\vee j$ gives a bijection from $\D_L(x)$ to $\D_{L^0}(x)$. Furthermore, \Cref{def:semidistrim} ensures that $\kappa_{L^0}(\alpha(j))=\kappa_L(j)$ for all $j\in\D_L(x)$. Hence, by \Cref{thm:pop_and_row}, \[\row_{L_0}(x)=\bigwedge\kappa_{L^0}(\D_{L^0}(x))=\bigwedge\kappa_{L^0}(\alpha(\D_L(x)))=\bigwedge\kappa_L(\D_L(x))=\row_L(x).\] Also, note that $\pop_L(x)=x\wedge \row_L(x)\in L^0$. This means that all of the elements covered by $x$ in $L$ are in $L^0$, so $\pop_{L^0}(x)=\pop_L(x)$. Consequently, $\{z\in L:\pop_L(x)=x\wedge z\}=\{z\in L^0:\pop_{L^0}(x)=x\wedge z\}$. We know by induction that $\row_{L^0}(x)$ is a maximal element of $\{z\in L^0:\pop_{L^0}(x)=x\wedge z\}$. We have seen that $\row_{L^0}(x)=\row_L(x)$, so $\row_L(x)$ is a maximal element of $\{z\in L:\pop_L(x)=x\wedge z\}$. 

\medskip 

\noindent {\bf Case 3}. Finally, assume $x\in L^0$ and $\row_L(x)\in L_0$. Then $\pop_L(x)=x\wedge\row_L(x)\in L_0$, so there must be an element $w\in L_0$ that is covered by $x$. \Cref{cor:down_labels} tells us that this element $w$ is unique and that the label of the cover relation $w\lessdot x$ is $j_0$. Because $L$ is semidistrim, the set $\D_L(x)$ is an independent set in the Galois graph $G_L$. Therefore, $\D_L(x)\setminus\{j_0\}$ is also an independent set in $G_L$. \Cref{thm:independent} guarantees that there exists a unique $x'\in L$ such that $\D_L(x')=\D_L(x)\setminus\{j_0\}$. Then $x'\neq x$, and \Cref{thm:labels} implies that $x=x'\vee j_0$. Hence, $x'\in L_0$. Appealing to Case 1, we find that $\row_{L_0}(x')\in\max\{z\in L_0:\pop_{L_0}(x')=x'\wedge z\}$. Suppose by way of contradiction that $\row_L(x)\not\in\max\{z\in L:\pop_L(x)=x\wedge z\}$. Then there exists $y\in L$ satisfying $\row_L(x)\lessdot y$ and $\pop_L(x)=x\wedge y$. We will prove that $\row_{L_0}(x')\lessdot y$ and $\pop_L(x')=x'\wedge y$, which will contradict the fact that $\row_{L_0}(x')\in\max\{z\in L_0:\pop_{L_0}(x')=x'\wedge z\}$.

Since $x'\in L_0$, we can apply \Cref{prop:row_lower_interval} with $v=m_0$ to find that $\row_{L_0}(x')=\row_L(x')\wedge m_0$. Appealing to \Cref{thm:pop_and_row} yields \[\row_{L_0}(x')=m_0\wedge \row_L(x')=m_0\wedge\bigwedge\kappa_L(\D_L(x'))=\kappa_L(j_0)\wedge\bigwedge\kappa_L(\D_L(x)\setminus\{j_0\})=\bigwedge\kappa_L(\D_L(x))\] \[=\row_L(x).\] It follows that $\row_{L_0}(x')\lessdot y$. Finally, since $x'<x$, we can use \Cref{thm:pop_and_row} again to find that \[x'\wedge y=x'\wedge x\wedge y=x'\wedge\pop_L(x)=x'\wedge x\wedge\bigwedge\kappa_L(\D_L(x))=x'\wedge\bigwedge\kappa_L(\D_L(x))\] \[=x'\wedge\kappa_L(j_0)\wedge\bigwedge\kappa_L(\D_L(x'))=(x'\wedge m_0)\wedge\bigwedge\kappa_L(\D_L(x'))=x'\wedge\bigwedge\kappa_L(\D_L(x'))=\pop_L(x'). \qedhere\]
\end{proof}

If $L$ is a meet-semidistributive lattice and $a,b\in L$ are such that $a\leq b$, then the set $\{z\in L:a=b\wedge z\}$ has a unique maximal element. Therefore, if $L$ is both meet-semidistributive and semidistrim, then \Cref{thm:row_max} tells us that for every $x \in L$, $\row_L(x)$ is the unique maximal element of $\{z \in L:  \pop_L(x)=x\wedge z\}$. This is interesting because it provides a natural way to extend the definition of rowmotion to arbitrary meet-semidistributive lattices that might not be semidistrim. More precisely, if $L$ is a meet-semidistributive lattice, then we define \defn{rowmotion} to be the operator $\row_L\colon L\to L$ such that for every $x\in L$, the element $\row_L(x)$ is the unique maximal element of $\{z \in L:  \pop_L(x)=x\wedge z\}$. Let us remark that this rowmotion operator is \emph{not} necessarily bijective in general (see \Cref{fig:meet_semi}). In fact, we have the following proposition. 

\begin{proposition}\label{prop:bijective_meet_semi}
A lattice $L$ is semidistributive if and only if it is meet-semidistributive and the rowmotion operator $\row_L\colon L\to L$ is bijective. 
\end{proposition}

\begin{proof}
We know already that rowmotion on a semidistributive lattice is bijective, so we only need to prove one direction of the proposition. Suppose $L$ is a meet-semidistributive lattice such that $\row_L$ is bijective. We want to show that $L$ is join-semidistributive. By \cite[Theorem~2.56]{freese95free}, it suffices to show that $|\J_L(m)|=1$ for every $m\in\mathcal M_L$. Suppose $m\in \M_L$ and $j\in\J_L(m)$. Then $m^*=m\vee j\neq m\vee j_*$. This forces us to have $m\geq j_*$ and $m\not\geq j$, so $m\wedge j=j_*$. Furthermore, $m^*\wedge j=j\neq j_*$. It follows that $m\in\M_L(j)=\max\{z\in L:\pop_L(j)=j\wedge z\}$. Hence, $m=\row_L(j)$. This proves that the only element of $\J_L(m)$ is $\row^{-1}(m)$. 
\end{proof}

\Cref{prop:bijective_meet_semi} implies that a lattice is semidistributive if and only if it is meet-semidistributive and semidistrim. 

\begin{figure}[htbp]
\begin{center}
\raisebox{-.5\height}{\scalebox{1}{\begin{tikzpicture}[scale=1]
\node[shape=circle,fill=black,scale=0.5,label={[xshift=0ex, yshift=-4.5ex]$\hat{0}$}] (0) at (0,0) {};
\node[shape=circle,fill=black,scale=0.5,label={[xshift=-2ex, yshift=-2ex]$d$}] (d) at (-1,1) {};
\node[shape=circle,fill=black,scale=0.5,label={[xshift=2ex, yshift=-2ex]$e$}] (e) at (1,1) {};
\node[shape=circle,fill=black,scale=0.5,label={[xshift=-2ex, yshift=-2ex]$a$}] (a) at (-2,2) {};
\node[shape=circle,fill=black,scale=0.5,label={[xshift=-2ex, yshift=-2ex]$b$}] (b) at (0,2) {};
\node[shape=circle,fill=black,scale=0.5,label={[xshift=2ex, yshift=-2ex]$c$}] (c) at (2,2) {};
\node[shape=circle,fill=black,scale=0.5,label={[xshift=0ex, yshift=1ex]$\hat{1}$}] (1) at (0,3) {};
\draw[very thick] (0) to (d) to (a) to (1) to (b) to (d);
\draw[very thick] (1) to (c) to (e) to (0);
\draw[very thick] (b) to (e);
\end{tikzpicture}}}
\hspace{1in}
\raisebox{-.5\height}{\scalebox{1}{\begin{tikzpicture}[scale=1]
\node (0) at (0,0) {$\hat{0}$};
\node (d) at (2,3) {$d$};
\node (e) at (-2,3) {$e$};
\node (a) at (-1,2) {$a$};
\node (b) at (0,1) {$b$};
\node (c) at (1,2) {$c$};
\node (1) at (0,-1) {$\hat{1}$};
\draw[very thick,->] (e) to (a);
\draw[very thick,->]  (a) to (b);
\draw[very thick,->] (b) to (0);
\draw[very thick,->] (d) to (c);
\draw[very thick,->]  (c) to (b);
\draw[very thick,->, bend left]  (0) to (1);
\draw[very thick,->, bend left]  (1) to (0);
\end{tikzpicture}}}
\end{center}
\caption{{\it Left:} A meet-semidistributive lattice $L$ that is not semidistrim.  {\it Right: } The action of rowmotion, defined by the condition $\max\{z\in L : \pop_L(x)=x\wedge z\}=\{\row_L(x)\}$.}
\label{fig:meet_semi}
\end{figure}
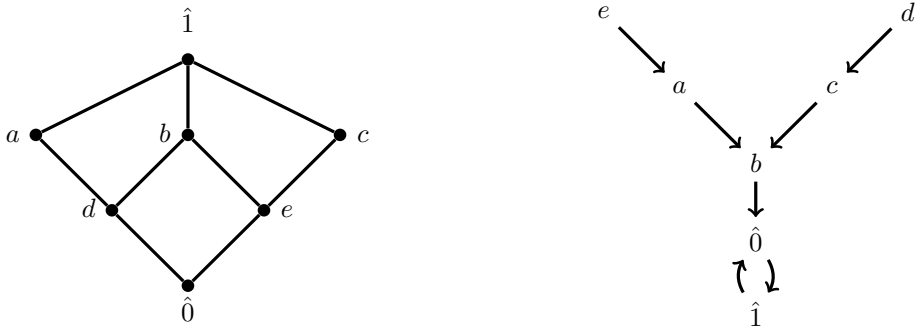

Let us now turn back to semidistrim lattices. The next proposition will be crucial for establishing further connections among pop-stack sorting, dual pop-stack sorting, and rowmotion. Although the proof of this proposition is quite short, let us stress that this is only because we already did most of the heavy lifting in \Cref{subsec:intervals} when we proved that the class of semidistrim lattices is closed under taking intervals. 

\begin{proposition}\label{prop:pop_label_containment}
If $L$ is a semidistrim lattice and $x \in L$, then \[\D_L(x) \subseteq \U_L(\pop_L(x))\quad\text{and}\quad\U_L(x) \subseteq \D_L(\popup_L(x)).\]
\end{proposition}

\begin{proof}
The second containment follows from the first by considering the dual lattice, so we will only prove the first containment. The proof is by induction on $|L|$. Let $u=\pop_L(x)$. Let $\alpha_{u,x}\colon J_L(x)\cap M_L(u)\to\J_{[u,x]}$ be the bijection from \Cref{cor:intervals_compatible}. Since every element covered by $x$ is in $[u,x]$, it follows from \Cref{cor:galois_label_intervals} that $\D_{[u,x]}(x)=\alpha_{u,x}(\D_L(x))$ and that $\U_{[u,x]}(x)\subseteq\alpha_{u,x}(\U_L(u))$. The interval $[u,x]$ is semidistrim by \Cref{thm:intervals}. If $[u,x]$ is a proper interval of $L$, then we can use induction to see that $\D_{[u,x]}(x)\subseteq \U_{[u,x]}(u)$. In this case, we have $\D_L(x)=\alpha_{u,x}^{-1}(\D_{[u,x]}(x))\subseteq\alpha_{u,x}^{-1}(\U_{[u,x]}(u))\subseteq\U_L(u)$, as desired. Now suppose $[u,x]=L$. This means that $x=\hat 1$ and $u=\pop_L(\hat 1)=\hat 0$. Then $\row_L(\hat{1}) = \row_L(\hat{1}) \wedge \hat{1} = \pop_L(\hat{1})$, so $\D_L(\hat{1}) =\U_L(\row_L(\hat 1))= \U_L(\pop_L(\hat{1}))$ by the definition of rowmotion.  
\end{proof}

\subsection{Popping pairs}\label{subsec:popping}
We are now in a position to discuss deeper connections among $\row_L$, $\pop_L$, and $\popup_L$. Let us begin with a fairly innocent question about rowmotion on a semidistrim lattice $L$. How many times does rowmotion on $L$ ``go down''? More precisely, how many elements $x\in L$ have the property that $\row_L(x)\leq x$? This question appears to be new even when $L$ is distributive. We will see that the answer is intimately related to $\pop_L$ and $\popup_L$, as well as independent dominating sets in the Galois graph $G_L$. Before giving more details, let us consider a natural process where we alternately apply $\pop_L$ and $\popup_L$. 

Begin with some element $z\in L$. Let $x_1=\popdown_L(z)$, $y_1=\popup_L(x_1)$, $x_2=\popdown_L(y_1)$, $y_2=\popup_L(x_2)$, and so on. In general, $y_i=\popup_L(x_i)$, and $x_{i+1}=\pop_L(y_i)$. It follows from \Cref{prop:pop_label_containment} that $\D_L(z)\subseteq\U_L(x_1)\subseteq\D_L(y_1)\subseteq\U_L(x_2)\subseteq\D_L(y_2)\subseteq\cdots$. Appealing to \Cref{thm:labels}, we find that we have a chain 
\begin{equation}\label{eq:chain}
    \cdots\leq x_2\leq x_1\leq z\leq y_1\leq y_2\leq\cdots.
\end{equation}
This chain is obviously finite because $L$ is finite. Therefore, there exists a positive integer $k$ such that $\popup_L(x_k)=y_k$ and $\pop_L(y_k)=x_k$. This motivates the following definition.

\begin{definition}\label{popping_pair}
Let $L$ be a semidistrim lattice, and let $x,y\in L$.  We say that $(x,y)$ is a \defn{popping pair} if $\popup_L(x)=y$ and $\popdown_L(y)=x$. 
\end{definition}

The next lemma shows how popping pairs relate to our question about rowmotion going down. 

\begin{lemma}\label{lem:easy_popping_pairs}
Let $L$ be a semidistrim lattice, and let $x,y\in L$. Then $(x,y)$ is a popping pair if and only if $x \leq y$ and $\U_L(x)=\D_L(y)$, and this occurs if and only if $x\leq y$ and $\row_L(y)=x$. 
\end{lemma} 

\begin{proof}
It is immediate from the definition of rowmotion that $\U_L(x)=\D_L(y)$ if and only if $\row_L(y)=x$. If $(x,y)$ is a popping pair, then certainly $x\leq y$, and we have $\U_L(x)\subseteq\D_L(\popup_L(x))=\D_L(y)\subseteq\U_L(\popdown_L(y))=\U_L(x)$ by \Cref{prop:pop_label_containment}. Hence, $\U_L(x)=\D_L(y)$.

Now suppose $x\leq y$ and $\U_L(x)=\D_L(y)$. Then $y=x\vee y=x\vee \bigvee\D_L(y)=x\vee \bigvee\U_L(x)=\popup_L(x)$, where we have used \Cref{thm:labels} and \Cref{thm:pop_and_row}. Similarly, $x=y\wedge x=y\wedge\bigwedge\kappa_L(\U_L(x))=y\wedge\bigwedge\kappa_L(\D_L(y))=\pop_L(y)$. 
\end{proof}

In general, when one is faced with a set $X$ and a noninvertible operator $f\colon X\to X$, it is natural to ask for a description or enumeration of the image of $f$. For example, the structural and enumerative properties of the image of the classical pop-stack sorting map on $S_n$ were studied in~\cite{Asinowski2, ClaessonPop2}. The next proposition connects the images of $\pop_L$ and $\popup_L$ with popping pairs. 

\begin{proposition}\label{prop:popping_pairs_images}
Let $L$ be a semidistrim lattice, and let $z\in L$. Then $z$ is in the image of $\pop_L$ if and only if there exists $y\in L$ such that $(z,y)$ is a popping pair. Similarly, $z$ is in the image of $\popup_L$ if and only if there exists $x\in L$ such that $(x,z)$ is a popping pair.  
\end{proposition}

\begin{proof}
We prove only the first statement since the second statement is dual to it. Certainly if there exists $y\in L$ such that $(z,y)$ is a popping pair, then $z=\pop_L(y)$ is in the image of $\pop_L$. 

To prove the converse, suppose $z$ is in the image of $\pop_L$, say $z=\pop_L(x)$. Let $y=\row_L^{-1}(z)$. Then $y=\bigvee\U_L(z)$ by \Cref{thm:pop_and_row}. We know from \Cref{thm:labels} that $x=\bigvee\D_L(x)$. Then \Cref{prop:pop_label_containment} tells us that $\D_L(x)\subseteq \U_L(z)$, so $z\leq x=\bigvee\D_L(x)\leq y$. This shows that $z\leq y$ and $z=\row_L(y)$, so we deduce from \Cref{lem:easy_popping_pairs} that $(z,y)$ is a popping pair. 
\end{proof}

\begin{remark}\label{rem:stabilize}
\Cref{prop:popping_pairs_images} implies that when we constructed the chain $\cdots\leq x_2\leq x_1\leq z\leq y_1\leq y_2\leq\cdots$ in \eqref{eq:chain}, the process actually stabilized after at most two steps. In other words, \[x_1=x_2=x_3=\cdots\quad\text{and}\quad y_1=y_2=y_3=\cdots. \qedhere\]
\end{remark}

\begin{corollary}\label{cor:equalities}
If $L$ is a semidistrim lattice, then \[|\{x\in L:\row_L(x)\leq x\}|=|\pop_L(L)|=|\popup_L(L)|.\]
\end{corollary}

\begin{proof}
It follows from \Cref{lem:easy_popping_pairs} and \Cref{prop:popping_pairs_images} that each of the three quantities involved in the statement of the corollary is equal to the number of popping pairs of $L$.
\end{proof}

\begin{remark}
The equality $|\pop_L(L)|=|\popup_L(L)|$, which holds when $L$ is semidistrim, is interesting in its own right and is not a simple consequence of the definitions of $\pop_L$ and $\popup_L$. Indeed, this equality \emph{fails} for many lattices. \Cref{fig:oct7} shows a lattice $L$ such that $|\pop_L(L)|=2$ and $|\popup_L(L)|=1$. 
\end{remark}

\begin{figure}[htbp]
\begin{center}
\raisebox{-.5\height}{\scalebox{1}{\begin{tikzpicture}[scale=1]
\node[shape=circle,fill=black,scale=0.5,label={[xshift=0ex, yshift=-4.5ex]$\hat{0}$}] (0) at (0,0) {};
\node[shape=circle,fill=black,scale=0.5,label={[xshift=-2ex, yshift=-2ex]$d$}] (e) at (0,1) {};
\node[shape=circle,fill=black,scale=0.5,label={[xshift=-2ex, yshift=-2ex]$a$}] (a) at (-2,2) {};
\node[shape=circle,fill=black,scale=0.5,label={[xshift=-2ex, yshift=-2ex]$b$}] (b) at (0,2) {};
\node[shape=circle,fill=black,scale=0.5,label={[xshift=2ex, yshift=-2ex]$c$}] (c) at (2,2) {};
\node[shape=circle,fill=black,scale=0.5,label={[xshift=0ex, yshift=1ex]$\hat{1}$}] (1) at (0,3) {};
\draw[very thick] (0) to (a) to (1) to (b);
\draw[very thick] (1) to (c) to (e) to (0);
\draw[very thick] (b) to (e);
\end{tikzpicture}}}
\end{center}
\caption{A non-semidistrim lattice $L$ with $2=|\pop_L(L)|\neq |\popup_L(L)|=1$.}
\label{fig:oct7}
\end{figure}

We now want to relate the quantities appearing in \Cref{cor:equalities} with a notion from graph theory. A set $I$ of vertices in an undirected graph $G$ is called \defn{dominating} if every vertex of $G$ is either in $I$ or is adjacent to a vertex in $I$.  An \defn{independent dominating set} of $G$ is an independent set of $G$ that is dominating. We write $\I^\mathrm{dom}(G)$ for the collection of independent dominating sets of $G$. See~\cite{goddard2013independent} for a survey of independent dominating sets. We are going to show that each of the quantities in \Cref{cor:equalities} is equal to the number of independent dominating sets in the undirected version of the Galois graph $G_L$. First, we need a lemma.

\begin{lemma}\label{lem:below_row} 
Let $L$ be semidistrim lattice, and let $x \in L$. If $j \in \U_L(\pop_L(x)) \setminus \D_L(x)$, then $j \leq \row_L(x)$. 
\end{lemma}

\begin{proof} 
Fix $j\in\U_L(\pop_L(x)) \setminus \D_L(x)$. Because $L$ is semidistrim, $\U_L(\popdown_L(x))$ is an independent set in $G_L$. This implies that $j \leq \kappa_L(j')$ for every $j' \in \U_L(\popdown_L(x))$ with $j'\neq j$. Since $\D_L(x) \subseteq \U_L(\popdown_L(x))$, this implies that $j \leq \kappa_L(j')$ for all $j'\in D_L(x)$. Therefore, \[j \leq \bigwedge\kappa_L(\D_L(x)) = \bigwedge\kappa_L(\U_L(\row_L(x))) = \row_L(x),\] where the last equality follows from \Cref{thm:labels}. 
\end{proof}

Given a directed simple graph $G$, let us write $\widetilde G$ for the undirected simple graph obtained by forgetting the orientations of the edges in $G$.

\begin{theorem}\label{thm:ind_dom} 
Let $L$ be a semidistrim lattice, and let $x \in L$. Then $\row_L(x) \leq x$ if and only if $\D_L(x)$ is an independent dominating set of $\widetilde{G}_L$. Therefore,
\[|\{ x \in L : \row_L(x) \leq x\}|=|\pop_L(L)|=|\popup_L(L)|=|\I^\mathrm{dom}(\widetilde{G}_L)|.\]
\end{theorem}

\begin{proof}
The chain of equalities follows from the first statement and \Cref{cor:equalities}, so we will focus on proving the first statement. Fix $x\in L$. Since $L$ is semidistrim, $\D_L(x)$ is an independent set of $\widetilde{G}_L$. Therefore, we just need to show that $\D_L(x)$ is dominating if and only if $\row_L(x)\leq x$. 

Suppose $j\in\J_L\setminus\D_L(x)$. Saying that $j$ is not adjacent to any element of $\D_L(x)$ is equivalent to saying that $j \leq \kappa_L(j')$ and $j' \leq \kappa_L(j)$ for every $j' \in \D_L(x)$. By \Cref{thm:labels} and \Cref{thm:pop_and_row}, this is equivalent to saying $j \leq \row_L(x)$ and $\kappa_L(j) \geq x$.

Suppose $\row_L(x) \leq x$. If $j\in\J_L\setminus\D_L(x)$ is not adjacent to any element of $\D_L(x)$, then $j \leq \row_L(x) \leq x \leq \kappa_L(j)$, which is impossible. Hence, $\D_L(x)$ is an independent dominating set.

Now suppose $\row_L(x) \not\leq x$. We want to show that $\D_L(x)$ is not a dominating set, which is equivalent to showing that there is a join-irreducible element $j\in\J_L\setminus\D_L(x)$ with $j \leq \row_L(x)$ and $\kappa_L(j) \geq x$. Let $v=\popup_L(\popdown_L(x))$. Since $\pop_L(x)$ is in the image of $\pop_L$, it follows from \Cref{prop:popping_pairs_images} that $(\popdown_L(x),v)$ is a popping pair. \Cref{prop:pop_label_containment} tells us that $\D_L(x)\subseteq\U_L(\pop_L(x))\subseteq\D_L(v)$, so $x\leq v$ by \Cref{thm:labels}. According to \Cref{lem:easy_popping_pairs}, we have $\U_L(\popdown_L(x))=\D_L(v)$ and $\row_L(v)=\pop_L(x)\leq x$. Since we have assumed $\row_L(x)\not\leq x$, this forces $x\neq v$. Consequently, there exists an element $y$ such that $x \leq y \lessdot v$. Let $j=j_{yv}$ be the label of the cover relation $y \lessdot v$. Then $j \in \D_L(v)= \U_L(\popdown_L(x))$. Note that $\kappa_L(j) \geq y \geq x$. This implies that $j \not\leq x$. It follows that $j \not\in \D_L(x)$, so we can use \Cref{lem:below_row} to see that $j \leq \row_L(x)$.
\end{proof}

\Cref{thm:ind_dom} further motivates us to ask about the sizes of the images of pop-stack sorting operators on specific interesting semidistrim lattices. This analysis was carried out in \cite{Asinowski2, ClaessonPop2} when $L$ is the weak order on $S_n$ and in \cite{Sapounakis} when $L$ is the lattice of order ideals of a type $A$ root poset. In \Cref{sec:further}, we list some conjectures about these images for other specific semidistrim lattices. 

\begin{remark}
There is very little known about the behavior of rowmotion on the weak order of $S_n$. Combining \Cref{thm:ind_dom} with the results of \cite{Asinowski2, ClaessonPop2} yields some information by telling us about the number of times rowmotion goes down. 
\end{remark}

\section{Shards}
\label{sec:shards}
Let $\mathcal{H}$ be a real simplicial hyperplane arrangement with fixed base region $B$; its poset of regions is a semidistributive lattice.  In~\cite{reading2011noncrossing}, Reading defined a geometric analogue of join-irreducible elements called \emph{shards}; these are certain connected components of hyperplanes, and there is a natural bijection between shards and join-irreducible elements of the poset of regions.  

The construction proceeds as follows.  A hyperplane $H$ in a subarrangement  of $\mathcal{H}$ is called \defn{basic} if $H$ bounds the region of the subarrangement containing $B$.  Given $H,H' \in \mathcal{H}$, form the rank-two subarrangement $\mathcal{H}(H,H')$ consisting of all hyperplanes containing the intersection $H \cap H'$, and remove from each nonbasic hyperplane of $\mathcal{H}(H,H')$ the points contained in the basic hyperplanes.  The closures of the resulting connected components are \defn{shards}, and for a region $R$, we write $\D(R)$ for the set of shards that bound $R$ and separate it from $B$.    Regions turn out to be in bijection with intersections of shards---a region $R$ corresponds to the intersection $\bigcap \D(R)$.  Because the join-irreducible elements are those regions with $|\D(R)|=1$, join-irreducible elements are in bijection with shards themselves.  We refer to~\cite{reading2016lattice} for further details.

We now indicate an analogy that carries over to semidistrim lattices.  Recall that atoms of a semidistributive lattice are join-prime, so that basic hyperplanes are necessarily join-prime elements.  Having fixed a region $R$, we observe that the subarrangement of $\mathcal{H}$ whose basic hyperplanes correspond to the shards in $\D(R)$ consists of those hyperplanes separating $\popdown(R)$ from $R$, since $\popdown(R)$ corresponds to the region of the subarrangement containing the base region $B$.  This motivates the following definition.

\begin{definition}
For $L$ a finite lattice, define the \defn{face} determined by an element $b \in L$ to be the interval $\face(b)=[\popdown_L(b),b]$.
\label{def:face}
\end{definition}

The following proposition strengthens the analogy between join-prime elements of semidistrim lattices and basic hyperplanes.

\begin{proposition}\label{prop:faces}
Let $L$ be a semidistrim lattice.  Let $j\in\J_L$ be join-prime, and let $b\in L$. If $j$ appears as a label of a cover relation in $\face(b)$, then it appears in $\D_L(b)$ and $\U_L(\pop_L(b))$.
\end{proposition}

\begin{proof}
Let $m$ be the meet-prime element such that $(j,m)$ is a prime pair. Then $m=\kappa_L(j)$ by \Cref{prop:pairing}. Suppose $j$ is the label of a cover relation $y \lessdot z$, where $\pop_L(b)\leq y\lessdot z\leq b$.  Then $y \leq m$, and $z\geq j$. This implies that $b \geq  j$ and that $\pop_L(b) \leq m$. Hence, $j\not\leq \pop_L(b)$. Since $\pop_L(b)$ is the meet of the elements covered by $b$, there must be some element $w$ covered by $b$ such that $j\not\leq w$. Because $(j,m)$ is a prime pair, we have $w\leq m$. The edge $w \lessdot b$ has label $j$, so $j\in\D_L(b)$. Finally, $\D_L(b)\subseteq\U_L(\pop_L(b))$ by~\Cref{prop:pop_label_containment}. 
\end{proof}

It would be interesting if the converse to~\Cref{prop:faces} held.  If $b$ is an element of a semidistrim lattice $L$, then the interval $\face(b)$ is a semidistrim lattice by \Cref{thm:intervals}. It follows that join-prime elements of the lattice $\face(b)$ appear as labels in $\D_L(b)$ and $\U_L(\pop_L(b))$ (we are identifying labels of edges in $\face(b)$ with the labels of the corresponding edges in $L$ via \Cref{cor:galois_label_intervals}). When $L$ is the poset of regions of a simplicial hyperplane arrangement $\mathcal H$, this corresponds to the fact that $\pop_L(R)$ is bounded by the basic hyperplanes in the subarrangement of $\mathcal{H}$ consisting of those hyperplanes containing $\bigcap \D(R)$.  We address further analogies between semidistrim lattices and shards in~\Cref{sec:core_label}.

\section{Further Directions}\label{sec:further}

In this section, we collect several open problems that we hope will stimulate further development of the theory of semidistrim lattices. 

\subsection{Properties of semidistrim lattices}
Birkhoff's representation theorem characterizes the Galois graphs of finite distributive lattices as the (directed) comparability graphs of finite posets. The recent article \cite{reading2019fundamental} characterizes the Galois graphs of finite semidistributive lattices, while a graph-theoretic description of the Galois graphs of trim lattices was given in~\cite[Theorem 4.11]{thomas2019rowmotion}. It would be interesting and useful to have a characterization of the Galois graphs of semidistrim lattices. 

In \Cref{sec:products_intervals}, we proved that products and intervals of semidistrim lattices are semidistrim, generalizing the corresponding statements for semidistributive and trim lattices.  While sublattices of semidistributive lattices are again semidistributive, the same is not true for sublattices of trim lattices.  \Cref{fig:not_sublattice} gives an example showing that sublattices of semidistrim lattices are not necessarily semidistrim.  It would be interesting to have other lattice operations that preserve the family of semidistrim lattices.  In particular, we have the following question that, if answered affirmatively, would generalize the corresponding statement for trim lattices proven in~\cite[Lemma 3.10]{thomas2019rowmotion}. 

\begin{question}
Is every quotient of a semidistrim lattice necessarily semidistrim? 
\end{question}

\begin{figure}[htbp]
\begin{center}
\raisebox{-.5\height}{\scalebox{1}{\begin{tikzpicture}[scale=1]
\node[shape=circle,fill=red, scale=0.5] (1) at (0,0) {};
\node[shape=circle,fill=black, scale=0.5] (2) at (-1,1) {};
\node[shape=circle,fill=red, scale=0.5] (3) at (0,1) {};
\node[shape=circle,fill=red, scale=0.5] (4) at (1,1) {};
\node[shape=circle,fill=red, scale=0.5] (5) at (-1,4) {};
\node[shape=circle,fill=black, scale=0.5] (6) at (.5,2) {};
\node[shape=circle,fill=red, scale=0.5] (7) at (1.5,2) {};
\node[shape=circle,fill=red, scale=0.5] (8) at (0,4) {};
\node[shape=circle,fill=black, scale=0.5] (9) at (.5,3) {};
\node[shape=circle,fill=black, scale=0.5] (10) at (1,4) {};
\node[shape=circle,fill=red, scale=0.5] (11) at (0,5) {};
\draw[very thick] (1) to (2) to (5) to (11) to (8) to (3) to (1);
\draw[very thick] (1) to (4) to (6) to (8);
\draw[very thick] (6) to (9) to (10) to (11);
\draw[very thick] (4) to (7) to (10);
\draw[very thick] (3) to (5);
\draw[very thick] (2) to (9);
\end{tikzpicture}}}
\end{center}
\caption{A sublattice of a semidistrim lattice that is not semidistrim (indicated in red).  This also serves as an example of a trim lattice with a non-trim sublattice (compare with~\cite[Theorem 3]{thomas2006analogue}).}
\label{fig:not_sublattice}
\end{figure}
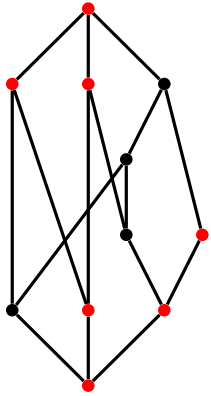

In \Cref{sec:pop_and_row}, we showed how the pop-stack sorting operator can be used to define rowmotion on a meet-semidistributive lattice $L$ that need not be semidistrim; namely, for $x\in L$, we defined $\row_L(x)$ to be the unique maximal element of $\{z\in L:\pop_L(x)=x\wedge z\}$. We saw in \Cref{prop:bijective_meet_semi} that this rowmotion operator is noninvertible whenever $L$ is meet-semidistributive but not semidistributive. Virtually all reasonable questions that one might wish to ask about these operators are open. For example, it would be interesting to describe the periodic points or the image of rowmotion on such a lattice. What can be said about the maximum number of preimages that an element can have under rowmotion? Perhaps there are interesting families of meet-semidistributive lattices where these noninvertible rowmotion operators have desirable properties.  

\subsection{Enumeration}
\label{sec:enumeration}
The image of the classical pop-stack sorting map on $S_n$, which coincides with $\pop_L$ when $L$ is the right weak order on $S_n$, was studied in~\cite{Asinowski2, ClaessonPop2}. \Cref{thm:ind_dom} motivates the investigation of the image of $\pop_L$ for other nice families of semidistrim lattices $L$. It is also natural to refine the enumeration by considering the generating function \[\popnone(L;q) = \sum_{b\in\pop_L(L)}q^{|\U_L(b)|}=\sum_{b\in\popup_L(L)}q^{|\D_L(b)|},\] where the second equality follows from \Cref{lem:easy_popping_pairs} and \Cref{prop:popping_pairs_images}. It follows from \Cref{thm:ind_dom} that $\popnone(L;q)$ enumerates the independent dominating sets in $\widetilde G_L$ according to cardinality.  We write $[q^i]\popnone(L;q)$ for the coefficient of $q^i$ in $\popnone(L;q)$.

When $L=\mathrm{Weak}(A_{n-1})$ is the right weak order on $S_n$, $\popnone(L;q)$ enumerates permutations in the image of the pop-stack sorting map according to the number of ascending runs, which is one of the primary focuses of the article \cite{Asinowski2}.  For example, Proposition 5 of that article proves that \[[q^{n-2}]\popnone(\mathrm{Weak}(A_{n-1});q)=2^n-2n.\]

The article \cite{Sapounakis} considers a certain \emph{filling} operator on Dyck paths, which is equivalent to the dual pop-stack sorting operator on the lattice of order ideals of a type $A$ root poset. Therefore, it follows from \cite[Proposition~4.3]{Sapounakis} that the sizes of the images of dual pop-stack sorting operators (equivalently, the images of pop-stack sorting operators) on lattices of order ideals of type $A$ root posets are given by the OEIS sequence \href{https://oeis.org/A086581}{A086581}: \[\popnone(J(\Phi^+_{A_n});1) = \sum_{k=0}^n \frac{1}{k+1}\binom{2k}{k} \binom{n+k}{3k}.\]

In the same spirit as the above results, we make the following conjectures, with accompanying data given in~\Cref{fig:weak,fig:tamari,fig:bipartite_catalan}.  Write $\mathrm{Weak}(W)$ for the right weak order on a finite Coxeter group $W$.  Write $\mathrm{Tamari}(W)$ for a Cambrian lattice corresponding to a Coxeter element obtained by taking the product of simple reflections in some linear orientation of the Dynkin diagram (defined only when $W$ has a Dynkin diagram that is a path).  Write $\mathrm{Camb}_\mathrm{bi}(W)$ for a Cambrian lattice arising from a bipartite Coxeter element.  Write $J(\Phi^+_W)$ for the distributive lattice of order ideals in the positive root poset of type $W$.

\begin{conjecture} The following equalities hold:
\begin{align*}
[q^{n-1}]\popnone(\mathrm{Weak}(B_n);q)&=3^n - 2n - 1 \\
\popnone(\mathrm{Tamari}(A_n);q)&=\sum_{k=0}^n \frac{1}{k+1}\binom{2k}{k} \binom{n}{2k} q^{n-k}  & \text{(\href{https://oeis.org/A0055151}{A0055151})} \\
\popnone(\mathrm{Tamari}(B_n);q)&=\sum_{k=0}^{\lfloor \frac{n+1}{2} \rfloor} \binom{n-1}{k}\binom{n+1-k}{k} q^{n-k} & \text{see (\href{https://oeis.org/A025566}{A025566})} \\
\popnone(\mathrm{Camb}_\mathrm{bi}(A_n);q)&=
\sum_{k=1}^{\lfloor\frac{(n+3)}{2}\rfloor} \frac{k (-1)^{k-1}}{n-k+3} \sum_{j=0}^{n-k+3}  \binom{j}{n-j+3} \binom{n-k+3}{j} q^{j-2} & \text{see (\href{https://oeis.org/A089372}{A089372})} \\
\popnone(J(\Phi^+_{A_n});q)&=\sum_{k=0}^n \frac{1}{k+1} \sum_{j=0}^{n-k+1} \binom{k+1}{j-1}\binom{k+1}{j} \binom{n-j+1}{n-k-j+1} q^{k+1} & \text{(\href{https://oeis.org/A145904}{A145904})}\\
\popnone(J(\Phi^+_{B_n});q)&=(-q)^n+\sum_{k=0}^n \sum_{j=1}^{k} \binom{k+1}{j}\binom{n-k-1}{j-1}\binom{n-j}{n-k}. & \text{(\href{https://oeis.org/A103881}{A103881})}
\end{align*}

\label{conj:enumerations}
\end{conjecture}

The parts of \Cref{conj:enumerations} concerning the lattices $\text{Tamari}(A_n)$, $\text{Tamari}(B_n)$, $\text{Camb}_{\text{bi}}(A_n)$, and $J(\Phi_{B_n}^+)$ are open even when we specialize $q=1$. 

\renewcommand{\arraystretch}{1.2}

\begin{figure}[htbp]
\[\begin{array}{|c|c|c|c|} \hline
\text{Type} & |W| &  \popnone(\mathrm{Weak}(W);1) & \popnone(\mathrm{Weak}(W);q) \\ \hline
A_2 & 6  & 3  & q^2 + 2q\\
A_3  &24  & 11 & q^3 + 8q^2 + 2q\\
A_4 &120 & 49 & q^4 + 22q^3 + 26q^2 \\
A_5 & 720 & 263 & q^5 + 52q^4 + 168q^3 + 42q^2\\
A_6 & 5040 & 1653 & q^6 + 114q^5 + 804q^4 + 692q^3 + 42q^2\\
A_7 & 40320 & 11877 & q^7 + 240q^6 + 3270q^5 + 6500q^4 + 1866q^3\\
\cdots & \cdots & \cdots & \cdots \\
A_n & n! & \href{https://oeis.org/A307030}{A307030} & \text{unknown} \\ \hline
B_2 & 8 & 5 & q^2 + 4q \\
B_3 & 48  & 27 & q^3 + 20q^2 + 6q \\
B_4 & 384 & 191 & q^4 + 72q^3 + 118q^2 \\
B_5 & 3840 & 1719  & q^5 + 232q^4 + 1136q^3 + 350q^2\\
B_6 & 46080 & 18585 & q^6 + 716q^5 + 8236q^4 + 9032q^3 + 600q^2 \\
B_7 & 645120 & 233857 & q^7 + 2172q^6 + 51666q^5 + 135092q^4 + 44926q^3 \\
\cdots & \cdots & \cdots & \cdots \\
B_n & 2^n n! & \text{unknown} & \text{unknown} \\ \hline
D_3 & 24& 11 & q^3 + 8q^2 + 2q \\
D_4 & 192 & 81 & q^4 + 40q^3 + 30q^2 + 10q \\
D_5 & 1920 & 693 &  q^5 + 152q^4 + 400q^3 + 140q^2 \\
D_6 & 23040 & 7421 & q^6 + 524q^5 + 3692q^4 + 2650q^3 + 554q^2 \\
D_7 & 322560 & 93611 & q^7 + 1724q^6 + 27410q^5 + 46162q^4 + 16606q^3 + 1708q^2 \\
\cdots & \cdots & \cdots & \cdots \\
D_n & 2^{n-1}n! & \text{unknown}& \text{unknown} \\ \hline
E_6 & 51840 & 15971 & q^6 + 1266q^5 + 7510q^4 + 7194q^3 \\ \hline
F_4 & 1152 & 551 & q^4 + 232q^3 + 318q^2 \\ \hline
I_2(m) & 2m & 2m-3 & q^2+(2m-4)q \\ \hline
H_3 & 120 & 75 & q^3 + 56q^2 + 18q \\ \hline
H_4 & 14400 & 7919 & q^4 + 2632q^3 + 5286q^2 \\ \hline
\end{array}\]
\caption{Some data for $\popnone(\mathrm{Weak}(W);q)$, where $\mathrm{Weak}(W)$ is the weak order on the finite Coxeter group $W$ (we did not compute the data for $E_7$ and $E_8$). 
As stated in \Cref{conj:enumerations}, the coefficient of $q^{n-1}$ in $\popnone(\mathrm{Weak}(B_n);q)$ appears to be $ 3^n - 2n - 1$.}
\label{fig:weak}
\end{figure}

\begin{figure}[htbp]
\[\begin{array}{|c|c|c|c|} \hline
\text{Type} & \mathrm{Cat}(W) & \popnone(\mathrm{Tamari}(W);1) & \popnone(\mathrm{Tamari}(W);q) \\ \hline
A_1 & 2 & 1 & q   \\
A_2 & 5 & 2 & q^2 + q  \\
A_3 & 14 & 4 & q^3 + 3q^2 \\ 
A_4 & 42 & 9 & q^4 + 6q^3 + 2q^2 \\
A_5 & 132 & 21 & q^5 + 10q^4 + 10q^3 \\
A_6 & 429 & 51 & q^6 + 15q^5 + 30q^4 + 5q^3 \\
\cdots & \cdots & \cdots & \cdots \\
A_n & \href{https://oeis.org/A000108}{A000108} &  \href{https://oeis.org/A001006}{A001006} \text{ (conj.)} & \href{https://oeis.org/A0055151}{A0055151} \text{ (conj.)}\\ \hline
B_2 & 6 & 3 & q^2 + 2q  \\
B_3 & 20 & 8 & q^3 + 6q^2 + 2q \\ 
B_4 & 70 & 22 & q^4 + 12q^3 + 9q^2\\
B_5 & 252 & 61 & q^5 + 20q^4 + 36q^3 + 4q^2 \\
B_6 & 924 & 171 & q^6 + 30q^5 + 100q^4 + 40q^3 \\
\cdots & \cdots & \cdots & \cdots \\
B_n & \href{https://oeis.org/A000984}{A000984} &  \href{https://oeis.org/A025566}{A025566} \text{ (conj.)} & \sum_{k=0}^{\lfloor \frac{n+1}{2} \rfloor} \binom{n-1}{k}\binom{n+1-k}{k} q^{n-k} \text{ (conj.)} \\\hline
\end{array}\]
\caption{Some data supporting \Cref{conj:enumerations} for $\popnone(\mathrm{Tamari}(W);q)$ for Cambrian lattices coming from linearly-oriented Coxeter elements of types $A$ and $B$.}
\label{fig:tamari}
\end{figure}

\begin{sidewaysfigure}[htbp]
\[\begin{array}{|c|c|c|c|c|c|} \hline
\text{Type} & \mathrm{Cat}(W) & \popnone(\mathrm{Camb}_\mathrm{bi}(W);1) & \popnone(\mathrm{Camb}_\mathrm{bi}(W);q) & \popnone(J(\Phi^+_W);1) & \popnone(J(\Phi^+_W);q) \\ \hline
A_1 & 2 & 1 & q & 1 & q \\
A_2 & 5 & 2 & q^2 + q & 2 & q^2 + q\\
A_3 & 14 & 5 & q^3 + 3q^2+q & 5 & q^3 + 3q^2 + q \\ 
A_4 & 42 & 12 & q^4 + 6q^3 + 5q^2 & 13& q^4 + 6q^3 + 5q^2 + q\\
A_5 & 132 & 29 & q^5 + 10q^4 + 16q^3 + 2q^2 & 35 & q^5 + 10q^4 + 16q^3 + 7q^2 + q\\
A_6 & 429 & 72 & q^6 + 15q^5 + 40q^4 + 16q^3 & 97 & q^6 + 15q^5 + 40q^4 + 31q^3 + 9q^2 + q\\
\cdots & \cdots & \cdots & \cdots & \cdots & \cdots \\
A_n & \href{https://oeis.org/A000108}{A000108} & \href{https://oeis.org/A089372}{A089372} \text{ (conj.)} &  \text{ see \Cref{conj:enumerations}} & \substack{\href{https://oeis.org/A082582}{A082582} \\ \href{https://oeis.org/A086581}{A086581} \\ \href{https://oeis.org/A025242}{A025242}} & \href{https://oeis.org/A145904}{A145904}  \text{ (conj.)}\\ \hline
B_2 & 6 & 3 & q^2 + 2q & 3 & q^2 + 2q \\
B_3 & 20 & 9 & q^3 + 6q^2 + 2q & 9 & q^3 + 6q^2 + 2q\\ 
B_4 & 70 & 25 & q^4 + 12q^3 + 12q^2& 27 & q^4 + 12q^3 + 12q^2 + 2q\\
B_5 & 252 & 69 & q^5 + 20q^4 + 42q^3 + 6q^2 & 83 & q^5 + 20q^4 + 42q^3 + 18q^2 + 2q\\
B_6 & 924 & 193 & q^6 + 30q^5 + 110q^4 + 52q^3 & 259 & q^6 + 30q^5 + 110q^4 + 92q^3 + 24q^2 + 2q\\
\cdots & \cdots & \cdots & \cdots & \cdots & \cdots \\
B_n & \href{https://oeis.org/A000984}{A000984} &  \text{unknown} & \text{unknown} & \href{https://oeis.org/A171155}{A171155}  \text{ (conj.)} &  \href{https://oeis.org/A103881}{A103881}  \text{ (conj.)}\\\hline
D_3 & 14 & 5 & q^3 + 3q^2 + q & 5 & q^3 + 3q^2 + q, \\ 
D_4 & 50 & 17 & q^4 + 8q^3 + 6q^2 + 2q& 17 & q^4 + 8q^3 + 6q^2 + 2q\\
D_5 & 182 & 47 & q^5 + 15q^4 + 23q^3 + 8q^2 & 53 & q^5 + 15q^4 + 24q^3 + 11q^2 + 2q\\
D_6 & 672 & 131 & q^6 + 24q^5 + 67q^4 + 33q^3 + 6q^2 & 167 & q^6 + 24q^5 + 70q^4 + 54q^3 + 16q^2 + 2q\\
\cdots & \cdots & \cdots & \cdots & \cdots & \cdots \\
D_n & \href{https://oeis.org/A051924}{A051924} & \text{unknown} & \text{unknown} & \text{unknown} &  \text{unknown}\\\hline
E_6 & 833 & 164 & q^6 + 30q^5 + 82q^4 + 51q^3 & 210 & q^6 + 30q^5 + 89q^4 + 69q^3 + 18q^2 + 3q  \\ \hline
E_7 & 4160 & 624 & q^7 + 56q^6 + 260q^5 + 261q^4 + 46q^3 & 912 & q^7 + 56q^6 + 285q^5 + 367q^4 + 166q^3 + 33q^2 + 4q \\ \hline
E_8 & 25080 & 2853 & q^8 + 112q^7 + 840q^6 + 1344q^5 + 556q^4 & 4787 & q^8 + 112q^7 + 926q^6 + 1880q^5 + 1367q^4 + 428q^3 + 67q^2 + 6q \\ \hline
F_4 & 105 & 40 & q^4 + 20q^3 + 19q^2 & 44 & q^4 + 20q^3 + 19q^2 + 4q \\ \hline
I_2(m) & m+2 & m-1 & q^2 + (m-2)q & m-1 & q^2+(m-2)q \\ \hline
H_3 & 32 & 17 & q^3+12q^2+4q & 17 & q^3+12q^2+4q  \\ \hline
H_4 & 280 & 125 & q^4+56q^3+68q^2 & & \\ \hline
\end{array}\]
\caption{Some data supporting \Cref{conj:enumerations} for $\popnone(\mathrm{Camb}_{\mathrm{bi}}(W);q)$ and $\popnone(J(\Phi^+_W);q)$, where $\mathrm{Camb}_{\mathrm{bi}}(W)$ is a Cambrian lattice for a bipartite Coxeter element and $J(\Phi^+_W)$ is the distributive lattice of order ideals of the positive root poset of type $W$ (nonnesting partitions).  For the noncrystallographic types $I_2(m)$ and $H_3$, Armstrong's root posets are used for the nonnesting partitions~\cite[Figure 5.15]{armstrong2009generalized}.}
\label{fig:bipartite_catalan}
\end{sidewaysfigure}

\subsection{Uniquely completely paired lattices}
Let $L$ be a lattice. Note that if $j\in \J_L$ and $m\in \M_L$, then we have $\popdown_L(j)=j_*$ and $\popup_L(m)=m^*$, so the sets $\M_L(j)$ and $\J_L(m)$ defined in~\Cref{sec:uniquely_paired_lattices} satisfy $\M_L(j)=\max\{z\in L:\popdown_L(j)=j\wedge z\}$ and $\J_L(m)=\min\{z \in L: \popup_L(j)=j\vee z\}$.   It is reasonable to extend the notion of \emph{uniquely paired} in~\Cref{def:uniquely_paired} to every element of $L$ as follows.

\begin{definition} A lattice $L$ is \defn{completely uniquely paired} if there is a unique bijection $\lambda: L \to L$ so that $\lambda(x) \in \max\{z\in L:\popdown_L(x)=x\wedge z\}$ and $\lambda^{-1}(x) \in \min\{z \in L: \popup_L(x)=x\vee z\}$ for every $x\in L$.
\end{definition}

By \Cref{thm:row_max}, if $L$ is both semidistrim and completely uniquely paired, then the bijection $\lambda$ must be rowmotion---and so it makes sense to use the name \defn{rowmotion} for the bijection $\lambda$ in a completely uniquely paired lattice.  While completely uniquely paired lattices have some desired properties, they generally lack others common to semidistributive and trim lattices---for example, not all completely uniquely paired lattices have join-prime elements, which are a major tool for dealing with semidistrim lattices.  \Cref{fig:unique_complete} illustrates a completely uniquely paired lattice that does not have a join-prime or meet-prime element. While not every completely uniquely paired lattice is semidistrim, we expect the converse should hold. 

\begin{conjecture}
Every semidistrim lattice is completely uniquely paired. 
\end{conjecture}

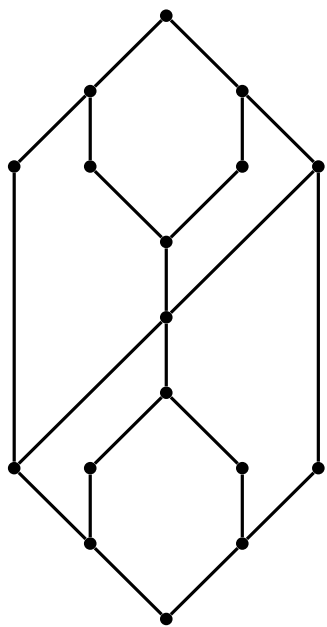
\begin{figure}[htbp]
\begin{center}
\raisebox{-.5\height}{\scalebox{1}{\begin{tikzpicture}[scale=1]
\node[shape=circle,fill=black, scale=0.5] (1) at (0,0) {};
\node[shape=circle,fill=black, scale=0.5] (2) at (-1,1) {};
\node[shape=circle,fill=black, scale=0.5] (3) at (1,1) {};
\node[shape=circle,fill=black, scale=0.5] (4) at (-2,2) {};
\node[shape=circle,fill=black, scale=0.5] (5) at (-1,2) {};
\node[shape=circle,fill=black, scale=0.5] (6) at (1,2) {};
\node[shape=circle,fill=black, scale=0.5] (7) at (2,2) {};
\node[shape=circle,fill=black, scale=0.5] (8) at (0,3) {};
\node[shape=circle,fill=black, scale=0.5] (9) at (0,4) {};
\node[shape=circle,fill=black, scale=0.5] (11) at (-2,6) {};
\node[shape=circle,fill=black, scale=0.5] (12) at (-1,6) {};
\node[shape=circle,fill=black, scale=0.5] (13) at (0,5) {};
\node[shape=circle,fill=black, scale=0.5] (14) at (1,6) {};
\node[shape=circle,fill=black, scale=0.5] (15) at (2,6) {};
\node[shape=circle,fill=black, scale=0.5] (16) at (-1,7) {};
\node[shape=circle,fill=black, scale=0.5] (17) at (1,7) {};
\node[shape=circle,fill=black, scale=0.5] (18) at (0,8) {};
\draw[very thick] (1) to (2) to (4) to (9) to (15) to (7) to (3) to (1);
\draw[very thick] (2) to (5) to (8) to (9);
\draw[very thick] (3) to (6) to (8);
\draw[very thick] (4) to (11) to (16) to (18) to (17) to (15);
\draw[very thick] (16) to (12) to (13) to (14) to (17);
\draw[very thick] (13) to (9);
\end{tikzpicture}}}
\end{center}
\caption{A uniquely completely paired lattice with no join-prime or meet-prime elements.  See also~\Cref{fig:noatoms}.}
\label{fig:unique_complete}
\end{figure}

\subsection{Core label orders}
\label{sec:core_label}
In this section, we closely follow the ideas of~\cite{MuhleCore} in the setting of our semidistrim lattices.  For $L$ a semidistrim lattice and $b\in L$, write \[\popshard(b)=J_L(b)\cap M_L(\pop_L(b)),\] so that $\popshard(b)$ is the set of join-irreducibles that label the cover relations in the interval $\face(b)$ (by \Cref{cor:intervals_compatible} and \Cref{cor:galois_label_intervals}).  It is reasonable to think of the set $\popshard(b)$ as an extension of the edge labeling to faces, since when $j \in \J_L$ with the cover $j_* \lessdot j$ labeled by $j$ itself, we recover $\shard(j)=\{j\}$.  Similarly, we define \[\rowshard(b)=J_L(b)\cap M_L(\row_L(b)).\] By~\Cref{prop:pop_label_containment}, we have that $\rowshard(b) \subseteq \popshard(b)$.  

\begin{question}
For which semidistrim lattices $L$ do we have $\rowshard(b) = \popshard(b)$ for all $b \in L$? Does this hold when $L$ is the lattice of regions of a real simplicial hyperplane arrangement?  
\end{question}

Define the \defn{pop-core label order} to be the partial order $\preceq_{\mathsf{Pop}}$ on $L$ given by $x \preceq_{\mathsf{Pop}} y$ if and only if $\popshard(x) \subseteq \popshard(y)$. This definition extends M\"uhle's definition of the \emph{core label order} \cite{MuhleCore}, which was formulated as a generalization of Reading's \emph{shard intersection order} \cite{reading2011noncrossing}. Define the \defn{row-core label order} by $x \preceq_\row y$ if and only if $\rowshard(x) \subseteq \rowshard(y)$.  In general, we do not have that $\preceq_{\mathsf{Pop}}=\preceq_\row$---\cite[Figure 7]{MuhleCore} gives an example for which $\preceq_{\mathsf{Pop}}$ is not a lattice but $\preceq_\row$ is a lattice.

It would be interesting to see what properties one can say about the row-core label orders of semidistrim lattices. In particular, we have the following question. 

\begin{question}\label{quest:core_meet_semilattice}
When is the pop-core or row-core label order on a semidistrim lattice a meet-semilattice?
\end{question}

In order to approach the preceding question, it could be useful to understand graph-theoretically what the sets $\popshard(b)$ and/or $\rowshard(b)$ are when we view them as subsets of the vertex set of the Galois graph. For example, when $L$ is distributive, the pop-core label order and the row-core label order are equal to each other, and each is a meet-semilattice because \[\{\popshard(b):b\in L\}=\{\rowshard(b):b\in L\}=\I(G_L).\]

\section*{Acknowledgements}
N.W. was partially supported by a Simons Foundation Collaboration Grant. C.D. was supported by a Fannie and John Hertz Foundation Fellowship and an NSF Graduate Research Fellowship (grant number DGE-1656466).  We thank Henri M\"uhle for useful correspondence.  This work benefited from computations in \texttt{Sage}~\cite{sagemath} and the combinatorics features developed by the \texttt{Sage-Combinat} community~\cite{Sage-Combinat}.

\bibliographystyle{amsalpha}
\bibliography{literature}

\end{document}